\documentclass[platex,11pt]{article}
\usepackage{amssymb}
\usepackage{fancybox}
\usepackage{mathrsfs}
\usepackage{amsrefs}
\usepackage{mathtools}
\usepackage{yhmath}

\usepackage{amscd}
\usepackage{amsmath}
\usepackage{amsthm}

\usepackage{comment}
\usepackage{xypic}
\usepackage[margin=3cm]{geometry}
\usepackage{eucal}
\usepackage{bm}

\usepackage{color}
\usepackage{accents}

\theoremstyle{plain}
\newtheorem{thm}{Theorem}[section]
\newtheorem{prop}[thm]{Proposition}
\newtheorem{lem}[thm]{Lemma}
\newtheorem{cor}[thm]{Corollary}

\theoremstyle{definition}
\newtheorem{defi}[thm]{Definition}

\newtheorem{rem}[thm]{Remark}

\newcommand{\ZZ}{\mathbb{Z}}
\newcommand{\calA}{\mathcal{A}}

\newcommand{\QQ}{\mathbb{Q}}
\newcommand{\RR}{\mathbb{R}}
\newcommand{\CC}{\mathbb{C}}
\newcommand{\PP}{\mathbb{P}}
\newcommand{\KK}{\mathbb{K}}
\newcommand{\CS}{\mathbb{C}^{*}}

\newcommand{\GR}{\mathrm{gr}}

\newcommand{\wt}[1]{\widetilde{#1}}
\newcommand{\ov}[1]{\overline{#1}}

\newcommand{\DR}{\mathrm{R}}

\newcommand{\simto}{\overset{\sim}{\to}}

\newcommand{\DBkai}{\mathrm{{D}^{b}}}

\newcommand{\calS}{\mathcal{S}}

\newcommand{\calM}{\mathcal{M}}
\newcommand{\calO}{\mathcal{O}}

\newcommand{\calE}{\mathscr{E}}
\newcommand{\calV}{\mathcal{V}}
\newcommand{\calL}{\mathcal{L}}
\newcommand{\calN}{\mathcal{N}}

\newcommand{\calP}{\mathcal{P}}

\newcommand{\an}{\mathrm{an}}

\newcommand{\dt}{\partial_{t}}

\newcommand{\dpush}{{{}_{\scalebox{0.5}{$\mathrm{D}$}}}}

\newcommand{\tunagi}{\quad\mbox{and}}

\newcommand{\semi}{\mathrm{se}}
\newcommand{\nilp}{\mathrm{ni}}
\newcommand{\kaa}{\mathrm{Ker}}
\newcommand{\cokaa}{\mathrm{Coker}}
\newcommand{\skaa}{\mathrm{ker}}
\newcommand{\scokaa}{\mathrm{coker}}

\newcommand{\FOU}[1]{{}^F\hspace{-1mm}{#1}}

\newcommand{\vardbtilde}[1]{\widetilde{\raisebox{0pt}[0.85\height]{$\widetilde{#1}$}}}

\makeatletter
\newcommand{\subjclass}[2][2010]{%
  \let\@oldtitle\@title%
  \gdef\@title{\@oldtitle\footnotetext{#1 \emph{Mathematics subject classification.} #2}}%
}
\newcommand{\keywords}[1]{%
  \let\@@oldtitle\@title%
  \gdef\@title{\@@oldtitle\footnotetext{\emph{Key words and phrases.} #1.}}%
}
\makeatother

\newcommand{\Addresses}{{
  \bigskip
  \footnotesize

  Takahiro Saito, \textsc{Research Institute for Mathematical Sciences, Kyoto University, Kyoto 606-8502, Japan}\par\nopagebreak
  \textit{E-mail address} : \texttt{takahiro@kurims.kyoto-u.ac.jp}

}}

\title{A description of monodromic mixed Hodge modules}
\date{}
\author{Takahiro Saito}

\subjclass{14F10, 32S35, 32S40}
\keywords{Mixed Hodge modules, D-modules, Perverse sheaves, Irregular Hodge filtrations}

\begin{document}
\maketitle

\begin{abstract}
For a smooth algebraic variety $X$, a monodromic $D$-module on $X\times \CC$ is decomposed into a direct sum of some $D$-modules on $X$.
We show that the Hodge filtration of a mixed Hodge module on $X\times \CC$ whose underlying $D$-module is monodromic is also decomposed.
Moreover, we show that there is an equivalence of categories between the category of monodromic mixed Hodge modules on $X\times \CC$ and the category of ``gluing data''.
As an application, we endow the Fourier-Laplace transformation of the underlying $D$-module of a monodromic mixed Hodge module with a mixed Hodge module structure.
\end{abstract}
\tableofcontents

\section{{I}ntroduction}
Let $X$ be a smooth algebraic variety of finite type over $\CC$
and $\CC_{t}$ the affine line with the coordinate $t$.
We consider a graded polarizable mixed Hodge module $\calM=(M,F_{\bullet}M,K,W_{\bullet}K)$ on $X\times \CC_{t}$,
where $M$ is the underlying $D$-module,
$F_{\bullet}M$ is the Hodge filtration, $K$ is the underlying $\QQ$-perverse sheaf 
and $W_{\bullet}K$ is the weight filtration.
Here, in this paper, ``$D$-module" means ``left $D$-module".
Since the underlying $D$-module of a mixed Hodge module is always regular holonomic, so is $M$.
Brylinski~\cite{BryFou} defined the Fourier-Laplace transformation $\FOU{M}$ of $M$, whose underlying constructible sheaf is the Fourier-Sato transformation of $K$.
$\FOU{M}$ is a $D$-module on $X\times \CC_{\tau}$, where $\CC_{\tau}$ is the dual vector space of $\CC$.
While $\FOU{M}$ is holonomic, it may not be regular in general.
If $\FOU{M}$ is not regular, it can not be the underlying $D$-module of a mixed Hodge module.
Brylinski showed that if $M$ is monodromic (see Definition~\ref{mondef}), $\FOU{M}$ is regular (see loc. cit.).
Therefore, 
it is natural to ask whether there is a natural mixed Hodge module structure on the Fourier-Laplace transformation $\FOU{M}$ of a monodromic $D$-module $M$.
Reichelt~\cite{ReiLau} gave a definition of a mixed Hodge module structures on the homogeneous $A$-hypergeometric $D$-modules (GKZ-systems), which are expressed as the Fourier-Laplace transforms of certain $D$-modules.
Moreover, Reichelt-Walther~\cite{RWWGKZ} defined a mixed Hodge module structures of the Fourier-Laplace transforms of mixed Hodge modules whose underlying $D$-modules are monodromic.
To treat the Hodge filtration on $\FOU{M}$ more clearly,
we take a different approach from their one.

In order to answer the question above, we should observe the structures of mixed Hodge modules whose underlying $D$-modules are monodromic (we call them monodromic mixed Hodge modules) in detail.
As explained in Lemma~\ref{peq},
we will identify quasi-coherent $D_{X\times \CC}$-modules (resp. $O_{X\times \CC}$-modules) with quasi-coherent $D_{X}[t]\langle\partial_{t}\rangle(:=D_{X}\otimes_{\CC}\CC[t]\langle\partial_{t}\rangle)$-modules (resp. $O_{X}[t](:=O_{X}\otimes_{\CC} \CC[t])$-modules) by using the pushforward and the pullback by the projection $p\colon X\times \CC\to X$, where $\CC[t]\langle\partial_{t}\rangle$ is the ring of differential operators on $\CC_{t}$.
Assume that $M$ is monodromic.
For $\beta\in \QQ$, we define a $D_{X}$-submodule $M^{\beta}$ of $M$ (we regard $M$ as a $D_{X}[t]\langle\partial_{t}\rangle$-module) as
\[M^{\beta}:=\bigcup_{l\geq 0}\mathrm{Ker}((t\partial_{t}-\beta)^{l}\colon M\to M).\]
Since $M$ is monodromic,
we have a decomposition
\begin{align}\label{meron}
M=\bigoplus_{\beta\in \QQ}M^{\beta},
\end{align}
where we endow the right hand side with a $D_{X}[t]\langle\partial_{t}\rangle$-structure by using the morphisms $t\colon M^{\beta}\to M^{\beta+1}$ and $\partial_{t}\colon M^{\beta}\to M^{\beta-1}$.
Then, our first result is the following:

\begin{thm}[see Theorem~\ref{main1}]\label{hakujitu}
For $p\in \ZZ$, $F_{p}M\subset M$ is decomposed as
\[F_{p}M=\bigoplus_{\beta\in \QQ}F_{p}M^{\beta},\]
where we set $F_{p}M^{\beta}:=F_{p}M\cap M^{\beta}$
and the $O_{X}[t]$-module structure of the right hand side is defined by the morphism $t\colon F_{p}M^{\beta}\to F_{p}M^{\beta+1}$.
\end{thm}

In other words, the Hodge filtration $F_{\bullet}M$ is decomposed with respect to the decomposition (\ref{meron}).
Let us explain what this theorem means in terms of the nearby and vanishing cycles.
Note that we have (see Proposition~\ref{isomprop})
\[
\left\{\begin{array}{l}
\mbox{$t\colon M^{\beta}\to M^{\beta+1}$ is an isomorphism if $\beta\neq -1$.} \\
\mbox{$\partial_{t} \colon M^{\beta}\to M^{\beta-1}$ is an isomorphism if $\beta\neq 0$.}
\end{array}\right.
\]
Therefore, we can reconstruct the $D$-module $M$ from the data:
\[
\left\{
\begin{array}{ll}
\mbox{(i)}&\mbox{The $D$-modules $M^{\alpha}$ on $X$ for $\alpha\in [-1,0]\cap \QQ$.}\\
\mbox{(ii)}&\mbox{The nilpotent endomorphisms $t\partial_{t}-\alpha\colon M^{\alpha}\to M^{\alpha}$ for $\alpha\in (-1,0)\cap \QQ$.}
\\
\mbox{(iii)}&\mbox{The morphisms $\partial_{t}\colon M^{0}\to M^{-1}$ and $t\colon M^{-1}\to M^{0}$}\\
&\mbox{such that $t\partial_{t}$ and $\partial_{t}t$ are nilpotent.}
\end{array}
\right.
\]
On the other hand,
by the definitions of the nearby cycle (resp. the $e(\alpha)(:=\exp(-2\pi\sqrt{-1}\alpha))$-nearby cycle, the unipotent vanishing cycle) $\psi_{t}M$ (resp. $\psi_{t,e(\alpha)}M$, $\phi_{t,1}M$) of the $D$-module $M$ along $t=0$ (note that they are $D$-module on $X$),
we have (see Corollary~\ref{nevancor})
\begin{align}\label{orange}
\left\{
\begin{array}{l}
\psi_{t}M=\bigoplus_{\alpha\in (-1,0]\cap \QQ}M^{\alpha},\\
\psi_{t,e(\alpha)}M=M^{\alpha},\tunagi\\
\phi_{t,1}M=M^{-1}.
\end{array}
\right.
\end{align}
Moreover, through the identifications,
$\frac{-1}{2\pi\sqrt{-1}}$ times
the logarithm of the unipotent part of the monodromy automorphism of the nearby and vanishing cycles are corresponding to $t\partial_{t}-\alpha$ on $M^{\alpha}$ for $\alpha\in [-1,0]$,
the can morphism $\mathrm{can}\colon \psi_{t,1}M\to \phi_{t,1}M$ (resp. the var morphism $\mathrm{var}\colon \phi_{t,1}M\to \psi_{t,1}M$) corresponds to $-\partial_{t}\colon M^{0}\to M^{-1}$ (resp. $t\colon M^{-1}\to M^{0}$) (see Corollary~\ref{nevancor}).
Therefore,
we can say that the monodromic $D$-module is determined by
the nearby and vanishing cycles of it along $t=0$ with some morphisms between them.
Moreover, under the identifications (\ref{orange}) the Hodge filtrations of $\psi_{t}M$, $\psi_{t,e(\alpha)}M$ and $\phi_{t,1}M$ are
\begin{align*}
\left\{
\begin{array}{l}
F_{\bullet}\psi_{t}M=\bigoplus_{\alpha\in (-1,0]}F_{\bullet}M^{\alpha},\\
F_{\bullet}\psi_{t,e(\alpha)}M=F_{\bullet}M^{\alpha},\tunagi\\
F_{\bullet}\phi_{t,1}M=F_{\bullet+1}M^{-1},
\end{array}
\right.
\end{align*}
where $F_{\bullet}M^{\alpha}$ is $F_{\bullet}M\cap M^{\alpha}$ (see Remark~\ref{nao}).
By the strict specializability (see Definition~\ref{stsp}),
we have
\[
\left\{\begin{array}{l}
\mbox{$t\colon F_{p}M^{\beta}\to F_{p}M^{\beta+1}$ is an isomorphism if $\beta> -1$.} \\
\mbox{$\partial_{t} \colon F_{p}M^{\beta}\to F_{p+1}M^{\beta-1}$ is an isomorphism if $\beta< 0$.}
\end{array}\right.
\] 
Therefore, Theorem~\ref{hakujitu} means that the Hodge filtration of the monodromic mixed Hodge module is determined from the ones of the nearby and vanishing cycles of it (up to shift).
The natural question arises whether it is possible to reconstruct a monodromic mixed Hodge module on $X\times \CC$ from its nearby and vanishing cycles.
Our second main result gives an affirmative answer to this question.

We consider a tuple $((\calM_{(-1,0]},T_{s},N),\calM_{-1},c,v)$,
where $\calM_{(-1,0]}$ and $\calM_{-1}$ are mixed Hodge modules on $X$ and
$T_{s}\colon \calM_{(-1,0]}\simto \calM_{(-1,0]}$,
$N \colon \calM_{(-1,0]}\to \calM_{(-1,0]}(-1)$,
$c\colon \calM_{0}(:=\mathrm{Ker}(T_{s}-1)\subset \calM_{(-1,0]})\to \calM_{-1}$
and 
$v\colon \calM_{-1}\to \calM_{0}(-1)$ are morphisms between mixed Hodge modules with the following properties:
\begin{enumerate}
\item[($\star$-i)] $T_{s}$ commutes with $N$.
\item[($\star$-ii)] The underlying $D$-module $M_{(-1,0]}$ of $\calM_{(-1,0]}$ is decomposed as
\[M_{(-1,0]}=\bigoplus_{\alpha\in (-1,0]\cap \QQ}M_{\alpha},\]
where $M_{\alpha}:=\mathrm{Ker}(T_{s}-e(\alpha))\subset M_{(-1,0]}$.
\item[($\star$-iii)] $vc\colon \calM_{0}\to \calM_{0}(-1)$ is $-N$.
\end{enumerate}
Since the weight filtration of $\calM_{(-1,0]}$ is a finite filtration, $N$ is a nilpotent operator.
Moreover, by ($\star$-ii), $F_{p}M_{(-1,0]}$ ($p\in \ZZ$) is decomposed as
\[F_{p}M_{(-1,0]}=\bigoplus_{\alpha\in (-1,0]\cap \QQ}F_{p}M_{\alpha},\]
where $F_{p}M_{(-1,0],\alpha}=F_{p}M_{(-1,0]}\cap M_{\alpha}$.
Let $\mathscr{G}(X)$ be the category of such tuples.
For a monodromic mixed Hodge module $\calM$ on $X\times \CC$,
we obtain an object in $\mathscr{G}(X)$ by setting $\calM_{(-1,0]}=\psi_{t}\calM$ and $\calM_{-1}=\phi_{t,1}\calM$.
In this way, we can define a functor from the category $\mathrm{MHM}^{p}_{\mathrm{mon}}(X\times \CC)$ of monodromic graded polarizable mixed Hodge modules on $X\times \CC$ to the category $\mathscr{G}(X)$.
Our second main result is the following.
\begin{thm}[see Theorem~\ref{main2kai}]\label{carot}
There is an equivalence of categories
\[\mathrm{MHM}^{p}_{\mathrm{mon}}(X\times \CC)\simeq \mathscr{G}(X).\]
\end{thm}
In fact, we can naturally construct a monodromic mixed Hodge module on $X\times \CC$ from an object in $\mathscr{G}(X)$.

We go back to the original question.
Let $\calM$ be a monodromic mixed Hodge module on $X\times \CC$.
The Fourier-Laplace transformation $\FOU{M}$ of the underlying $D$-module $M$ (for the definition, see Subsection~\ref{bell}) is again monodromic.
Therefore, as explained above, its $D$-module structure is determined by $(\FOU{M})^{\alpha}(:=\bigcup_{l\geq 0}\mathrm{Ker}(t\partial_{t}-\alpha)^l\subset \FOU{M})$ for $\alpha\in [-1,0]\cap \QQ$ with some morphisms.
By the definition of the Fourier-Laplace transformation,
we have
\begin{align*}
(\FOU{M})^{\beta}\simeq M^{-\beta-1}
\end{align*}
as $D$-modules on $X$ and we thus obtain
\[\FOU{M}(=\bigoplus_{\beta\in \QQ}(\FOU{M})^{\beta})=\bigoplus_{\beta\in \QQ}M^{-\beta-1}.\]
By setting $\calM_{(-1,0]}:=\phi_{t,1}\calM\oplus \psi_{t,\neq 1}\calM$ and $\calM_{-1}:=\psi_{t,1}\calM(-1)$,
we can naturally define an object in $\mathscr{G}(X)$.
By Theorem~\ref{carot}, we can construct a monodromic mixed Hodge module $\FOU{\calM}$ on $X\times \CC_{\tau}$ whose underlying $D$-module is $\FOU{M}$. 
Through the identification $\FOU{M}^{\alpha}= M^{-\alpha-1}$,
we have
\[
F_{p}\FOU{M}^{\alpha}:=
\left\{
\begin{array}{l}
F_{p+1}M^{-1}\quad (\alpha=0)\\
F_{p}M^{-\alpha-1}\quad (\alpha\in [-1,0)),
\end{array}
\right.
\]
for $\alpha\in [-1,0]$ and $p\in \ZZ$ (see Lemma~\ref{gegege}).
Note that 
the Hodge filtration $F_{\bullet}\FOU{M}$ is decomposed by Theorem~\ref{hakujitu} and determined only by 
$F_{\bullet}\FOU{M}^{\alpha}$ for $\alpha\in [-1,0]$.

The author expects these results to be linked with the irregular Hodge theory.
More precisely, he hopes there are some relationships between the filtrations on $\FOU{M}$: the Hodge filtration defined above, the irregular Hodge filtration defined in Esnault-Sabbah-Yu~\cite{ESY}, Sabbah-Yu~\cite{SabYuIrrFil} and Sabbah~\cite{IrrHodge}, and the one defined in \cite{RWWGKZ}.
This is left for future work.

\subsection*{Acknowledgments}
A part of this study was done while the author was visiting CMLS, \'Ecole polytechnique from April to October 2018, whose hospitality is gratefully acknowledged.
Another part was fruitfully conducted while he went to University of Tsukuba until March 2020.
The author would like to express his sincere gratitude to Professor Claude Sabbah for suggesting the problem treated in this paper, having stimulating discussions and answering many questions.
Moreover, the author would like to thank him for a careful reading of the manuscript and valuable suggestions. 
He would like to express appreciation to Professor Takuro Mochizuki for inspiring discussions and helpful advice.
He wishes to thank Tatsuki Kuwagaki for fruitful conversations.
His thanks go also to Yuichi Ike for answering his questions.
He would like to thank Professor Kiyoshi Takeuchi for his constant encouragement.
He also thanks the referee for useful comments.
This work is supported by JSPS KAKENHI Grant Number 20J00922.

\section{{M}onodromic $D$-modules}
\subsection{{M}onodromic $D$-modules}
We first recall some definitions and facts about monodromic $D$-modules.
Almost all the results in this section are basic, but we will provide some proofs for the convenience.
We refer to \cite{BryFou} for the details.
We also refer to Section~10 of \cite{Ginsburg} for some results on monodromic $D$-modules, especially, an index formula for monodromic $D$-modules.

Throughout this paper, we use left $D$-modules.
Accordingly, there are some differences in the conventions compared to \cite{HM88} and \cite{MHM}.
For example, the Kashiwara-Malgrange filtrations (see Subsection~\ref{kasmalsec}) are decreasing filtrations in this paper, while in \cite{HM88} and \cite{MHM} they are increasing filtrations.

Let $X$ be a smooth algebraic variety of finite type over $\CC$, $O_{X}$ the structure sheaf on $X$ 
and $D_{X}$ the sheaf of differential operators on $X$.
In this paper, we will basically consider only algebraic (quasi-coherent) $D$-modules (nevertheless, we use analytic $D$-modules in the proof).
Moreover, all the $D$-modules and $O$-modules are quasi-coherent.
Let $t$ be the coordinate of $\CC$.
We sometimes write $\CC_{t}$ to emphasize it.

\begin{defi}\label{mondef}
Let $M$ be a $D$-module on $X\times \CC_{t}$ (resp. $X\times \CS_{t}$).
If for any affine open subset $U\subset X$ and any section $m$ of $M$ on $U\times \CC_{t}$ (resp. $U\times \CC_{t}^*$)
there exists a non-zero polynomial $b(s)\in \CC[s]$ such that
$b(t\partial_{t})m=0$,
we say that $M$ is a monodromic $D$-module.
Moreover, if we can take such $b(s)$ whose roots are contained in $\KK$,
we say that $M$ is $\KK$-monodromic, where $\KK$ is $\RR$ or $\QQ$.
\end{defi}

\begin{lem}
For a $D$-module $M$ on $X\times\CC_{t}$,
$M$ is monodromic if and only if $M|_{X\times \CS_{t}}$ is monodromic.
\end{lem}
\begin{proof}
Let $m$ be a section (on an affine open subset of the form $U\times \CC_{t}$) of $M$ and
assume that $M|_{X\times \CS_{t}}$ is monodromic.
Then, there exists a polynomial $b(s)\in \CC[s]$ such that
$b(t\partial_{t})m|_{X\times \CS_{t}}=0$.
Therefore, the support of $b(t\partial_{t})m$ is contained in $X\times \{0\}$,
and hence for sufficiently large $l\in \ZZ_{\geq 0}$, $t^{l}b(t\partial_{t})m$ is zero.
This implies that there is a polynomial $b'(s)\in \CC[s]$ such that
$b'(t\partial_{t})m=0$. 
\end{proof}

\begin{rem}\label{charmon}
When $M$ is regular holonomic,
we can define ``monodromic'' as a property of the corresponding perverse sheaf as follows.
Let $K$ be a perverse sheaf on $X\times \CC_{t}$ such that we have $\mathrm{DR}_{X\times \CC}(M)\simeq K$, where $\mathrm{DR}_{X\times \CC}$ is the de Rham functor.
Then, 
by Proposition~7.12 of \cite{BryFou},
$M$ is monodromic if and only if
all the cohomology sheaves $H^{j}(K)$ ($j\in \ZZ$) of the complex of sheaves $K$
are locally constant on any $\CC^*$-orbit of $X\times \CC_{t}$, i.e. $\{x\}\times \CS$ and $\{(x,0)\}$ for $x\in X$. 

We can write this condition in terms of the micro support (the singular support) $\mathrm{SS}(K)$ of $K$.
Let $p\colon X\times \CS\to X$ be the projection and ${}^tp'\colon (X\times \CS)\times_{X}T^{*}X\to T^{*}(X\times \CS)$ the transposed map of the differential map of $p$.
Then, by Proposition~5.4.5 of \cite{KS}, the the above condition is equivalent to the condition:
\begin{align*}
\mathrm{SS}(K|_{X\times \CS})\subset {}^tp'((X\times \CS)\times_{X}T^{*}X)=T^{*}X\times T^{*}_{\CS}\CS,
\end{align*}
where $T^{*}_{\CS}\CS$ is the bundle of zero sections of the cotangent bundle $T^{*}\CS$.
Assume that $M$ is regular.
Then, since the micro support of $K$ coincides the characteristic variety $\mathrm{Char}(M)$ of $M$,
we can say that
$M$ is monodromic if and only if it satisfies that
\begin{align*}
\mathrm{Char}(M|_{X\times \CS})\subset T^{*}X\times T^{*}_{\CS}\CS.
\end{align*}

\end{rem}

\begin{rem}\label{rsprmon}
Assume that $M$ is coherent.
In the situation of Definition~\ref{mondef},
for a section $m\in M$ the set of all polynomials satisfying $b(t\partial_{t})m=0$ forms an non-zero ideal in $\CC[s]$.
Take a monic generator $b_{0}(s)$ of it.
Then, by Proposition~\ref{Vprop} below, we can see that
any root $\beta\in \CC$ of $b_{0}(s)$ 
is a root of the Bernstein polynomial of the Kashiwara-Malgrange filtration of $M$ along $t=0$ up to $\ZZ$.
Therefore, 
$M$ is $\KK$($=\QQ$ or $\RR$)-monodromic
if and only if
$M$ is monodromic 
and $\KK$-specializable along $t=0$ (see subsection~\ref{kasmalsec}).
\end{rem}

\begin{rem}\label{qmon}
Since we only deal with $\QQ$-monodromic $D$-modules in this paper,
we will simply say ``monodromic'' as ``$\QQ$-monodromic''.
\end{rem}

We denote by $p\colon X\times \CC\to X$ the projection of $X\times \CC$ to $X$.
Then, we get a $p_{*}D_{X\times \CC}$-module $p_{*}M$ and we can also endow it with a $D_{X}$-module structure by the adjunction $D_{X}\to p_{*}p^{-1}D_{X}\to p_{*}D_{X\times \CC}$.
Moreover, 
$t$ and $\partial_{t}$ define $D_{X}$-module endomorphisms of $p_{*}M$,
and hence $p_{*}M$ is a $D_{X}[t]\langle{\partial_{t}}\rangle(:=D_{X}\otimes_{\CC}\CC[t]\langle\partial_{t}\rangle)$-module,
where $\CC[t]\langle\partial_{t}\rangle$ is the ring of differential operators on $\CC_{t}$.
Conversely,
for a $D_{X}[t]\langle{\partial_{t}}\rangle$-module $N$,
we define a $D_{X\times \CC_{t}}$-module $p^{*}N$ as
\[p^{*}N:=D_{X\times \CC_{t}} \otimes_{p^{-1}D_{X}[t]\langle{\partial_{t}}\rangle} p^{-1}N.\]
Similarly, 
for an $O_{X\times \CC_{t}}$-module $M$, $p_{*}M$ is an $O_{X}[t](:=O_{X}\otimes_{\CC}\CC[t])$-module
and for an $O_{X}[t]$-module $N$, we can define an $O_{X\times \CC_{t}}$-module $p^{*}N$ as
\[p^{*}N:=O_{X\times \CC_{t}}\otimes_{p^{-1}O_{X}[t]}p^{-1}N.\]
Since $\CC_{t}$ is affine and our $D$-modules and $O$-modules are quasi-coherent, we can easily check the following.

\begin{lem}\label{peq}
The operators $p_{*}$ and $p^{*}$ define an equivalence of categories
between the category of quasi-coherent $D_{X\times \CC_{t}}$-modules (resp. $O_{X\times \CC_{t}}$-modules) and the one of quasi-coherent $D_{X}[t]\langle\partial_{t}\rangle$-modules (resp. $O_{X}[t]$-modules).
The same statements hold also for the categories of $D_{X\times \CS_{t}}$,
$D_{X}[t^{\pm}]\langle\partial_{t}\rangle$,
$O_{X\times \CS_{t}}$ and
$O_{X}[t^{\pm}]$-modules.
\end{lem}

In the following, we identify $D_{X\times \CC_{t}}$-modules (resp. $D_{X\times \CS_{t}}$-modules, $O_{X\times \CC_{t}}$-modules, $O_{X\times \CS_{t}}$-modules) with $D_{X}[t]\langle\partial_{t}\rangle$-modules (resp. $D_{X}[t^{\pm}]\langle\partial_{t}\rangle$-modules, $O_{X}[t]$-modules, $O_{X}[t^{\pm}]$-modules) through the functors $p_{*}$ and $p^{*}$.

For a $D_{X\times \CC_{t}}$ or $D_{X\times \CS}$-module $M$ and $\beta\in \QQ$,
we define a $D_{X}$-submodule $M^{\beta}$ of $M$ by
\begin{align*}
M^{\beta}:=\bigcup_{l\geq 0}\mathrm{Ker}((t\partial_{t}-\beta)^l\colon M\to M).
\end{align*}

\begin{prop}\label{bunkailem}
A $D_{X\times \CC_{t}}$ or $D_{X\times \CS_{t}}$-module $M$ is monodromic
if and only if $M$ is a direct sum of the family of $D_{X}$-modules
$\{M^{\beta}\}_{\beta\in \QQ}$, i.e. we have a decomposition: 
\begin{align}\label{jhiphop}
M=\bigoplus_{\beta\in \QQ}M^{\beta}.
\end{align}

\end{prop}

\begin{proof}
For any $\beta_{1}\neq \beta_{2}$,
take a germ of a local section $m\in M^{\beta_{1}}\cap M^{\beta_{2}}$.
By the definition of $M^{\beta}$,
for sufficiently large $l\geq 0$ we have
$(t\partial_{t}-\beta_{1})^lm=0$ and $(t\partial_{t}-\beta_{2})^lm=0$.
Since polynomials $(s-\beta_{1})^l$ and $(s-\beta_{2})^l$ are coprime,
there exist polynomials $u(s), v(s)\in \CC[s]$ such that
$u(s)(s-\beta_{1})^l+v(s)(s-\beta_{2})^l=1$.
Therefore, 
we have $m=0$.
This implies that $\bigoplus_{\beta\in \QQ}M^{\beta}$ is an $O_{X}$-submodule of $M$.

We assume that $M$ is monodromic.
Let $m$ be a germ of a local section of $M$.
Then, there exists a monic polynomial $b(s)\in \CC[s]$ such that $b(t\dt)m=0$.
Let $b(s)=\prod_{i=1}^{r}(s-\beta_{i})^{k_{i}}$ be its factorization.
To show $m$ is in $\bigoplus_{\beta\in \QQ} M^{\beta}$, we use induction on $r$.
If $r=1$, $m$ is in $M^{\beta_{1}}(\subset \bigoplus_{\beta\in \QQ}M^{\beta})$.
If $r\geq 2$,
since $(s-\beta_{1})^{l_{1}}$ and $\prod_{i=2}^{r}(s-\beta_{i})^{l_{i}}$ are coprime,
there exist polynomials $u(s), v(s)\in \CC[s]$ such that
$u(s)(s-\beta_{1})^{l_{1}}+v(s)\prod_{i=2}^{r}(s-\beta_{i})^{l_{i}}=1$.
Then, we have
\[m=u(t\dt)(t\dt-\beta_{1})^{l_{1}}m+(v(t\dt)\prod_{i=2}^{r}(t\dt-\beta_{i})^{l_{i}})m.\]
It is clear that $(v(t\dt)\prod_{i=2}^{r}(t\dt-\beta_{i})^{l_{i}})m$ is in $M^{\beta_{1}}$.
Moreover, 
$u(t\dt)(t\dt-\beta_{1})^{l_{1}}m$ is killed by the operation $\prod_{i=2}^{r}(t\dt-\beta_{i})^{l_{i}}$.
By the induction hypothesis, $u(t\dt)(t\dt-\beta_{1})^{l_{1}}m$ is in $\bigoplus_{\beta\in \QQ}M^{\beta}$, and hence $m$ is also in $\bigoplus_{\beta\in \QQ}M^{\beta}$.
This completes the proof.
\end{proof}

\begin{rem}\label{watanuki}
The category of monodromic $D_{X\times \CC_{t}}$-modules is closed under taking submodules, quotient modules and extensions in the category of $D_{X\times \CC_{t}}$-modules. 
Consequently, the category of monodromic $D_{X\times \CC_{t}}$-modules is abelian.
Moreover, let $M'$ be a submodule of a monodromic $D_{X\times \CC_{t}}$-module $M$.
Then, we have 
\[(M')^{\beta}=M'\cap M^{\beta}{\quad (\beta\in \QQ)},\]
and by Lemma~\ref{bunkailem} we obtain a decomposition 
\[M'=\bigoplus_{\beta\in \QQ}M'\cap M^{\beta}.\] 
Similar statements hold for quotient modules of $M$ and extensions of monodromic $D$-modules.
\end{rem}

\begin{rem}\label{affinerem}
We identify a quasi-coherent $O_{Z}$-module $F$ on an affine variety $Z$ with the $O(Z)$-module $\Gamma(Z,F)$ : the global sections of $F$.
Then,
if $X$ is affine, 
$M$ is monodromic
if and only if 
we have a decomposition
\begin{align}\label{Mdecomp}
M=\bigoplus_{\beta\in \QQ}M^{\beta},
\end{align}
as a $\Gamma(X\times \CC,D_{X\times \CC_{t}})=\Gamma(X,D_{X})\otimes\CC[t]\langle{\partial_{t}}\rangle$-module,
where $M^{\beta}=\{m\in M\ |\ (t\partial_{t}-\beta)^{l}m=0 \ \mbox{for some $l\geq 0$}\}$.
\end{rem}

\begin{prop}\label{isomprop}
Let $M$ be a monodromic $D_{X\times \CC_{t}}$-module.
Then, for $\beta\in \QQ$ we have the morphisms:
\begin{align*}
t&\colon M^{\beta}\to M^{\beta+1}\mbox{\qquad and}\\
\partial_{t}&\colon  M^{\beta}\to M^{\beta-1}.
\end{align*}
Moreover, the above morphisms except for
\begin{align*}
t&\colon M^{-1}\to M^{0}\mbox{\qquad and}\\
\partial_{t}&\colon  M^{0}\to M^{-1}
\end{align*}
are isomorphism.
\end{prop}

\begin{proof}
The first statement is clear by the definition of $M^{\beta}$.
We show the second statement.
Since the problem is local, we may assume that $X$ is affine.
So we use the notation of Remark~\ref{affinerem}.
For $\beta\in \QQ$,
we have
\[(t\partial_{t}-\beta)^l=Pt+(-\beta-1)^l,\]
for some operator $P$ of the form $\sum_{i-j=-1}a_{i,j}t^{i}\partial^{j}_{t}\ (a_{i,j}\in \CC)$.
Therefore, on the set $\mathrm{Ker}((t\partial_{t}-\beta)^l\colon M^{\beta}\to M^{\beta})$,
we have 
\[Pt=-(-\beta-1)^l.\]
Hence, if $\beta\neq -1$, the morphism $t\colon M^{\beta}\to M^{\beta+1}$ is injective.
Similarly,
on the set $\mathrm{Ker}((t\partial_{t}-\beta-1)^l\colon M^{\beta}\to M^{\beta})$,
we have
\[\partial_{t}Q=-(-\beta-1)^l\]
for some operator $Q$ of the form $\sum_{i-j=-1}a_{i,j}t^{i}\partial^{j}_{t}\ (a_{i,j}\in \CC)$.
Since the operator $Q$ sends $M^{\beta+1}$ to $M^{\beta}$,
this implies that the morphism $t\colon M^{\beta}\to M^{\beta+1}$ is surjective if $\beta\neq -1$.
We can check a similar statement for $\partial_{t}\colon M^{\beta}\to M^{\beta-1}$ in the same way.  
\end{proof}

\begin{rem}
By Proposition~\ref{isomprop}, we can reconstruct the monodromic $D$-module $M$ on $X\times \CC$ from the data:
\[
\left\{
\begin{array}{ll}
\mbox{(i)}&\mbox{The $D$-modules $M^{\alpha}$ on $X$ for $\alpha\in [-1,0]\cap \QQ$.}\\
\mbox{(ii)}&\mbox{The nilpotent endomorphisms $t\partial_{t}-\alpha\colon M^{\alpha}\to M^{\alpha}$ for $\alpha\in (-1,0)\cap \QQ$.}
\\
\mbox{(iii)}&\mbox{The morphisms $\partial_{t}\colon M^{0}\to M^{-1}$ and $t\colon M^{-1}\to M^{0}$}\\
&\mbox{such that $t\partial_{t}$ and $\partial_{t}t$ are nilpotent.}
\end{array}
\right.
\]
\end{rem}
\begin{prop}\label{cohprop}
Let $M$ be a monodromic $D$-module on ${X\times \CC_{t}}$.
Then, 
$M$ is a coherent $D_{X\times \CC_{t}}$-module
if and only if the following two properties hold:
\[\left\{
\begin{array}{ll}
\mbox{\rm(i)}&\mbox{there exists a finite subset $A\subset [-1,0]\cap \QQ$ s.t. $M^{\beta}=0\ (\beta\in [-1,0]\setminus A)$,}\\
\mbox{\rm(ii)}&\mbox{all $M^{\alpha}\ (\alpha\in A)$ are coherent $D_{X}$-modules
}.\end{array}
\right.\]
\end{prop}

\begin{proof}
Since the problem is local,
we may assume that $X$ is affine.

Assume that $M$ is a coherent $D_{X\times \CC_{t}}$-module.
We take generators $m_{1},\dots,m_{r}$ of $M$ as a $D_{X\times \CC_{t}}$-module.
We may assume that $m_{i}$ are homogeneous, i.e. $m_{i}\in M^{\beta_{i}}$ for some $\beta_{i}\in \QQ$.
Then, the submodule of $M$ generated by $m_{i}$ is contained in $\bigoplus_{l\in \ZZ}M^{\beta_{i}+l}$.
Therefore, the condition (i) above must be satisfied.
For any $\beta\in \QQ$, take an element $m\in M^{\beta}$.
Then, consider the union of the two subsets of $M^{\beta}$: 
\begin{align*}
\{\partial^{\beta_{i}-\beta}_{t}(t\partial_{t})^lm_{i}\in M\ |\ l\geq 0,\ \beta_{i}\in \beta+\ZZ_{\geq 0}\},\mbox{\ and}\\
\{t^{\beta-\beta_{i}}(t\partial_{t})^lm_{i}\in M\ |\ l\geq 0,\ \beta_{i}\in \beta+\ZZ_{\leq 0}\}.
\end{align*}
Since $(t\partial_{t}-\beta_{i})^lm_{i}=0$ for some $l\geq 0$,
the subset of $M^{\beta}$ is finite.
We can see that $M^{\beta}$ is generated by this finite subset as $D_{X}$-module, i.e. the condition (ii) is satisfied.

Conversely, assume that we have the condition (i) and (ii).
We consider the union of finite generators of all $M^{\beta}$ for $\beta\in [-1,0]\cap \QQ$ as $D_{X}$-module.
By the condition (i), this set is finite.
Moreover, by Proposition~\ref{isomprop}, this set generates $M=\bigoplus_{\beta\in \QQ}M^{\beta}$ as $D_{X\times \CC_{t}}$-module.
\end{proof}

\begin{rem}\label{yugenkrem}
Assume that $X$ is affine and $M^{\beta}$ is coherent.
Let $m_{1},\dots,m_{r}$ be generators of $M^{\beta}$.
Since each $m_{i}$ is killed by $(t\partial_{t}-\beta)^{l_{i}}$ for some $l_{i}\geq 0$,
for a sufficiently large $l\geq 0$
we have
\[M^{\beta}=\mathrm{Ker}((t\partial_{t}-\beta)^l\colon M\to M).\]

Since our $X$ is of finite type over $\CC$, this formula is valid for a sufficiently large $l\geq 0$ even if $X$ is not affine.
\end{rem}

\begin{prop}\label{intlemkai}
Let $M$ be a coherent and monodromic $D_{X\times \CC_{t}}$-module.
Then, $M|_{X\times \CC^*}$ is a locally free $O_{X\times \CC^*}$-module of finite rank if and only if all the coherent $D_{X}$-modules $M^{\alpha}$ for $\alpha\in (-1,0]\cap \QQ$ are locally free $O_{X}$-modules of finite rank.
\end{prop}

\begin{proof}
Assume that $M^{\alpha}$ are locally free of finite rank for every $\alpha\in (-1,0]\cap \QQ$.
Since there are finitely many non-zero $M^{\alpha}$ ($\alpha\in (-1,0]\cap \QQ$)
by Proposition~\ref{cohprop},
we can take a neighborhood $U$ of a point in $X$ such that
all $M^{\alpha}|_{U}$ are locally free of finite rank.
Since we have $M|_{U\times \CS}=(\bigoplus_{\alpha\in (-1,0]\cap \QQ}M^{\alpha}|_{U})\otimes \CC[t^{\pm}]$ by Proposition~\ref{isomprop},
this is also a locally free $O_{X\times \CS}$-module of finite rank.

Conversely, assume that $M|_{X\times \CC^*}$ is a locally free $O_{X\times \CS}$-module of finite rank.
On a sufficiently small neighborhood of a point in $X\times \CC^*$,
we take finite generators (as an $O_{X\times \CS}$-module) $\sigma_{1},\dots,\sigma_{l}$ of $M|_{X\times \CC^*}$.
By Proposition~\ref{isomprop},
we may assume that $\sigma_{i}$ is in $M^{\alpha_{i}}$ for some $\alpha_{i}\in (-1,0]\cap \QQ$ for $i=1,\dots,l$.
Then, we can see that for any $\alpha\in (-1,0]\cap \QQ$, $M^{\alpha}$ is generated (as an $O_{X}$-module) by $\{\sigma_{i}\ |\ \alpha_{i}=\alpha\}\subset M^{\alpha}$.
Since any $D$-module which is coherent over $O$ is a locally free $O$-module of finite rank (see Theorem~1.4.10 of \cite{HTT}),
$M^{\alpha}$ is a locally free $O_{X}$-module of finite rank.
\end{proof}

\subsection{The Kashiwara-Malgrange filtrations of monodromic $D$-modules}\label{kasmalsec}
We consider the Kashiwara-Malgrange filtration of a monodromic $D$-module.
We refer to \cite{MHM}, \cite{MHMProj}, \cite{LeiNear}, \cite{MaiMeb}, \cite{HoSiD}, \cite{Schnell}, \cite{Shannon}, \cite{Kashivan} for the definition and results for the $V$-filtration and Kashiwara-Malgrange filtration. 
For a coherent $D_{X\times \CC}$-module $M$, we denote by $\{V^{\beta}_{t}M\}_{\beta\in \RR}$ the Kashiwara-Malgrange filtration of $M$ along $t=0$ when it exists.
Recall that $V^{\beta}_{t}M$ is a coherent $V^{0}_{t}D_{X\times \CC}(=D_{X}[t,t\partial_{t}])$-module and $\GR^{\beta}_{V}M=V_{t}^{\beta}M/V_{t}^{>\beta}M$ is killed by $(t\partial_{t}-\beta)^l$ for some $l\in \ZZ_{\geq 0}$.
We can express the Kashiwara-Malgrange filtration of a monodromic $D$-module as follows.
\begin{prop}\label{Vprop}
Let $M$ be a monodromic coherent $D_{X\times \CC_{t}}$-module.
Then, 
the Kashiwara-Malgrange filtration of $M$ along $t=0$ exists and we have
\begin{align*}
V^{\gamma}_{t}M=\bigoplus_{\beta\geq \gamma}M^{\beta}.
\end{align*}
Therefore,
\begin{align*}
\GR^{\gamma}_{V}M(:=V^{\gamma}_{t}M/V^{>\gamma}_{t}M)\simeq M^{\gamma}.
\end{align*}
\end{prop}
\begin{proof}
By the uniqueness of the Kashiwara-Malgrange filtration,
the problem is local,
and hence we may assume that $X$ is affine.
We set $U^{\gamma}M:=\bigoplus_{\beta\geq \gamma}M^{\beta}$ for $\gamma\in \RR$.
By Proposition~\ref{cohprop},
all $M^{\beta}$ are coherent.
Since $M^{\beta}$ is killed by $(t\partial_{t}-\beta)^l$ for sufficiently large $l\geq 0$,
by Proposition~\ref{isomprop} 
we can see that $U^{\gamma}M$ is a coherent $D_{X}[t]\langle t\partial_{t}\rangle$-module.
Moreover, $U^{\gamma}M/U^{>\gamma}M$ is isomorphic to $M^{\gamma}$ and killed by $(t\partial_{t}-\gamma)^l$ for some $l\geq 0$.
Furthermore, it is clear that 
$\bigcup_{\gamma\in \RR}U^{\gamma}M=M$,
$tU^{\gamma}M\subset U^{\gamma}M$ and
$\partial_{t}U^{\gamma}M\subset U^{\gamma-1}M$.
By Proposition~\ref{isomprop},
$tU^{\gamma}M=U^{\gamma+1}M$ for $\gamma>-1$.
Because the Kahiwara-Malgrange filtration is characterized by these properties,
$\{U^{\gamma}M\}_{\gamma\in \RR}$ is the Kashiwara-Malgrange filtration.
\end{proof}

\begin{rem}
Since our monodromic $D$-module $M$ is $\QQ$-monodromic (Remark~\ref{qmon}),
the filtration $V_{t}^{\bullet}M$ jumps only at rational numbers.
Therefore, we can say that $M$ is $\QQ$-specializable along $t=0$.  
\end{rem}

For $\alpha\in (-1,0]\cap \QQ$,
we set $e(\alpha):=\exp(-2\pi\sqrt{-1}\alpha)$.
Then,
we define the $e(\alpha)$-nearby cycle $\psi_{t,e(\alpha)}(M)$ of $M$ as $V^{\alpha}_{t}M/V^{>\alpha}_{t}M$ and
the unipotent vanishing cycle $\phi_{t,1}(M)$ of $M$ as $V^{-1}_{t}M/V^{>-1}_{t}M$ ($\psi_{t,e(\alpha)}(M)$ and $\phi_{t,1}(M)$ are $D$-modules on $X$).
Moreover, we have a natural morphism 
$\psi_{t,1}(M)\to \phi_{t,1}(M)$ defined by $-\partial_{t}$
and
$\phi_{t,1}(M)\to \psi_{t,1}(M)$ defined by $t$.
The former (resp. latter) one is written by $\mathrm{can}$ (resp. $\mathrm{var}$).
We also define the automorphisms of $\psi_{t,e(\alpha)}(M)$ and $\phi_{t,1}(M)$
by $T:=\exp(-2\pi\sqrt{-1}t\partial_{t})$.
When $M$ is regular holonomic, 
these notions correspond to 
the ones in the category of perverse sheaves respectively.
For example, we have
\begin{align*}
\mathrm{DR_{X}}(\psi_{t,e(\alpha)}(M))&\simeq {}^p\psi_{t,\mathrm{e(\alpha)}}(\mathrm{DR}_{X\times \CC}(M)),\quad\mbox{and}\\
\mathrm{DR_{X}}(\phi_{t,1}(M))&\simeq {}^p\phi_{t,1}(\mathrm{DR}_{X\times \CC}(M)),
\end{align*}
where ${}^p\psi_{t,\mathrm{e(\alpha)}}$ (resp. $ {}^p\phi_{t,1}$) is the nearby (resp. vanishing) cycle functor between the category of perverse sheaves and $\mathrm{DR}_{X}$ is the de Rham functor.

As a corollary of Proposition~\ref{Vprop}, we can express the nearby and vanishing cycles as follows.

\begin{cor}\label{nevancor}
Let $M$ be a monodromic coherent $D_{X\times \CC_{t}}$-module.
Then, for $\alpha\in (-1,0]\cap \QQ$ we have
\begin{align*}
\psi_{t,e(\alpha)}(M)&=M^{\alpha},\quad \mbox{and}\\
\phi_{t,1}(M)&=M^{-1}.
\end{align*}
Moreover, 
$\mathrm{can}\colon \psi_{t,1}(M)\to \phi_{t,1}(M)$ 
is $-\partial_{t}\colon M^{0}\to M^{-1}$,
$\mathrm{var}\colon \phi_{t,1}(M)\to \psi_{t,1}(M)$
is $t\colon M^{-1}\to M^{0}$,
and 
$\frac{-1}{2\pi\sqrt{-1}}$ times the logarithm of the unipotent part of the monodromy automorphism $T=\exp(-2\pi\sqrt{-1}t\partial_{t})$ on $\psi_{t,e(\alpha)}(M)=M^{\alpha}$ (resp. $\phi_{t,1}M=M^{-1}$) is $t\partial_{t}-\alpha$ (resp. $t\partial_{t}+1$).
\end{cor}

We also use the nearby cycle $\psi_{t}(M):=V_{t}^{>-1}M/V_{t}^{>0}M=
\bigoplus_{\alpha\in (-1,0]\cap \QQ}\psi_{t,e(\alpha)}(M)$ of $M$ along $t=0$.
Then, we have $\psi_{t}(M)= \bigoplus_{\alpha\in (-1,0]\cap \QQ}M^{\alpha}$ and
$\frac{-1}{2\pi\sqrt{-1}}$ times the logarithm of the unipotent part of the monodromy 
automorphism $T$ can be expressed as
$\bigoplus_{\alpha\in (-1,0]\cap \QQ}(t\partial_{t}-\alpha)$.

\subsection{{D}uality of monodromic $D$-modules}
We consider the dual of a monodromic $D$-module.
Our proof in this section is inspired by \cite{MaiMeb}.

\begin{lem}\label{duares}
Let $M$ be a monodromic coherent $D$-module on $X\times \CC_{t}$.
Then, locally, there exists a resolution of $M$:
\[\dots \to\calE_{2}\to \calE_{1}\to \calE_{0}\to M\to 0,\]
where $\calE_{i}$ is a direct sum of finitely many monodromic $D$-modules of the form:
$D_{X\times \CC_{t}}/D_{X\times \CC_{t}}(t\partial_{t}-\alpha)^{l}$
for some $\alpha\in [-1,0]\cap \QQ$ and $l\in \ZZ_{\geq 1}$.
\end{lem}

\begin{proof}
We may assume that $X$ is affine.
Take generators $m_{1},\dots,m_{r}$ of $M$ as $D$-module such that $m_{i}$ is in $M^{\alpha_{i}}$ for some $\alpha_{i}\in [-1,0]\cap \QQ$.
Then, for a sufficiently large $l_{i}\in \ZZ_{\geq 1}$,
we have $(t\partial_{t}-\alpha_{i})^{l_{i}}m_{i}=0$.
We thus obtain a surjection
\begin{align}\label{GOGO}
\calE_{0}:=\bigoplus_{i=1}^{r}D_{X\times \CC_{t}}/D_{X\times \CC_{t}}(t\partial_{t}-\alpha_{i})^{l_{i}}
\twoheadrightarrow M\to 0.
\end{align}
By the definition, $\calE_{0}$ is coherent and monodromic and so is the kernel of (\ref{GOGO}).
Therefore, we can repeat this procedure and obtain a desired resolution.   
\end{proof}

We remark that $\calE_{i}$ in Lemma~\ref{duares} is not holonomic (unless $X$ is one point).
Moreover, the $D$-module $D_{X\times \CC_{t}}/D_{X\times \CC_{t}}(t\partial_{t}-\alpha)^{l}$ has a resolution
\[0\to D_{X\times \CC_{t}}\to D_{X\times \CC_{t}}\to D_{X\times \CC_{t}}/D_{X\times \CC_{t}}(t\partial_{t}-\alpha)^{l}\to 0.\]
Recall that for a $D_{X\times \CC}$-module $M$ its dual is defined as 
\[\mathbb{D}M=\DR \mathcal{H} om_{D_{X\times \CC}}(M,D_{X\times \CC})\otimes_{O_{X\times \CC}}\omega_{X\times \CC}^{\otimes -1}[\dim{X}+1].\]
\begin{lem}\label{actaa}
Each cohomology of the dual $\mathbb{D}M$ of a monodromic coherent $D$-module $M$ is also monodromic.
\end{lem}

\begin{proof}
We may assume that $X$ is affine.
We write $D$ for $D_{X\times \CC}$ in this proof.
Take a resolution $\calE_{\bullet}\to M\to 0$ of $M$ as in Lemma~\ref{duares}.
Then, we have $\DR \mathcal{H} om_{D}(M,D)\simeq \DR \mathcal{H} om_{D}(\calE_{\bullet},D)$ in $\mathrm{D}^{-}(D_{X\times \CC})$.
For each $\calE_{i}$, we can take a resolution
\[0\to D^{l_{i}}\to D^{l_{i}}\to\calE_{i}\to 0.\]
Then, we can consider the following diagram (we can define horizontal arrows):
\[\xymatrix{
&0&0&0&&\\
\dots \ar[r]&\calE_{2}\ar[u]\ar[r] &\calE_{1}\ar[u]\ar[r]& \ar[u]\calE_{0}\ar[r]& M\ar[r]& 0\\
\dots \ar[r]&D^{l_{2}}\ar[u]\ar[r]&D^{l_{1}}\ar[u]\ar[r]&\ar[u]D^{l_{0}}&&\\
\dots \ar[r]&D^{l_{2}}\ar[u]\ar[r]&D^{l_{1}}\ar[u]\ar[r]&\ar[u]D^{l_{0}}&&\\
&0\ar[u]&0\ar[u]&0\ar[u]&&.
}\]
Since the vertical sequences are exact and $\mathcal{H} om_{D}(D,D)$ is isomorphic to $D$,
$\DR \mathcal{H} om_{D}(\calE_{\bullet},D)$ is 
the single complex of a double complex of the form:
\[\xymatrix{
\dots &D^{l_{2}}\ar[d]\ar[l]&D^{l_{1}}\ar[d]\ar[l]&\ar[l]D^{l_{0}}\ar[d]\\
\dots &D^{l_{2}}\ar[l]&D^{l_{1}}\ar[l]&\ar[l]D^{l_{0}},
}\]
where the degree of 
$D^{l_{0}}$ at the top right
is $0$ and the vertical arrows $D^{l_{i}}\to D^{l_{i}}$ are the multiplications by $\bigoplus_{j=1}^{l_{i}}(t\partial_{t}-\alpha_{j})^{s_{j}}$ for some $\alpha_{j}\in [-1,0]\cap\QQ$ and $s_{j}\in \ZZ_{\geq 1}$.
The cohomology of its single complex is clearly monodromic.
\end{proof}

\begin{rem}
The characteristic variety of $\mathbb{D}M$ is the same as the one of $M$.
Therefore, if $M$ is regular holonomic, Lemma~\ref{actaa} is a direct consequence of Remark~\ref{charmon}.
\end{rem}

\begin{rem}\label{dualem}
For a monodromic coherent $D$-module on $X\times \CC_{t}$ and $j\in \ZZ$
we have a decomposition of the $j$-th cohomology of its dual:
\[H^{j}\mathbb{D}M=\bigoplus_{\beta\in \QQ}(H^{j}\mathbb{D}M)^{\beta}.\]
By Proposition~4.6-2 of \cite{MaiMeb},
we have
\[
(H^{j}\mathbb{D}M)^{\alpha}\simeq \left\{
\begin{array}{l}
H^{j}\mathbb{D}(M^{-1-\alpha})\quad (\alpha\in (-1,0)),\\
H^{j}\mathbb{D}(M^{0})\quad (\alpha=0),\\
H^{j}\mathbb{D}(M^{-1})\quad (\alpha=-1).
\end{array}\right.
\]
Moreover, for $\alpha\in (-1,0)$ the nilpotent endomorphism $t\partial_{t}-\alpha$ on $(H^{j}\mathbb{D}M)^{\alpha}$ is 
the transpose of $t\partial_{t}+1+\alpha$ on $M^{-1-\alpha}$
and $t\colon (H^{j}\mathbb{D}M)^{-1}\to (H^{j}\mathbb{D}M)^{0}$ (resp. $\partial_{t}\colon (H^{j}\mathbb{D}M)^{0}\to (H^{j}\mathbb{D}M)^{-1}$) is the transpose of
$-\partial_{t}\colon M^{0}\to M^{-1}$ (resp. $t\colon M^{-1}\to M^{0}$) (see loc. cit.).
In this way, we can describe the $D$-module structure on $H^{j}\mathbb{D}M$ concretely in terms of $M^{\alpha}$ ($\alpha\in [-1,0]$) and the morphisms between them.
\end{rem}

\subsection{{R}egular holonomic monodromic $D$-modules}

We consider the holonomicity and the regularity of monodromic $D$-modules.

\begin{prop}\label{holprop}
Let $M$ be a coherent monodromic $D_{X\times \CC_{t}}$-module.
Then, $M$ is holonomic if and only if
all the coherent $D_{X}$-modules $M^{\beta}$ ($\beta\in \QQ$) are holonomic.
\end{prop}

\begin{rem}\label{nenndoro}
By Proposition~\ref{isomprop}, we can also state that
$M$ is holonomic if and only if 
all the coherent $D_{X}$-modules $M^{\alpha}\ (\alpha\in [-1,0]\cap \QQ)$ are holonomic. 
\end{rem}

\begin{proof}[Proof of Proposition~\ref{holprop}]

Recall that a coherent $D$-module is holonomic if and only if the cohomologies of its dual vanish except for the $0$-th one (see Corollary~2.6.8 of \cite{HTT}).
Therefore, the assertion follows from Remark~\ref{dualem}.
\end{proof}

\begin{prop}\label{regprop}
Let $M$ be a holonomic and monodromic $D_{X\times \CC_{t}}$-module.
Then, $M$ is regular if and only if all the holonomic $D_{X}$-modules
$M^{\alpha}\ (\alpha\in [-1,0]\cap \QQ)$ are regular.
\end{prop}

\begin{proof}
Assume first that $M$ is regular.
Since the functors $\psi_{t,e(\alpha)}$ and $\phi_{t,1}$ preserves the regularity (see Th\'eor\`eme~4.3.2 of \cite{MebCom}),
$M^{\alpha}=\psi_{t,e(\alpha)}(M)$ for $\alpha\in (-1,0]\cap \QQ$ and $M^{-1}=\phi_{t,1}(M)$ (see Corollary~\ref{nevancor}) are also regular. 

Conversely, assume that all $M^{\alpha}$ are regular.
Take sufficiently large $k\in \ZZ_{\geq 1}$
such that $(t\partial_{t}-\alpha)^k$ is zero on $M^{\alpha}$ for all $\alpha\in [-1,0]\cap \QQ$. 
Then, we have a morphism of $D_{X\times \CC}$-modules
\[M^{\alpha}\boxtimes D_{\CC}/D_{\CC}(t\partial_{t}-\alpha)^k\to M,\]
which sends $m\boxtimes [P]$ to $Pm$ for $m\in M^{\alpha}$ and $P\in D_{\CC}$.
Let $\alpha_{1},\dots,\alpha_{l}$ be rational numbers in $[-1,0]$
such that $M^{\alpha}$ ($\alpha\in [-1,0]\cap \QQ$) is zero if $\alpha$ is not $\alpha_{1},\dots,\alpha_{l}$.
Then, by Proposition~\ref{isomprop}, the direct sum of the above morphisms
\[\bigoplus_{i=1}^{l}(M^{\alpha_{i}}\boxtimes D_{\CC}/D_{\CC}(t\partial_{t}-\alpha_{i})^k)\to M\]
is a surjection.
Since any $D_{\CC}/D_{\CC}(t\partial_{t}-\alpha_{i})^k$ is regular and
any quotient object of a regular $D$-module is also regular,
$M$ is regular.
\end{proof}

\section{{T}he Hodge filtrations of monodromic mixed Hodge modules}
\subsection{The decompositions of Hodge filtrations}
In this section, we consider algebraic mixed Hodge modules whose underlying $D$-modules are monodromic.
For the definition and the properties of mixed Hodge module,
see \cite{MHM}, \cite{HM88}, \cite{MHMProj}, \cite{MorihikoYoung}, \cite{Schnell}, \cite{HoSiD}.
In this paper, mixed Hodge modules are always graded polarizable.
Let $X$ be a smooth algebraic variety over $\CC$.
Let $\calM=(M,F_{\bullet}M, K, W_{\bullet}K)$ be an algebraic mixed Hodge module on $X\times \CC_{t}$,
where $M$ is its underlying $D_{X\times \CC_{t}}$-module,
$F_{\bullet}M$ is its Hodge filtration (note that each $F_{p}M$ ($p\in \ZZ$) is a coherent $O_{X\times \CC}$-module),
$K$ is a $\QQ$-perverse sheaf on $X\times \CC_{t}$ 
with an isomorphisms
$\mathrm{DR}_{X\times \CC}(M)\simeq K_{\CC}(:=K\otimes_{\QQ} \CC)$,
and its filtration $W_{\bullet}K$.

\begin{defi}
We say that $\calM$ is a monodromic mixed Hodge module
if the underlying $D$-module $M$ is monodromic. 
\end{defi}

In the same way, we also define monodromic mixed Hodge module on $X\times \CS$. 
If $(M,F_{\bullet}M, K, W_{\bullet}K)$ is monodromic,
by Proposition~\ref{bunkailem} the $D$-module $M$ is of the form:
\[M=\bigoplus_{\beta\in \QQ}M^{\beta}.\]
Since $M$ is the underlying $D$-module of a mixed Hodge module,
$M$ is a regular holonomic $D$-module on $X\times \CC$.
Therefore, by Propositions~\ref{cohprop}, \ref{holprop}, \ref{regprop} and \ref{isomprop},
each $M^{\beta}$ $(\beta\in \QQ)$ is a regular holonomic $D$-module on $X$.
We remark that each $W_{k}M$ ($k\in \ZZ$) is itself $D_{X\times \CC}$-module.
Since any $D_{X\times \CC_{t}}$-submodule of a monodromic $D_{X\times \CC_{t}}$-module is monodromic,
$W_{k}M$ is also monodromic and has a decomposition
\[W_{k}M=\bigoplus_{\beta\in \QQ}W_{k}M^{\beta},\]
where we set $W_{k}M^{\beta}=W_{k}M\cap M^{\beta}$.
Similarly, $\GR^{W}_{k}M$ is also monodromic and we have
\[\GR^{W}_{k}M(=W_{k}M/W_{k-1}M)=\bigoplus_{\beta\in \QQ}\GR^{W}_{k}M^{\beta}.\]

In this section, we will prove the following.
\begin{thm}\label{main1}
Let $\calM=(M,F_{\bullet}M, K,W_{\bullet}K)$ be a monodromic mixed Hodge module on $X\times \CC_{t}$.
Then, we have
\begin{align}\label{wakamono}
F_{p}M=\bigoplus_{\beta\in \QQ}F_{p}M^{\beta},
\end{align}
where we set $F_{p}M^{\beta}:=F_{p}M\cap M^{\beta}$ for $\beta\in \QQ$.
\end{thm}

Before we prove this theorem, we recall some facts for mixed Hodge modules.
One of the most important properties (or constraints) of mixed Hodge modules is the strict specializability.
We set $F_{p}\GR^{\beta}_{V}M=F_{p}V^{\beta}_{t}M/F_{p}V^{>\beta}M$.

\begin{defi}(loc. cit.)\label{stsp}
Let $(M,F_{\bullet}M)$ be a holonomic $D$-module $M$ with its good filtration on $X\times \CC_{t}$.
Then, we say that $(M,F_{\bullet}M)$ is strictly $\KK$($=\QQ$ or $\RR$)-specializable along $t=0$
if $M$ is $\KK$-specializable and
for $p\in \ZZ$
\begin{enumerate}
\item[(i)] for any $\beta>-1$, $t\colon F_{p}\GR^{\beta}_{V}M\to F_{p}\GR^{\beta+1}_{V}M$
\mbox{\ is onto, and} 
\item[(ii)] for any $\beta<0$, $\partial_{t}\colon F_{p}\GR^{\beta}_{V}M\to F_{p+1}\GR^{\beta-1}_{V}M$\mbox{\ is onto}.
\end{enumerate}
\end{defi}

For the ``meaning'' of this condition, see
Section~3.2 of \cite{HM88} and Section~11 of \cite{Schnell}.
Any mixed Hodge module has such a property.
Therefore, if $F_{\bullet}M$ is decomposed as (\ref{wakamono}),
the filtrations $F_{\bullet}M^{\beta}$ for $\beta\notin [-1,0]\cap \QQ$ are determined by
$F_{\bullet}M^{\alpha}$ for $\alpha\in [-1,0]\cap \QQ$ as follows.

\begin{lem}\label{decomplem}
Assume that Theorem~\ref{main1} holds, i.e. we have a decomposition (\ref{wakamono}).
Then, for any $l\in \ZZ_{\geq 0}$ and $p\in \ZZ$ we have
\begin{align}\label{park1}
F_{p}M^{\alpha+l}&=t^{l}F_{p}M^{\alpha}\quad (\alpha\in (-1,0]\cap \QQ), \quad \mbox{and}\\\label{park2}
F_{p}M^{\alpha-l}&=\partial_{t}^{l}F_{p-l}M^{\alpha} \quad (\alpha\in [-1,0)\cap \QQ).
\end{align}
\end{lem}

\begin{proof}
Since the problem is local, 
we may assume that $X$ is affine.

Assume that we have the decomposition (\ref{wakamono}).
By Proposition~\ref{Vprop},
for $\gamma\in \QQ$ we have
\[F_{p}V^{\gamma}_{t}M=\bigoplus_{\beta\geq \gamma}F_{p}M^{\beta}.\]
Therefore, we have
\[F_{p}\GR^{\gamma}_{V}M=F_{p}M^{\gamma}.\]
The desired assertion follows from the conditions (i) and (ii) in the definition of the strict specializability (Definition~\ref{stsp}).
\end{proof}

\begin{rem}\label{nao}
Let us consider the nearby (resp. vanishing) cycle $\psi_{t}(\calM)$ (resp. $\phi_{t,1}(\calM)$) of a mixed Hodge module $\calM$, whose underlying perverse sheaf is ${}^p\psi_{t}K$ (resp. ${}^p\phi_{t,1}K$) (see loc. cit.).
The Hodge filtration of $\psi_{t}M$ (resp. $\phi_{t,1}M$) is
\begin{align*}
F_{\bullet}\psi_{t}M=\bigoplus_{\alpha\in (-1,0]\cap \QQ}F_{\bullet}&\psi_{t,e(\alpha)}M=\bigoplus_{\alpha\in (-1,0]\cap \QQ}F_{\bullet}\GR_{V}^{\alpha}M\\
(\mbox{resp.\ } F_{\bullet}\phi_{t,1}M&=F_{\bullet+1}\GR^{-1}_{V}M).
\end{align*}
The weight filtration of ${}^p\psi_{t}K$ (resp. ${}^p\phi_{t,1}K$) is the monodromy weight filtration relative to ${}^p\psi_{t}W_{\bullet+1}K$ (resp. ${}^p\phi_{t,1}W_{\bullet}K$).
See Subsection~\ref{gohho} for more details on these filtrations.
Assume that $X$ is affine and $\calM$ is monodromic. 
Then, for $\alpha\in (-1,0]\cap \QQ$, since we have $\psi_{t,e(\alpha)}M\simeq M^{\alpha}$,
we can regard $F_{\bullet}M^{\alpha}$ as the Hodge filtration of the nearby cycle.
Similarly, $F_{\bullet+1}M^{-1}$ is the Hodge filtration of the vanishing cycle.
In general, these data do not recover the original filtration of $M$.
However, Theorem~\ref{main1} implies that the Hodge filtration of $M$ is determined by the Hodge filtration of nearby and vanishing cycle of $M$ if $M$ is monodromic.
\end{rem}

\subsection{The pure case}

We show Theorem~\ref{main1} in the pure case.
First, we give a summary of the proof.
We will show that an analytic variation of pure Hodge structure on $X\times \CS$ has a simple form by using Deligne~\cite{DelUnTh} (Lemma~\ref{ganon}).
Taking a compactification of the base space, we will deduce a similar fact for an algebraic variation of pure Hodge structure (Lemma~\ref{rettugo}).
By using it and the strict specializability, we show Theorem~\ref{main1} for a smooth monodromic pure Hodge module (Corollary~\ref{kurumi} and Corollary~\ref{mouikkai}).
After that, by using the strict support decomposition of a pure Hodge module,
we will prove Theorem~\ref{main1} for a monodromic pure Hodge module (Proposition~\ref{titanic}).

\begin{rem}\label{CHodge}
In this section, we only focus on the underlying $D$-module with its filtrations, not consider its $\QQ$-structure.
Therefore, we can obtain similar results in this section for $\CC$-Hodge modules. 
Here,
in this paper,
a pure $\CC$-Hodge module means a pair $(M,F_{\bullet}M)$ of a $D$-module $M$ and its good filtration $F_{\bullet}M$ which is a direct summand of the underlying $D$-module with the Hodge filtration of a pure Hodge module.
For a pure Hodge module $\calM=(M,F_{\bullet}M, K)$,
we call the pair $\calM^{\CC}:=(M,F_{\bullet}M)$ the underlying $\CC$-Hodge module of it.
Note that in \cite{MHMProj} (unpublished),
``$\CC$-Hodge module'' is defined in more sophisticated way without $\QQ$-structures.
But, for our purposes, the above definition is enough.
\end{rem}

\begin{rem}\label{smHodge}
Let $(\calV, F_{\bullet}\calV, \calL)$ be a polarizable variation of pure Hodge structure on a smooth algebraic variety $Z$, where $\calV$ is the underlying integrable connection, $F_{\bullet}\calV$ is its Hodge filtration and $\calL$ is the underlying local system.
Recall that $(\calV, F_{\bullet}\calV, \calL[\dim{Z}])$ is a pure Hodge module.
The converse is true;
if a pure Hodge module $\calM=(M,F_{\bullet}M,K)$ on $Z$ is smooth, i.e. $M$ is a locally free $O_{Z}$-module of finite rank, $(M,F_{\bullet}M, H^{-\dim{Z}}K)$ is a polarizable variation of pure Hodge structure.
In this way, we will identify smooth pure Hodge modules with polarizable variations of pure Hodge structure. 
\end{rem}

Let $\calM=(M,F_{\bullet}M,K)$ be a monodromic pure Hodge module on $X\times \CC$.
To show Theorem~\ref{main1}, we may assume that $X$ is affine.
First, we assume the following condition:
\begin{align}\label{doredake}
\mathrm{SS}(M)\subset (T^{*}X\times T^{*}_{0}\CC)\cup (T^{*}_{X}X\times T^{*}_{\CC}\CC).
\end{align}
In particular, $M|_{X\times \CS}$ is smooth.
We denote by $p_{1}$ (resp. $p_{2}$) the first (resp. second) projection of $X\times \CS$.
Moreover, for a Hodge module $\calM$, we denote by $\calM^{\an}$ the analytification of it.


\begin{lem}\label{ganon}
Assume that $\calM$ is pure and satisfies (\ref{doredake}).
Then,
$(\calM|_{X\times \CS}^{\an})^\CC$ is a direct sum of some polarizable variations of pure $\CC$-Hodge structure of the form: $\calV_{1}\boxtimes \calV_{2}(=p_{1}^{*}\calV_{1}\otimes p_{2}^{*}\calV_{2})$,
where $\calV_{1}$ is a polarizable variation of pure $\CC$-Hodge structure on $X$ and $\calV_{2}$ is a polarizable variation of pure $\CC$-Hodge structure on $\CS$ of rank $1$.
\end{lem}

\begin{proof}
By Proposition~1.13 of Deligne~\cite{DelUnTh},
a polarizable pure variation of $\CC$-Hodge structure on smooth quasi-projective variety is of the form $\bigoplus_{l}\calN_{l}\otimes H_{l}$,
where $\calN_{l}$ is a polarizable variation of $\CC$-Hodge structure whose underlying local system is simple,
and $H_{l}$ is a pure $\CC$-Hodge structure (we regard it as a constant variation of Hodge structure).
We consider such a decomposition of $(\calM^{\an}|_{X\times \CS})^\CC$.
Recall that a local system is a finite dimensional representation of the fundamental group.
Therefore, since we have $\pi_{1}(X\times \CS)=\pi_{1}(X)\times \pi_{1}(\CS)$,
the underlying local system of $\calN_{l}$ is of the form:
$\calL_{1}\boxtimes \calL_{2}$, where $\calL_{1}$ (resp. $\calL_{2}$) is a local system on $X$ (resp. $\CS$).
Since the underlying local system of $\calN_{l}$ is simple, $\calL_{2}$ is of rank $1$.
Let $\calA_{l}$ be the rank $1$ pure variation of $\CC$-Hodge structure of the weight $0$ whose underlying local system is $\calL_{2}$.
Let $i\colon X\hookrightarrow X\times \CS\ (x\mapsto (x,1))$ be the inclusion map.
Then, the underlying local system of the polarizable variation of pure Hodge structure $i^{*}\calN_{l}\boxtimes \calA_{l}$ is $\calL_{1}\boxtimes \calL_{2}$.
Moreover, its restriction to a point $(x,1)\in X\times \CS$ is the same as the restriction of $\calN_{l}$ to the point $(x,1)$ (as a polarizable pure Hodge structure).
By the rigidity property of variation of Hodge structure (see Proposition~7.12 on p.124 of \cite{CompMFD}),
$\calN_{l}$ is isomorphic to $i^{*}\calN_{l}\boxtimes \calA_{l}$.
Moreover, it is clear that $(i^{*}\calN_{l}\boxtimes \calA_{l})\otimes H_{l}=(i^{*}\calN_{l}\otimes H_{l})\boxtimes \calA_{l}$.
This completes the proof.
\end{proof}

Lemma~\ref{ganon} holds for not only $(\calM^{\an}|_{X\times \CS})^{\CC}$ but also $(\calM|_{X\times \CS})^{\CC}$, as follows.

\begin{lem}\label{rettugo}
In the situation as above,
$(\calM|_{X\times \CS})^{\CC}$ is isomorphic (as an algebraic $\CC$-Hodge module) to a direct sum of some Hodge modules of the form $((j_{X})_{!*}\calV_{1})|_{X}\boxtimes ((j_{\CS})_{!*}\calV_{2})|_{\CS}$,
where $\calV_{1}$ is a polarizable variation of pure $\CC$-Hodge structure on $X$
and $\calV_{2}$ is a polarizable variation of pure $\CC$-Hodge structure on $\CS$ of rank $1$.
\end{lem}

\begin{proof}
Let $\calV_{1}$ (resp. $\calV_{2}$) be a polarizable variation of pure $\CC$-Hodge structure on $X$ (resp. $\CS$).
We regard them as pure Hodge modules.
Let $V_{i}$ be the underlying $D$-module of $\calV_{i}$.
Take a smooth compactification $\ov{X}$ of $X$ such that $\ov{X}\setminus X$ is a normal crossing divisor.
We denote by $j_{X}$ (resp. $j_{\CS}$) the inclusion $X\hookrightarrow \ov{X}$ (resp. $\CS\hookrightarrow \PP$).
We consider 
the minimal extensions $(j_{X})_{!*}\calV_{1}$ and $(j_{\CS})_{!*}\calV_{2}$
and the exterior product of them:
$(j_{X})_{!*}\calV_{1}\boxtimes (j_{\CS})_{!*}\calV_{2}$.
Since $\ov{X}\times \PP$ is a projective variety,
we can regard them as algebraic Hodge modules.
Then, the analytification of $((j_{X})_{!*}\calV_{i})|_{X}$ is $\calV_{i}$ for $i=1,2$.
By the exactness of $\boxtimes$ in the category of Hodge modules,
the image of the morphism
\[
(j_{X}\times 1_{\CC})_{!}(\calV_{1}\boxtimes \calV_{2})(\simeq (j_{X})_{!}\calV_{1}\boxtimes \calV_{2})\to 
(j_{X}\times 1_{\CC})_{*}(\calV_{1}\boxtimes \calV_{2})(\simeq (j_{X})_{*}\calV_{1}\boxtimes \calV_{2})
\]
is $(j_{X})_{!*}\calV_{1}\boxtimes \calV_{2}$.
Therefore, we have
\[(j_{X}\times 1_{\CC})_{!*}(\calV_{1}\boxtimes \calV_{2})\simeq (j_{X})_{!*}\calV_{1}\boxtimes \calV_{2}.\]
In the same way, we obtain
\[(1_{\ov{X}}\times j_{\CS})_{!*}((j_{X})_{!*}\calV_{1}\boxtimes \calV_{2})\simeq (j_{X})_{!*}\calV_{1}\boxtimes (j_{\CS})_{!*}\calV_{2}.\]
Since $(j_{X}\times j_{\CS})_{!*}\simeq (1_{\ov{X}}\times j_{\CS})_{!*}\circ (j_{X}\times 1_{\CS})_{!*}$,
we have a natural isomorphism 
\[(j_{X}\times j_{\CS})_{!*}(\calV_{1}\boxtimes \calV_{2})\simeq (j_{X})_{!*}\calV_{1}\boxtimes (j_{\CS})_{!*}\calV_{2}.\]
Note that the analytification of an algebraic Hodge module
$(j_{X}\times j_{\CS})_{!*}(\calM|_{X\times \CS})$ is
$(j_{X}\times j_{\CS})_{!*}(\calM|_{X\times \CS}^{\an})$.
Therefore, by Lemma~\ref{ganon},
$(j_{X}\times j_{\CS})_{!*}(\calM|_{X\times \CS})^{\CC}$
is isomorphic to
a direct sum of
Hodge modules of the form:
$(j_{X})_{!*}\calV_{1}\boxtimes (j_{\CS})_{!*}\calV_{2}$,
where $\calV_{1}$ (resp. $\calV_{2}$) is a polarizable variation of $\CC$-Hodge structure on $X$ (resp. $\CS$)
and $\calV_{2}$ is of rank $1$.
This implies the desired result.
\end{proof}

\begin{cor}\label{kurumi}
For a monodromic pure Hodge module $\calM=(M,F_{\bullet}M,K)$ with the condition (\ref{doredake}),
for each $\alpha\in (-1,0]\cap \QQ$
there exists a filtration $F_{\bullet}M^{\alpha}$ (in fact, $F_{p}M^{\alpha}=F_{p}M\cap M^{\alpha}$) of $M^{\alpha}$
such that we have
\[F_{\bullet}M|_{X\times \CS}=\bigoplus_{\alpha\in (-1,0]\cap \QQ}F_{\bullet}M^{\alpha}\otimes \CC[t^{\pm}].\]
Moreover, 
$M^{\alpha}$ ($\alpha\in (-1,0]\cap \QQ$) is killed by $t\partial_{t}-\alpha$, i.e. we have
\[M^{\alpha}=\mathrm{Ker}(t\partial_{t}-\alpha)\subset M.\]
\end{cor}

\begin{proof}
We remark again that the analytification of the algebraic $\CC$-Hodge module
$((j_{X})_{!*}\calV_{1})|_{X}$ (resp. $((j_{\CS})_{!*}\calV_{2})|_{\CS}$) is $\calV_{1}$ (resp. $\calV_{2}$).
Since $\calV_{2}$ is of rank $1$,
the Hodge filtration of $(j_{\CS})_{!*}\calV_{2}$ has only one jump, say $r\in \ZZ$.
Therefore, the Hodge filtration of $(j_{X})_{!*}\calV_{1}\boxtimes (j_{\CS})_{!*}\calV_{2}$ is of the form:
$(F_{\bullet-r}(j_{X})_{!*}\calV_{1})\boxtimes (j_{\CS})_{!*}\calV_{2}$.
Moreover, 
let $\alpha\in (-1,0]\cap \QQ$ be the integer such that the monodromy of the underlying local system of $\calV_{2}$ is $\exp(-2\pi\sqrt{-1}\alpha)$.
Then,
we have a decomposition
$((j_{\CS})_{!*}\calV_{2})|_{\CS}=\bigoplus_{l\in \ZZ}(((j_{\CS})_{!*}\calV_{2})|_{\CS})^{\alpha+l}$,
where $(((j_{\CS})_{!*}\calV_{2})|_{\CS})^{\alpha+l}$ is a vector space defined as
$(((j_{\CS})_{!*}\calV_{2})|_{\CS})^{\alpha+l}:=\bigcup_{s\geq 0}\mathrm{Ker}((t\partial_{t}-\alpha-l)^s\colon ((j_{\CS})_{!*}\calV_{2})|_{\CS}\to ((j_{\CS})_{!*}\calV_{2})|_{\CS})$.
Note that since $\calV_{2}$ is of rank $1$, we have
\begin{align}\label{buiki}
(((j_{\CS})_{!*}\calV_{2})|_{\CS})^{\alpha+l}=\mathrm{Ker}((t\partial_{t}-\alpha-l)\colon ((j_{\CS})_{!*}\calV_{2})|_{\CS}\to ((j_{\CS})_{!*}\calV_{2})|_{\CS}).
\end{align}
We define $(((j_{X})_{!*}\calV_{1})|_{X}\boxtimes ((j_{\CS})_{!*}\calV_{2})|_{\CS})^{\alpha+l}\subset ((j_{X})_{!*}\calV_{1})|_{X}\boxtimes ((j_{\CS})_{!*}\calV_{2})|_{\CS}$ similarly.
Then, we have
\[
(((j_{X})_{!*}\calV_{1})|_{X}\boxtimes ((j_{\CS})_{!*}\calV_{2})|_{\CS})^{\alpha+l}
=p_{1}^{*}(((j_{X})_{!*}\calV_{1})|_{X})\otimes p_{2}^{*}(((j_{\CS})_{!*}\calV_{2})|_{\CS})^{\alpha+l}.
\]
Then, we obtain
\[
((j_{X})_{!*}\calV_{1})|_{X}\boxtimes ((j_{\CS})_{!*}\calV_{2})|_{\CS}
=
\bigoplus_{l\in \ZZ} p_{1}^{*}(((j_{X})_{!*}\calV_{1})|_{X})\otimes_{\CC}(((j_{\CS})_{!*}\calV_{2})|_{\CS})^{\alpha+l}.
\]
In this expression, we have
\begin{align*}
F_{\bullet}(((j_{X})_{!*}\calV_{1})|_{X}\boxtimes ((j_{\CS})_{!*}\calV_{2})|_{\CS})
=&(F_{\bullet+r}((j_{X})_{!*}\calV_{1})|_{X})\boxtimes ((j_{\CS})_{!*}\calV_{2})|_{\CS}\\
=&\bigoplus_{l\in \ZZ} (F_{\bullet+r}((j_{X})_{!*}\calV_{1})|_{X})\otimes_{\CC} (((j_{\CS})_{!*}\calV_{2})|_{\CS})^{\alpha+l}.
\end{align*}
This implies the first statement.
The second assertion follows from (\ref{buiki})).
\end{proof}

Furthermore, we can show the following.
\begin{cor}\label{mouikkai}
For an monodromic algebraic pure Hodge module on $X\times \CC$ with the condition (\ref{doredake}),
we have a decomposition
\begin{align}\label{dedede}
F_{\bullet}M=\bigoplus_{\beta\in \QQ}F_{\bullet }M^{\beta},
\end{align}
where $F_{\bullet}M^{\beta}=F_{\bullet}M\cap M^{\beta}$.
Moreover, 
any $M^{\beta}$ ($\beta\in \QQ$) is killed by $t\partial_{t}-\beta$, i.e. we have
\[M^{\beta}=\mathrm{Ker}(t\partial_{t}-\beta)\subset M.\]
\end{cor}
\begin{proof}
We denote by $j$ the inclusion $X\times \CS\hookrightarrow X\times \CC$.
Then, by the strict specializability of $M$ (see Proposition~3.2.2 of \cite{HM88} or Exercise~11.1 of \cite{Schnell}), 
we have
\begin{align*}
F_{\bullet }M\cap V^{>-1}_{t}M = j_{*}(F_{\bullet}M|_{X\times \CS})\cap V^{>-1}_{t}M,
\end{align*}
where $V^{\bullet}_{t}M$ is the Kashiwara-Malgrange filtration along $X\times \{0\}$ of $M$.
Therefore, by Lemma~\ref{kurumi}:
$F_{\bullet}M|_{X\times \CS}=\bigoplus_{\alpha\in (-1,0]\cap \QQ}F_{\bullet}M^{\alpha}\otimes \CC[t^{\pm}]$,
we have
\begin{align}\label{nee}
F_{\bullet }M\cap V^{>-1}_{t}M= 
 \bigoplus_{l\in \ZZ_{\geq 0}}\bigoplus_{\alpha\in (-1,0]\cap \QQ}F_{\bullet}M^{\alpha+l}.
\end{align}
For a section $m\in F_{p}M$ on $X\times \CC$ let $m=\sum_{\beta\in \QQ}m^{\beta}\ (m^{\beta}\in M^{\beta})$ be its decomposition with respect to the decomposition $M=\bigoplus_{\beta\in \QQ}M^{\beta}$.

Assume that $m$ is in $F_{p}M\cap V^{\geq -1}_{t}M$.
Then, $tm=\sum_{\beta}tm^{\beta}$ is in $F_{p}M\cap V^{>-1}_{t}M$.
By (\ref{nee}), each component $tm^{\beta}$ is also in $F_{p}M$.
Since $t\colon F_{p}M\cap V^{>-1}_{t}M\to F_{p}M\cap V^{>0}_{t}M$ is isomorphism by the strict specializability,
for $\beta>-1$ we have $m^{\beta}\in F_{p}M$.
Hence, $m^{-1}=m-\sum_{\beta\in \QQ}m^{\beta}$ is also in $F_{p}M$.

To show in the general case, we use induction.
We remark that for any $\beta\in \RR$ and sufficiently small $\epsilon>0$ we have
$V^{\geq \beta}M=V^{>\beta-\epsilon}M$.
Assume that for fixed $\beta_{0}\in \RR_{>1}$, for a section in $F_{p}M\cap V^{>-\beta_{0}}M$ each component of it is also in $F_{p}M$.
Let $m$ be a section in $F_{p}M\cap V^{\geq -\beta_{0}}M$.
By the strict specializability, $\partial_{t}\colon F_{p-1}\GR^{-\beta_{0}+1}_{V}M\to F_{p}\GR^{-\beta_{0}}_{V}M$ is surjective.
Therefore, there exists a section $m'\in F_{p-1}M\cap V^{\geq -\beta_{0}+1}_{t}M$ such that $m-\partial_{t} m'$ is in $F_{p}M\cap V^{>-\beta_{0}}_{t}M$.
By the inductive assumption, all the components of $m'$ and $m-\partial_{t}m'$ are also in $F_{p}M$.
Hence, each component of $m$ is also in $F_{p}M$.
This completes the proof of the decomposition (\ref{dedede}).

It remains to show the second assertion for $\beta=-1$.
For a section $m\in M^{-1}$,
we have $(t\partial_{t}+1)m=\partial_{t}tm$.
Since $tm$ is in $M^{0}$,
$t\partial_{t}(tm)=0$ by Corollary~\ref{kurumi}.
Because $t\colon M^{0}\to M^{1}$ is an isomorphism,
$\partial_{t}tm$ is also zero.
This implies the second assertion.
\end{proof}

Next, we consider a monodromic pure Hodge module without the condition (\ref{doredake}).
We will show the following.

\begin{prop}\label{titanic}
For a monodromic pure Hodge module $\calM=(M,F_{\bullet}M,K)$ on $X\times \CC$,
we have a decomposition
\[F_{\bullet}M=\bigoplus_{\beta\in \QQ}F_{\bullet}M^{\beta},\]
where we set $F_{\bullet}M^{\beta}:=F_{\bullet}M\cap M^{\beta}$.
Moreover, any section in $M^{\beta}$ is killed by $t\partial_{t}-\beta$,
i.e. we have
\[M^{\beta}=\mathrm{Ker}(t\partial_{t}-\beta)\subset M.\]
\end{prop}

We need the following lemma.
\begin{lem}\label{pupupu}
Let $Y$ be a smooth closed subvariety of $X$ and $\mathcal{N}=(N,F_{\bullet}N,K_{N},W_{\bullet}K_{N})$ a monodromic mixed Hodge module on $Y\times \CC$.
We denote by $i\colon Y\hookrightarrow X$ (resp. $\wt{i}\colon Y\times \CC\hookrightarrow X\times \CC$) the inclusion of $Y$ (resp. $Y\times \CC$).
Then, the pushforward $\dpush \wt{i}_{*}N$ of $N$ as $D$-module by $\wt{i}$ is mondoromic and we obtain a decomposition
\begin{align}\label{serori}
\dpush \wt{i}_{*}N=\bigoplus_{\beta\in \QQ}(\dpush \wt{i}_{*}N)^{\beta}.
\end{align}
Moreover, for $\beta\in \QQ$ we have
\[(\dpush \wt{i}_{*}N)^{\beta}\simto \dpush {i}_{*}(N^{\beta}).\]
If the Hodge filtration of $N$ is decomposed with respect to the decomposition $N=\bigoplus_{\beta\in \QQ}N^{\beta}$, then 
so is the filtration of $\dpush \wt{i}_{*}N$.
Conversely, 
if the Hodge filtration of $\dpush \wt{i}_{*}N$ is decomposed,
so is the filtration of $N$.
\end{lem}

\begin{proof}
Set $l:=\mathrm{codim}_{X}Y$.
Take a local (\'etale) chart $(x_{1},\dots, x_{n})$ of $X$ such that $Y$ is defined by $x_{1}=\dots =x_{l}=0$.
We write $\partial_{i}:=\partial_{x_{i}}$.
Then, $\dpush \wt{i}_{*}N$ is expressed as 
\[\dpush \wt{i}_{*}N=\wt{i}_{*}N\otimes_{\CC}\CC[\partial_{1},\dots ,\partial_{l}],\]
and its section is expressed as a sum of 
some sections of the form $s\otimes \partial_{1}^{k_{1}}\dots \partial_{l}^{k_{l}}\delta$ ($s\in N$ and $k_{1},\dots,k_{l}\in \ZZ_{\geq 0}$).
Since $t\partial_{t}(s\otimes \partial_{1}^{k_{1}}\dots \partial_{l}^{k_{l}}\delta)=(t\partial_{t}s)\otimes \partial_{1}^{k_{1}}\dots \partial_{l}^{k_{l}}\delta$,
$\dpush \wt{i}_{*}N$ is monodromic and
we have
\[(\dpush \wt{i}_{*}N)^{\beta}=\dpush i_{*}N^{\beta}.\]

We assume that $F_{\bullet}N$ is decomposed as
$F_{\bullet}N=\bigoplus_{\beta\in \QQ}F_{\bullet}N^{\beta}$.
By the definition of the Hodge filtration of $\dpush \wt{i}_{*}N$,
we have
\begin{align*}
F_{p}(\dpush \wt{i}_{*}N)&=\sum_{k_{1}+\dots+k_{l}+q=p}
i_{*}(F_{q-l}N)\otimes \partial_{1}^{k_{1}}\dots \partial_{l}^{k_{l}}\delta\\
&=\bigoplus_{\beta\in \QQ}\sum_{k_{1}+\dots+k_{l}+q=p}
i_{*}(F_{q-l}N^{\beta})\otimes \partial_{1}^{k_{1}}\dots \partial_{l}^{k_{l}}\delta \quad \mbox{(by the assumption)}.
\end{align*}
Moreover, for $\beta\in \QQ$ we have
\[F_{p}(\dpush {i}_{*}N^{\beta})=\sum_{k_{1}+\dots+k_{l}+q=p}i_{*}(F_{q-l}N^{\beta})\otimes \partial_{t}^{k_{1}}\dots\partial_{t}^{k_{l}}\delta.\]
Therefore, we have
\begin{align*}
F_{p}(\dpush \wt{i}_{*}N)
=\bigoplus_{\beta\in \QQ}F_{p}(\dpush i_{*}N^{\beta}).
\end{align*}
This implies that if we set $F_{p}(\dpush \wt{i}_{*}N)^{\beta}:=
F_{p}(\dpush \wt{i}_{*}N)\cap (\dpush \wt{i}_{*}N)^{\beta}$,
we have
\begin{align}\label{tume}
F_{p}(\dpush \wt{i}_{*}N)=\bigoplus_{\beta\in \QQ}F_{p}(\dpush \wt{i}_{*}N)^{\beta}.
\end{align}

Conversely, assume the decomposition (\ref{tume}).
For simplicity, we assume $\mathrm{codim}_{X}(Y)=1$.
We can show it similarly in the general case.
Then we have $\dpush \wt{i}_{*}N= \wt{i}_{*}N[\partial_{1}]$.
We denote by $V_{Y}^{\bullet}(\dpush \wt{i}_{*}N)$ the Kashiwara-Malgrange filtration of $\dpush \wt{i}_{*}N$ along $Y$.
For $\alpha\in (-1,0]\cap \QQ$, since $(\dpush \wt{i}_{*}N)^{\alpha}\simeq \psi_{t,e(\alpha)}(\dpush \wt{i}_{*}N)$ and $(\dpush \wt{i}_{*}N)^{-1}\simeq \phi_{t,1}(\dpush \wt{i}_{*}N)$ are again strictly $\RR$-specializable by the definition of mixed Hodge module, any $(\dpush \wt{i}_{*}N)^{\beta}$ is also strictly $\RR$-specializable.
So, we can also consider the Kashiwara-Malgrange filtration $V^{\bullet}_{Y}((\dpush \wt{i}_{*}N)^{\beta})$ of $(\dpush \wt{i}_{*}N)^{\beta}$ along $Y$.
In this situation, we have
\[V_{Y}^{\bullet}(\dpush \wt{i}_{*}N)=\bigoplus_{\beta\in \QQ}
V_{Y}^{\bullet}(\dpush \wt{i}_{*}N)^{\beta}.
\]
By the assumption, we have
\begin{align}
F_{p}V_{Y}^{\bullet}(\dpush \wt{i}_{*}N)=\bigoplus_{\beta\in \QQ}
F_{p}V_{Y}^{\bullet}(\dpush \wt{i}_{*}N)^{\beta},
\end{align}
where we set
$F_{p}V_{Y}^{\bullet}(\dpush \wt{i}_{*}N)^{\beta}:=
F_{p}(\dpush \wt{i}_{*}N)^{\beta}\cap V_{Y}^{\bullet}(\dpush \wt{i}_{*}N)^{\beta}
$.
We can recover the filtration of $F_{\bullet}N$ by the formula:
\[F_{p}N=(F_{p}\GR^{0}_{V_{Y}}\dpush \wt{i}_{*}N)|_{Y}.\]
Therefore, we have
\begin{align*}
F_{p}N&=\bigoplus_{\beta \in \QQ}(
F_{p}V_{Y}^{0}(\dpush \wt{i}_{*}N)^{\beta}/
F_{p}V_{Y}^{>0}(\dpush \wt{i}_{*}N)^{\beta})|_{Y}.
\end{align*}
This implies that $F_{p}N$ is decomposed with respect to the decomposition $N=\bigoplus_{\beta\in \QQ}N^{\beta}$.
\end{proof}

\begin{proof}[Proof of Proposition~\ref{titanic}]
We consider a monodromic pure Hodge module $\calM:=(M,F_{\bullet}M,K)$ on $X\times \CC$ (without any other condition).
Since a pure Hodge module is a direct sum of some pure Hodge modules with strict supports,
we may assume that $\calM$ has strict support.
Then, since $M$ is monodromic, the strict support of $\calM$ is of the form that $Z\times \{0\}$ or $Z\times \CC$ for some irreducible closed subvariety $Z\subset X$.
In the former case, $M$ satisfies (\ref{doredake}), and hence the assertion has been proved in Corollary~\ref{mouikkai}. 
So, we assume that the support of $\calM$ is $Z\times \CC$.
By Proposition~\ref{intlemkai}, there is a smooth open (in $Z$) subvariety $Z'\subset Z$ of $Z$ such that $M|_{Z'\times \CS}$ is a integrable connection.
By Corollary~\ref{mouikkai}, we have 
\begin{align}\label{kabi}
F_{\bullet}M|_{Z'\times \CC}=\bigoplus_{\beta\in \QQ}F_{\bullet}M^{\beta}|_{Z'\times \CC}.
\end{align}
Take a (local) regular function $g$ such that $Z\setminus Z\cap g^{-1}(0)=Z'$.
Set $X':=X\setminus g^{-1}(0)$.
We denote by $i\colon Z'\times \CC\hookrightarrow X' \times \CC$ the closed embedding of $Z'\times \CC$ to $X' \times \CC$.
Note that we have $\calM|_{X'\times \CC}=i_{*}(\calM|_{Z'\times \CC})$ as a Hodge module.
By (\ref{kabi}) and Lemma~\ref{pupupu},
the Hodge filtration of $M_{X'\times \CC}$ is decomposed, i.e. we have
\begin{align}\label{naniwa}
F_{p}M|_{X'\times \CC}= \bigoplus_{\beta\in \QQ}F_{p}M^{\beta}|_{X'\times \CC}.
\end{align}
We will extend it to the equality on $X\times \CC$.
Let 
$\iota_{g}$ (resp. $\wt{\iota}_{g}$) be the graph embedding $X\hookrightarrow X\times \CC_{s}\ (x\mapsto (x,g(x)))$ (resp. $X\times \CC_{t}\hookrightarrow X\times \CC_{s}\times \CC_{t}\ ((x,t)\mapsto (x, g(x),t))$),
where we denote by $s$ the coordinate of the first $\CC$ in $X\times \CC\times \CC$.
We consider the pushforward ${(\wt{\iota}_{g}})_{*}\calM$ of $\calM$ by $\wt{\iota_{g}}$ (as a Hodge module),
whose underlying $D$-module is $\dpush(\wt{\iota}_{g})_{*}M(=(\wt{\iota}_{g})_{*}M{[\partial_{s}]})$. 
By Lemma~\ref{pupupu},
${(\wt{\iota}_{g})}_{*}\calM$ is a monodromic pure Hodge module on $(X\times \CC_{s})\times \CC_{t}$ and we have (as $D_{X\times \CC_{s}}$-modules)
\[({\dpush(\wt{\iota}_{g}})_{*}M)^{\beta}={\dpush({\iota}_{g}})_{*}(M^{\beta}),\]
i.e. we can describe $\dpush({\wt{\iota}_{g}})_{*}M$ as
\begin{align}\label{tomoare}
\dpush({\wt{\iota}_{g}})_{*}M=\bigoplus_{\beta\in \QQ}{\dpush({\iota}_{g}})_{*}(M^{\beta}).
\end{align}
Recall that we have
\[F_{p}({\dpush(\wt{\iota}_{g}})_{*}M)=\sum_{q+l=p}(\wt{\iota}_{g})_{*}F_{q-1}M\otimes \partial_{s}^{l}\delta.\]
Therefore, by (\ref{naniwa}),
we have 
\begin{align*}
F_{p}({\dpush(\wt{\iota}_{g})}_{*}M)|_{X\times \CS_{s}\times \CC_{t}}
=\bigoplus_{\beta\in \QQ} \sum_{q+l=p}(\iota'_{g})_{*}(F_{q-1}M^{\beta}|_{X'\times \CS_{s}})\otimes \partial_{s}^{l}\delta,
\end{align*}
under the identification (\ref{tomoare}).
This means that the Hodge filtration of ${\dpush(\wt{\iota}_{g})}_{*}M$ is decomposed as
\[F_{p}({\dpush(\wt{\iota}_{g})}_{*}M)|_{X\times \CS_{s}\times \CC_{t}}=\bigoplus_{\beta\in \QQ}F_{p}({\dpush({\iota}_{g})}_{*}M^{\beta})|_{X\times \CS_{s}},\]
where we set $F_{p}{(\dpush({\iota}_{g})}_{*}M^{\beta})=F_{p}({\dpush(\wt{\iota}_{g})}_{*}M)\cap ({\dpush({\iota}_{g})}_{*}M^{\beta})$.
Since ${\dpush(\wt{\iota}_{g})}_{*}M$ (resp. ${\dpush({\iota}_{g})}_{*}M^{\beta}$) is strictly $\RR$-specializable,
we can consider the Kashiwara-Malgrange filtration $V_{s}^{\bullet}({\dpush(\wt{\iota}_{g})}_{*}M)$ (resp. $V^{\bullet}_{s}({\dpush({\iota}_{g})}_{*}M^{\beta})$) of ${\dpush(\wt{\iota}_{g})}_{*}M$ (resp. ${\dpush({\iota}_{g})}_{*}M^{\beta}$) along $s=0$.
Then, we have
\[V_{s}^{\bullet}({\dpush(\wt{\iota}_{g})}_{*}M)=\bigoplus_{\beta\in \QQ}V_{s}^{\bullet}({\dpush({\iota}_{g})}_{*}M^{\beta}).\]
We denote by $\wt{j}$ (resp. $j$) the inclusion $X\times \CS_{s}\times\CC_{t} \hookrightarrow X\times \CC_{s}\times \CC_{t}$ (resp. $X\times \CS_{s} \hookrightarrow X\times \CC_{s}$).
By the strict specializability,
we have
\[F_{p}V_{s}^{>-1}({\dpush(\wt{\iota}_{g})}_{*}M)= \wt{j}_{*}(F_{p}({\dpush(\wt{\iota}_{g})}_{*}M)|_{X\times \CS_{s}\times \CC_{t}})\cap V_{s}^{>-1}({\dpush(\wt{\iota}_{g})}_{*}M).\]
Therefore, we have
\[
F_{p}V_{s}^{>-1}({\dpush(\wt{\iota}_{g})}_{*}M)= 
\bigoplus_{\beta\in \QQ}j_{*}(F_{p}({\dpush({\iota}_{g})}_{*}M^{\beta})|_{X\times \CS_{s}})\cap V^{>-1}_{s}({\dpush({\iota}_{g})}_{*}M^{\beta}).\]
Since ${\dpush(\wt{\iota}_{g})}_{*}M$ has strict support $\wt{\iota}_{g}(Z)$,
we have
\begin{align*}
F_{p}({\dpush(\wt{\iota}_{g})}_{*}M)&= \sum_{q+l=p}\partial_{s}^{l}F_{q}V_{s}^{>-1}{\dpush(\wt{\iota}_{g})}_{*}M\\
&=
\bigoplus_{\beta\in \QQ}
\sum_{q+l=p}\partial_{s}^{l}(
j_{*}(F_{q}({\dpush({\iota}_{g})}_{*}M^{\beta})|_{X\times \CS_{s}})\cap V^{>-1}_{s}({\dpush({\iota}_{g})}_{*}M^{\beta})).
\end{align*}
This means that the Hodge filtration of ${\dpush(\wt{\iota}_{g})}_{*}M$ is decomposed with respect to the decomposition (\ref{tomoare}).
By Lemma~\ref{pupupu}, we conclude that so is the Hodge filtration of $M$.

Since the section of $M^{\beta}|_{X'\times \CS}$ is killed by $t\partial_{t}-\beta$ by Corollary~\ref{mouikkai} and $M^{\beta}$ has strict support $Z\times \CC$,
we have $M^{\beta}=\mathrm{Ker}(t\partial_{t}-\beta)$.
This completes the proof.

\end{proof}

\begin{rem}
Assume that $\calM$ is pure of weight $w$.
Then, the weight filtration of $\calM$ is defined as 
\[
W_{k}K=\left\{
\begin{array}{ll}
K& (k\geq w)\\
0& (k<w).
\end{array}
\right.
\]
Recall that the weight filtration of ${}^p\psi_{t}K$ is the monodromy weight filtration relative to ${}^p\psi_{t}W_{\bullet+1}K$.
The second assertion of Proposition~\ref{titanic} implies that
the monodromy automorphism of ${}^p\psi_{t}K$ is semisimple.
Hence, the weight filtration of ${}^p\psi_{t}K$ is just ${}^p\psi_{t}W_{\bullet+1}K$,
i.e. $\psi_{t}\calM$ is pure of weight $w-1$ (see Lemma~\ref{arashi}).
\end{rem}

\subsection{The mixed case}
Next, we consider mixed Hodge modules and show Theorem~\ref{main1}.
We give a summary of the proof.
First, for a smooth monodromic mixed Hodge module, we define a variation of mixed Hodge structure on $X\times \CS$ so that its underlying $D$-module is the original one and the Hodge filtration is decomposed like the formula in Theorem~\ref{main1} (Lemma~\ref{deaeta}).
Here, to define it, we will use the result in the previous section, i.e. Proposition~\ref{titanic}.
Then, by the rigidity property of a graded polarizable admissible variation of mixed Hodge structure, we can see the new one is equal to the original one (Lemma~\ref{kabutomusi}).
This leads to Theorem~\ref{main1} for a monodromic mixed Hodge module satisfying the condition~(\ref{doredake}) (Lemma~\ref{natumaturi}).
Finally, we will prove Theorem~\ref{main1} by induction on the dimension of the support and the strict specializability of a mixed Hodge module.

\begin{rem}\label{smmixedHodge}
Let $(\calV, F_{\bullet}\calV, \calL, W_{\bullet}\calL)$ be an admissible graded polarizable variation of mixed Hodge structure on a smooth algebraic variety $Z$.
Then, $(\calV, F_{\bullet}\calV, \calL[\dim{Z}],W_{\bullet-\dim{Z}}\calL[\dim{Z}])$ is a mixed Hodge module.
If a mixed Hodge module $\calM=(M,F_{\bullet}M,K, W_{\bullet}K)$ on $Z$ is smooth, i.e. $M$ is a locally free $O_{Z}$-module of finite rank, $(M,F_{\bullet}M, H^{-\dim{Z}}K, W_{\bullet+\dim{Z}}H^{-\dim{Z}}K)$ is an admissible graded polarizable variation of mixed Hodge structure.
In this way, we will identify smooth mixed Hodge modules with admissible graded polarizable variations of mixed Hodge structure. 
\end{rem}

Let $\calM=(M,F_{\bullet}M,K,W_{\bullet}K)$ be a monodromic mixed Hodge module on $X\times \CC$.
We remark that a sub mixed Hodge module $W_{k}\calM=(W_{k}M,F_{\bullet}W_{k}M,W_{k}K,W_{\bullet}K\cap W_{k}K)$ is again monodromic and
$\GR^{W}_{k}\calM$ is also.

We will prove the following.
\begin{lem}\label{natumaturi}
If $\calM$ satisfies the condition~(\ref{doredake}),
$F_{\bullet}M$ is decomposed with respect to the decomposition $M=\bigoplus_{\beta\in \QQ}M^{\beta}$.
\end{lem}

For a section $m\in M^{\alpha+l}$ ($l\in \ZZ, \alpha\in (-1,0]\cap \QQ$),
we set $\sigma(m):=t^{-l}m\in M^{\alpha}$.
We thus obtain a map $\sigma\colon M\to \bigoplus_{\alpha\in (-1,0]\cap \QQ}M^{\alpha}$.
We define a filtration $F'_{\bullet}M$ on $M^{\alpha}$ as
\[F'_{p}M^{\alpha}:=\sigma(F_{p}M)\cap M^{\alpha}.\]
Moreover, we define a filtration $F_{\bullet}'M|_{X\times \CS}$ on $M|_{X\times \CS}=\bigoplus_{\alpha\in (-1,0]\cap \QQ}M^{\alpha}\otimes \CC[t^{\pm}]$
as
\[F'_{\bullet}M|_{X\times \CS}:= \bigoplus_{k\in \ZZ}\bigoplus_{\alpha\in (-1,0]\cap\QQ}t^{k}F'_{\bullet}M^{\alpha}.\]

\begin{lem}\label{deaeta}
In this setting,
the tuple 
\[(M|_{X\times \CS}, F'_{\bullet}M|_{X\times \CS}, H^{-\dim{X}-1}K|_{X\times \CS}, W_{\bullet+\dim{X}+1}H^{-\dim{X}-1}K|_{X\times \CS})\] defines a graded polarizable admissible variation of mixed Hodge structure on $X\times \CS$.
\end{lem}

To show it, 
since $(M|_{X\times \CS}, F_{\bullet}M|_{X\times \CS}, H^{-\dim{X}-1}K|_{X\times \CS}, W_{\bullet+\dim{X}+1}H^{-\dim{X}-1}K|_{X\times \CS})$ is a graded polarizable admissible variation of mixed Hodge structure on $X\times \CS$,
it suffices to show 
\begin{align}\label{itunomanika}
F_{p}\GR^{W}_{k}M|_{X\times \CS}=F_{p}'\GR^{W}_{k}M|_{X\times \CS}.
\end{align}
Since $\GR^{W}_{k}M$ is a monodromic pure Hodge module on $X\times \CC$,
by Proposition~\ref{titanic} the filtration $F_{\bullet}\GR^{W}_{k}M$ is decomposed as
\[F_{p}\GR^{W}_{k}M=\bigoplus_{\beta\in \QQ}F_{p}\GR^{W}_{k}M\cap \GR^{W}_{k}M^{\beta}.\]
Therefore, to show (\ref{itunomanika}),
it suffices to show 
\begin{align}\label{ashitagaaru}
F_{p}\GR^{W}_{k}M|_{X\times \CS}\cap (\bigoplus_{\alpha\in (-1,0]\cap \QQ}\GR^{W}_{k}M^{\alpha})=F'_{p}\GR^{W}_{k}M|_{X\times \CS}\cap (\bigoplus_{\alpha\in (-1,0]\cap \QQ}\GR^{W}_{k}M^{\alpha}),
\end{align}
in $\bigoplus_{\alpha\in (-1,0]\cap \QQ}\GR^{W}_{k}M^{\alpha}$.

\begin{lem}\label{kyuu}
In the setting of Lemma~\ref{natumaturi}, for $x\in F_{p}M$,
if $\sigma(x)\in W_{k}M$,
there exists $y\in F_{p}M\cap W_{k}M$ such that
$[\sigma(x)]=[\sigma(y)]$ in $\bigoplus_{\alpha\in (-1,0]\cap \QQ}\GR^{W}_{k}M^{\alpha}$
and any $\beta(\in \QQ\setminus  (-1,0])$-component of $y$ is in $W_{k-1}M$.
\end{lem}
\begin{proof}
Let $x$ be a section in $F_{p}M$ such that $\sigma(x)\in W_{k}M$.
We may assume that $x$ is in $V^{>-1}_{t}M$ since we have $\sigma(t^{l}x)=\sigma(x)$ for any $l\in \ZZ$.
Assume that $x$ is in $W_{k'}M$ for some $k'\geq k$.
Let $x=\sum_{\beta>-1}x^{\beta}$ ($x^{\beta}\in M^{\beta}$) be the decomposition.
Since any component of any element in $F_{p}\GR^{W}_{k'}M$ is also in $F_{p}\GR^{W}_{k'}M$ by Proposition~\ref{titanic},
there exists $y_{\beta}\in F_{p}M\cap W_{k'}M$ such that
$x^{\beta}-y_{\beta}\in W_{k'-1}M$ (we remark that $y_{\beta}$ may not be in $M^{\beta}$ in general).
We take a positive integer $l$ such that $\beta=\alpha+l$ for some $\alpha\in (-1,0]\cap \QQ$.
Note that $[y_{\beta}](\in \GR^{W}_{k'}M)$ is in $V^{l-1}_{t}\GR^{W}_{k'}M$.
Then, by the strict specializability,
there exist $y'_{\beta}\in F_{p}M\cap W_{k'}M$ such that $[t^{l}y'_{\beta}]=[y_{\beta}]$ in $\GR^{W}_{k'}M$.
We set $y:=\sum_{\beta}y_{\beta}$ and $y':=\sum_{\beta}y'_{\beta}$.
Then, $x-y(:=z)$ is in $F_{p}M\cap W_{k'-1}M$ and $y'$ is in $F_{p}M\cap W_{k'}M$.
Note that since $x^{\beta}-y_{\beta}$ is in $W_{k'-1}M$,
any $\gamma(\neq \beta)$-component of $y_{\beta}$ is in $W_{k'-1}M$.
Therefore, any $\gamma(\notin (-1,0])$-component of $y'_{\beta}$ is in $W_{k'-1}M$ and hence any $\gamma(\notin (-1,0])$-component of $y'$ is in $W_{k'-1}M$.
Then, if $k'=k$, we obtain the desired assertion since we have $[\sigma(x)]=[\sigma(y')]$ in $\GR^{W}_{k}M$.
If $k'>k$, we have $\sigma(x)\in W_{k}M\subset W_{k'-1}M$ and hence
$\sigma(y')(=\sigma(y)=\sigma(x)-\sigma(z))$ is in $W_{k'-1}M$.
Therefore, since any $\beta(\in \QQ\setminus (-1,0])$-component of $y'$ is in $W_{k'-1}M$,
each $\alpha(\in (-1,0]\cap \QQ)$-component of $y'$ is also in $W_{k'-1}M$.
Hence, $y'$ is in $W_{k'-1}M$ and in particular $y'+z$ is in $W_{k'-1}M$ such that $\sigma(y'+z)=\sigma(x)$.
Thus, we can use induction on $k'$ and we get the desired assertion.

\end{proof}

\begin{proof}[Proof of Lemma~\ref{deaeta}]
We first consider a section $[x]$ in the left hand side of (\ref{ashitagaaru}),
where $x$ is a section of $F_{p}M\cap W_{k}M|_{X\times \CS}$.
Let $x=\sum_{\beta}x^{\beta}$ be its decomposition with respect to the decomposition
$M|_{X\times \CS}=\bigoplus_{\beta\in \QQ}(M|_{X\times \CS})^{\beta}(=\bigoplus_{\alpha\in (-1,0]\cap \QQ}M^{\alpha}\otimes\CC[t^{\pm}])$.
Since $[x]$ is in $\bigoplus_{\alpha\in (-1,0]\cap \QQ}\GR^{W}_{k}M^{\alpha}$,
the section $x^{\beta}$ is in $W_{k-1}M|_{X\times \CS}$ for any $\beta \in \QQ\setminus (-1,0]$.
Since we have $(M|_{X\times \CS})^{\beta}=M^{\beta}$ for $\beta >-1$,
$t^{l}x$ is in $F_{p}M\cap W_{k}M$ (not only in $F_{p}M\cap W_{k}M|_{X\times \CS}$) for a sufficiently large integer $l\geq 0$.
Note that $t^{l}x^{\beta}$ is in $W_{k-1}M$ for $\beta \in \QQ\setminus (-1,0]$.
Therefore, by the definition of $\sigma$,
we have $[\sigma(t^{l}x)]=[x^{\beta}]= [x]$ in $\GR^{W}_{k}M$.
Hence, $[x]$ is in the right hand side of (\ref{ashitagaaru}).

Next, we consider a section $[x]$ of the right hand side of (\ref{ashitagaaru}),
where $x$ is a section of $F_{p}'M\cap W_{k}M|_{X\times \CS}$.
By the definition of $F_{p}'M$,
there are sections $x_{1},\dots,x_{j}\in F_{p}M$ and integers $l_{1},\dots, l_{j}\in \ZZ$
such that $\sigma(x_{i})\in M^{\alpha_{i}}$ for some $\alpha_{i}\in (-1,0]\cap \QQ$ and $x=\sum_{i=1}^{j}t^{l_{i}}\sigma(x_{i})$.
Since $[x]$ is in $\bigoplus_{\alpha\in (-1,0]\cap \QQ}\GR^{W}_{k}M^{\alpha}$,
by putting $y=\sum_{l_{i}=0}x_{i}$ we have
$[x]=[\sigma(y)]$ in $\bigoplus_{\alpha\in (-1,0]\cap \QQ}\GR^{W}_{k}M^{\alpha}$.
Note that $y$ is in $F_{p}M$ and $\sigma(y)$ is in $W_{k}M$ (the problem is that we do not know if $y$ is in $W_{k}M$ or not).
Then, by Lemma~\ref{kyuu},
there exists $z\in F_{p}M\cap W_{k}M$ such that
$[\sigma(y)]=[\sigma(z)]$ in $\bigoplus_{\alpha\in (-1,0]\cap \QQ}\GR^{W}_{k}M^{\alpha}$
and any $\beta(\in \QQ\setminus (-1,0])$-component of $z$ is in $W_{k-1}M$.
By using the latter property of $z$, we have $[\sigma(z)]=[z]$ in $\GR^{W}_{k}M$
and hence $[x]=[z]$.
Therefore, $[x]$ is in the left hand side of (\ref{ashitagaaru}).
\end{proof}


\begin{lem}\label{kabutomusi}
In the setting of \ref{natumaturi},
we have
\[F_{\bullet}M|_{X\times \CS}=F_{\bullet}'M|_{X\times \CS}.\]
\end{lem}
\begin{proof}
We can express the restriction of $M$ to $X\times \{1\}$ as
\[M|_{X\times \{1\}}=V_{t}^{>-1}M/(t-1)V_{t}^{>-1}M\simeq \bigoplus_{\alpha\in (-1,0]\cap \QQ}M^{\alpha}.\]
Under this identification, by the definition of $F_{p}'M^{\alpha}$ we have
\[F_{p}M|_{X\times \{1\}}=\bigoplus_{\alpha\in (-1,0]\cap \QQ}F'_{p}M^{\alpha}.\]
Hence, we have
\[F_{p}M|_{X\times \{1\}}=F_{p}'M|_{X\times \{1\}}.\]
By the rigidity property of a graded polarizable admissible variation of mixed Hodge structure (see Proposition~7.12 on p.124 of \cite{CompMFD}),
we conclude that 
two variations: 
\begin{align*}
(M,F_{\bullet}M,H^{-\dim{X}-1}K,W_{\bullet+\dim{X}+1}H^{-\dim{X}-1}K)|_{X\times \CS},\tunagi\\ 
(M|_{X\times \CS}, F'_{\bullet}M|_{X\times \CS}, H^{-\dim{X}-1}K|_{X\times \CS}, W_{\bullet+\dim{X}+1}H^{-\dim{X}-1}K|_{X\times \CS})
\end{align*}
are the same.
In particular, the Hodge filtrations $F_{\bullet}M|_{X\times \CS}$ and $F_{\bullet}'M|_{X\times \CS}$ are the same.
\end{proof}

\begin{proof}[Proof of Lemma~\ref{natumaturi}]
Let $j\colon X\times \CS\hookrightarrow X\times \CC$ is the inclusion.
By the strict specializability,
we have
\[F_{p}V_{t}^{>-1}M=j_{*}(F_{p}M|_{X\times \CS})\cap V_{t}^{>-1}M.\]
Therefore, by Lemma~\ref{kabutomusi},
$F_{p}V_{t}^{>-1}M$ is decomposed with respect to the decomposition $V_{t}^{>-1}M=\bigoplus_{\beta>-1}M^{\beta}$.
Then, in exactly the same way as in the proof of Corollary~\ref{mouikkai} (use the strict specializability again),
we conclude that $F_{p}M$ is decomposed with respect to the decomposition of $M$.
\end{proof}

\begin{proof}[Proof of Theorem~\ref{main1}]
Let $\calM=(M,F_{\bullet}M,K,W_{\bullet}K)$ be a monodromic mixed Hodge module on $X\times \CC$ (without any other assumption).
Since $M$ is monodromic,
the support of $M|_{X\times \CS}$ is of the form:
$\mathrm{supp}(M|_{X\times \CS})=Z\times \CS$ for some closed subvariety $Z$ of $X$.
We use induction on the dimension of $Z$.

\noindent\textit{Step~1.}\ 
Assume that $\dim{Z}=0$.
Then, $F_{\bullet}M|_{Z\times \CS}$ are decomposed with respect to the decomposition of $M$ by Lemma~\ref{natumaturi},
and hence so are $F_{\bullet}M|_{X\times \CS}$.
We can show that $F_{\bullet}M$ are decomposed in the same way as in the proof of Corollary~\ref{mouikkai}.

\noindent \textit{Step~2.}\ 
For a fixed positive integer $l\in \ZZ_{\geq 1}$,
assume that Theorem~\ref{main1} holds if $\dim(Z)$ is less than $l$.
We take a smooth open subvariety $Z'\subset Z$ of $Z$ such that
$\mathrm{Ch}(M|_{Z'\times \CS})$ is contained in $T^{*}_{Z'\times \CS}(Z'\times \CS)$.
Let $g$ be a (local) regular function on $X$ which defines $Z\setminus Z'$.
Set $X':=X\setminus g^{-1}(0)$.
By Lemma~\ref{natumaturi}, $F_{\bullet}M|_{Z'\times \CC_{t}}$ are decomposed with respect to the decomposition $M|_{Z'\times \CC_{t}}=\bigoplus_{\beta\in \QQ}M^{\beta}|_{Z'}$ and hence so are $F_{\bullet}M|_{X'\times \CC_{t}}$.
We denote by $\iota_{g}$ (resp. $\wt{\iota_{g}}$) be a graph embedding $X\hookrightarrow X\times \CC_{s}$ ($x\mapsto (x,g(x))$) (resp. $X\times \CC_{t}\hookrightarrow X\times \CC_{s}\times \CC_{t}$ ($(x,t)\mapsto (x,g(x),t)$)) of $g$,
where we denote by $s$ the coordinate of the first $\CC$ of $X\times \CC_{s}\times \CC_{t}$.
We set $N:=\dpush (\wt{\iota_{g}})_{*}M$ and $N^{\beta}:=\dpush ({\iota_{g}})_{*}M^{\beta}$.
We denote by $V_{s}^{\bullet}N$ (resp. $V_{s}^{\bullet}N^{\beta}$) the Kashiwara-Malgrange filtration of $N$ (resp. $N^{\beta}$) along $s=0$.
Note that we have
\begin{align}\label{amai}
V_{s}^{\bullet}N=\bigoplus_{\beta\in \QQ}V_{s}^{\bullet}N^{\beta}.
\end{align}
Since $F_{\bullet}M|_{X'\times \CC_{t}}$ are decomposed,
$F_{\bullet}N|_{X\times \CS_{s}\times \CC_{t}}$ are also decomposed with respect to the decomposition 
$N=\bigoplus_{\beta\in \QQ}N^{\beta}$ by Lemma~\ref{pupupu}.

\noindent \textit{Step~3.}\ 
Let $j$ be the open embedding $X\times \CS_{s}\times \CC_{t}\hookrightarrow X\times \CC_{s}\times \CC_{t}$.
By the strict specializability,
we have
\[F_{\bullet}V^{>-1}_{s}N(=F_{\bullet}N\cap V_{s}^{>-1}N)
=j_{*}(F_{\bullet}N|_{X\times \CS_{s}\times \CC_{t}})\cap V^{>-1}_{s}N.\]
Therefore, by Step~2,
$F_{\bullet}V^{>-1}_{s}N$
are decomposed with respect to the decomposition (\ref{amai}).

\noindent \textit{Step~4.}\ 
We will show that $F_{\bullet}V^{\geq 1}_{s}N$ is decomposed (completed in Step~7).
Take a section $x\in F_{p}V^{\geq 1}_{s}N$.
Let $x=\sum_{\beta}x^{\beta}$ ($x^{\beta}\in N^{\beta}$) be the decomposition of $x$.
Then, since $sx$ is in $F_{p}V^{>-1}_{s}N$,
all the components $sx^{\beta}$ ($\beta\in \QQ$) of $sx$ are also in $F_{p}N$ by Step~3.

\noindent \textit{Step~5.}\ 
Consider the vanishing cycle $\phi_{s,1}N$ of $N$ along $s=0$.
Then, we have
\begin{align}\label{newhysummer}
\phi_{s,1}N(=\GR^{-1}_{V_{s}}N)=\bigoplus_{\beta\in \QQ}\GR^{-1}_{V_{s}}N^{\beta}.
\end{align}
Therefore,
$\phi_{s,1}N$ is also a monodromic mixed Hodge module.
Note that the support of $\phi_{s,1}N|_{X\times \CC_{s}\times \CC_{t}^*}$ is contained in $(Z\setminus Z')\times \{0\}\times \CS_{t}$.
Therefore, by the inductive assumption (in Step~2), the Hodge filtration of $\phi_{s,1}N$ is decomposed with respect to the decomposition (\ref{newhysummer}).

\noindent \textit{Step~6.}\ 
By Step~5, all the components $[x^{\beta}](\in \GR^{-1}_{V_{s}}N^{\beta})$ of $[x]\in \phi_{s,1}N$
are in $F_{p}\phi_{s,1}N$.
This implies that there exists $z_{\beta}\in F_{p}V_{s}^{\geq -1}N$ (remark that $z_{\beta}$ may not be in $N^{\beta}$) such that $[x^{\beta}]=[z_{\beta}]$ in $\phi_{s,1}N$, i.e.
\begin{align}\label{raprap}
x^{\beta}-z_{\beta}\in V^{>-1}_{s}N.
\end{align}
Let $z_{\beta}=\sum_{\gamma\in \QQ}z_{\beta}^{\gamma}$ ($z_{\beta}^{\gamma}\in N^{\gamma}$) be the decomposition of $z_{\beta}$.
Then, by (\ref{raprap}), we have $z_{\beta}^{\gamma}\in V^{>-1}_{s}N$ for any $\gamma\neq \beta$.
On the other hand, since $sx^{\beta}$ is in $F_{p}N$ by Step~4,
$sx^{\beta}-sz_{\beta}$ is in $F_{p}V^{>-1}_{s}N$.
By Step~3, all the components of $sx^{\beta}-sz_{\beta}$ are in $F_{p}N$.
Therefore, $sz_{\beta}^{\gamma}$ is in $F_{p}N$ for any $\gamma\neq \beta$.
Since $z_{\beta}^{\gamma}$ is in $V^{>-1}_{s}N$ and $sz_{\beta}^{\gamma}$ is in $F_{p}N$  for $\gamma\neq \beta$,
$z_{\beta}^{\gamma}$ is also in $F_{p}N$ by the strict specializability.
Hence, $z_{\beta}^{\beta}=z_{\beta}-\sum_{\gamma\neq \beta}z_{\beta}^{\gamma}$ is also in $F_{p}N$.

\noindent \textit{Step~7.}\ 
Note that $x-\sum_{\beta}z_{\beta}$ is in $F_{p}V^{>-1}_{s}N$ and hence
all the components of it is also in $F_{p}N$ by Step~3.
Therefore, by Step~6, all the components of $x$ is also in $F_{p}N$.
This means that $F_{p}V^{\geq -1}_{s}N$ is decomposed with respect to the decomposition (\ref{amai}).

\noindent \textit{Step~8.}\ 
Since we have $F_{p}N=\sum_{i\geq 0}\partial_{s}^{i}F_{p-i}V^{\geq -1}_{s}N$ by the strict specializability
and $F_{p-i}V^{\geq -1}_{s}N$ are decomposed by the previous step,
we conclude that $F_{p}N$ is decomposed.
By Lemma~\ref{pupupu}, so is $F_{p}M$.
This completes the proof.
\end{proof}

\section{A construction of monodromic mixed Hodge modules}

\subsection{{S}ome properties of monodromic mixed Hodge modules} \label{gohho}
In the previous section, we saw that the Hodge filtration of a monodromic mixed Hodge module is decomposed.
As we saw in Remark~\ref{nao}, the data of the Hodge filtration of $M$ is the same as the data of the Hodge filtrations of the nearby and vanishing cycles, i.e. $F_{\bullet}M^{\alpha}$ ($\alpha\in [-1,0]\cap \QQ$).
Based on this observation, 
we will construct a monodromic mixed Hodge module on $X\times \CC_{t}$
from a mixed Hodge module on $X$ with a ``gluing data''.
As a preparation, we show some properties of the nearby and vanishing cycles of a monodromic mixed Hodge module in this subsection.

Let us recall the nearby and vanishing cycle functors (we have already explained them briefly in Remark~\ref{nao}).
Let $\calM=(M,F_{\bullet}M,K,W_{\bullet}K)$ be a monodromic mixed Hodge module on $X\times \CC_{t}$.
Then, $\GR^{W}_{k}\calM(=(\GR^{W}_{k}M,F_{\bullet}\GR^{W}_{k}M,\GR^{W}_{k}K))$ be a monodromic pure Hodge module of weight $k$.
By Corollary~\ref{mouikkai},
for any $\beta\in \QQ$, $(\GR^{W}_{k}M)^{\beta}=\GR^{W}_{k}M^{\beta}$ is killed by $t\partial_{t}-\beta$, where 
we set $W_{k}M^{\beta}=W_{k}M\cap M^{\beta}$.
Consider the nearby (resp. unipotent vanishing) cycle $\psi_{t}(\calM)$ (resp. $\phi_{t,1}(\calM)$) of $\calM$.
Recall that we have $\psi_{t}(M)=\bigoplus_{\alpha\in (-1,0]\cap\QQ} \GR^{\alpha}_{V}M(=\bigoplus_{\alpha\in (-1,0]\cap\QQ}M^{\alpha})$ and $\phi_{t,1}(M)=\GR^{-1}_{V}M(=M^{-1})$.
For $p\in \ZZ$ and $\alpha\in (-1,0]\cap \QQ$, we set 
\begin{align*}
F_{p}\psi_{t,e(\alpha)}(M)&:=F_{p}\GR^{\alpha}_{V}M,\\
F_{p}\psi_{t}(M)&:=\bigoplus_{\alpha\in (-1,0]}F_{p}\GR^{\alpha}_{V}M, \quad\mbox{and}\\
F_{p}\phi_{t,1}(M)&:=F_{p+1}\GR^{-1}_{V}M.
\end{align*}
We remark that $F_{p}\psi_{t}(M)$ is not $F_{p}M\cap V^{>-1}_{t}M/F_{p}M\cap V^{>0}_{t}M$ in general.
$\frac{-1}{2\pi\sqrt{-1}}$ times the logarithm of the unipotent part of the monodromy automorphism $\exp(-2\pi\sqrt{-1}t\partial_{t})$ of $\psi_{t}(M)$ or $\phi_{t,1}(M)$ is denoted by $N$.
Then, the nilpotent morphism $N$ is $t\partial_{t}-\alpha$ (resp. $t\partial_{t}+1$) on $\psi_{t,e(\alpha)}(M)=M^{\alpha}$ (resp. $\phi_{t,1}(M)=M^{-1}$) (see Corollary~\ref{nevancor}).
Moreover, $N$ defines a morphism in the category of mixed Hodge modules:
\begin{align*}
\psi_{t}(\calM)\to& \psi_{t}(\calM)(-1), \tunagi\\
\phi_{t,1}(\calM)\to& \phi_{t,1}(\calM)(-1).
\end{align*}
We define a filtration $L_{\bullet}\psi_{t}(M)$\ (resp. $L_{\bullet}\phi_{t,1}(M)$) of $\psi_{t}(M)$\ (resp. $\phi_{t,1}(M)$) by
\begin{align*}
L_{k}\psi_{t}(M)&:=\psi_{t}(W_{k+1}M)\\
(\mbox{resp.\ } L_{k}\phi_{t,1}(M)&:=\phi_{t,1}(W_{k}M)),
\end{align*}
for $k\in \ZZ$. 
Then,
there exists the monodromy filtration $W_{\bullet}\psi_{t}(M)$ of $N$ relative to $L_{\bullet}\psi_{t}(M)$,
and the monodromy filtration $W_{\bullet}\phi_{t,1}(M)$ of $N$ relative to $L_{\bullet}\phi_{t,1}(M)$ 
(this is one of the conditions for $\calM$ to be a mixed Hodge module).
Recall that the monodromy filtration $W_{\bullet}A$ of $N$ relative to $L_{\bullet}A$ ($A=\psi_{t}(M)$ or $\phi_{t,1}(M)$) is caracterized by the two properties:
\begin{enumerate}
\item[(i)] $N(W_{k}A)$ is contained in $W_{k-2}A$ for $k\in \ZZ$.
\item[(ii)] $N^{l}\colon \GR^{W}_{k+l}\GR^{L}_{k}A\to \GR^{W}_{k-l}\GR^{L}_{k}A$ is isomorphism for $k\in \ZZ$ and $l\in \ZZ_{\geq 0}$.
\end{enumerate}
Corresponding to it,
we can also define a filtration $L_{\bullet}{}^p\psi_{t}(K)$ (resp. $L_{\bullet}{}^p\phi_{t,1}(K)$)
and the relative monodromy filtration $W_{\bullet}{}^p\psi_{t}(K)$ (resp. $W_{\bullet}{}^p\phi_{t,1}(K)$) of $N$.
Then, the 4-tuples
\begin{align}
(\psi_{t}(M),& F_{\bullet}\psi_{t}(M), {}^p\psi_{t}(K), W_{\bullet}{}^p\psi_{t}(K)),\quad \mbox{and}\\
(\phi_{t,1}(M),& F_{\bullet}\phi_{t,1}(M),{}^p\phi_{t}(K), W_{\bullet}{}^p\phi_{t}(K))
\end{align}
define mixed Hodge modules.
Note that this is also one of the conditions for $(M,F_{\bullet}M, K, W_{\bullet}K)$ to be a mixed Hodge module.

\begin{lem}\label{arashi}
For a monodromic mixed Hodge module $(M,F_{\bullet}M, K, W_{\bullet}K)$ on $X\times \CC_{t}$,
for $k\in \ZZ$
we have
\begin{align*}
W_{k}\psi_{t}(M)&=L_{k}\psi_{t}(M)(=\psi_{t}(W_{k+1}M)),\quad \mbox{and}\\
W_{k}\phi_{t,1}(M)&=L_{k}\phi_{t,1}(M)(=\phi_{t,1}(W_{k}M)).
\end{align*}
Similarly to them, we have
\begin{align*}
W_{k}{}^p\psi_{t}(K)&=L_{k}{}^p\psi_{t}(K)(={}^p\psi_{t}(W_{k+1}K)),\quad \mbox{and}\\
W_{k}{}^p\phi_{t,1}(K)&=L_{k}{}^p\phi_{t,1}(K)(={}^p\phi_{t,1}(W_{k}K)).
\end{align*}

\end{lem}
\begin{proof}
Since the problem is local, we may assume that $X$ is affine.
Then, we have
\begin{align*}
L_{k}\psi_{t}(M)&=\bigoplus_{\alpha\in (-1,0]}W_{k+1}M^{\alpha},\quad \mbox{and}\\
L_{k}\phi_{t}(M)&=W_{k}M^{-1}.
\end{align*}
Moreover, we have
\begin{align}\label{matsurino}
\GR^{L}_{k}\psi_{t}(M)&=\bigoplus_{\alpha\in (-1,0]}\GR^{W}_{k+1}M^{\alpha}, \quad\mbox{and}\\
\GR^{L}_{k}\phi_{t}(M)&=\GR^{W}_{k}M^{-1}.\notag
\end{align}
By the definition of the monodromy filtration, for $l\geq 0$ and $k\in \ZZ$
we have the isomorphism
\begin{align}\label{lalala}
N^{l}\colon \GR^{W}_{k+l}\GR^{L}_{k}\psi_{t}(M)\to \GR^{W}_{k-l}\GR^{L}_{k}\psi_{t}(M).
\end{align}
Note that by Corollary~\ref{mouikkai} and (\ref{matsurino}),
$N$ is zero on $\GR^{L}_{k}\psi_{t}(M)$.
Therefore, 
if $l>0$, the isomorphism (\ref{lalala}) is zero.
Hence, 
\begin{align*}
\GR^{W}_{s}\GR^{L}_{k}\psi_{t}(M)=0,
\end{align*}
for $s\neq k$.
Therefore, we have
\[W_{k}\psi_{t}(M)=L_{k}\psi_{t}(M).\]
By the same argument, 
we also have
\[W_{k}\phi_{t,1}(M)=L_{k}\phi_{t,1}(M).\]

In the exactly same way, we can also show the assertion for the perverse sheaves ${}^p\psi_{t}(K)$ and ${}^p\phi_{t}(K)$.
\end{proof}

\begin{cor}
In the situation of Lemma~\ref{arashi},
for $\alpha\in [-1,0]$ and $k\in \ZZ$ we have
\begin{align}
(t\partial_{t}-\alpha)(W_{k}M^{\alpha})\subset W_{k-2}M^{\alpha}.
\end{align}
\end{cor}
\begin{proof}
By the definition of the relative monodromy filtration,
we have
\begin{align*}
N(W_{k}\psi_{t}(M))&\subset W_{k-2}\psi_{t}(M),\quad\mbox{and}\\
N(W_{k}\phi_{t,1}(M))&\subset W_{k-2}\phi_{t,1}(M).
\end{align*}
Hence, the assertion follows from Lemma~\ref{arashi}.
\end{proof}

\begin{cor}\label{imamade}
In the situation of Lemma~\ref{arashi}, for $\alpha\in (-1,0]$
the nearby cycle (resp. unipotent vanishing cycle) of $\calM=(M,F_{\bullet}M,K, W_{\bullet}K)$ is the tuple
\begin{align*}
(\bigoplus_{\alpha\in (-1,0]}M^{\alpha},\bigoplus_{\alpha\in (-1,0]}F_{\bullet}M^{\alpha},{}^p\psi_{t}(K),{}^p\psi_{t}(W_{\bullet+1}K))\\
(resp. \quad
(M^{-1},F_{\bullet+1}M^{-1},{}^p\phi_{t,1}(K),{}^p\phi_{t,1}(W_{\bullet}K))).
\end{align*}
\end{cor}

Finally, we observe a smooth monodromic mixed Hodge module, i.e. the monodromic mixed Hodge module whose underlying $D$-module is locally free on $X\times \CS$.
Note that this condition is equivalent to all $M^{\alpha}$ ($\alpha\in (-1,0]$)
being locally free $O_{X}$-module by Proposition~\ref{intlemkai}.
Under this assumption, as $D$-modules (resp. perverse sheaves), 
the nearby cycle $\psi_{t}(M)$ (resp. ${}^p\psi_{t}K$) and
the restriction of $M$ (resp. $K$) to $X\times \{1\}$ are isomorphic.
In this case, these are isomorphic in the category of mixed Hodge modules as follows.

\begin{cor}
We assume that $M|_{X\times \CS}$ is locally free $O_{X\times \CS}$-module.
Then, the restriction of $\calM|_{X\times \CS}$ to $X\times \{1\}$ (as a mixed Hodge module) is isomorphic to the nearby cycle $\psi_{t}(\calM)$
\end{cor}

\begin{proof}
Let $i\colon X\times \{1\}\hookrightarrow X\times \CS$ be the inclusion.
Since $M|_{X\times \CS}$ is locally free $O_{X\times \CS}$-module,
the pullback of $M|_{X\times \CS}$ by $i$ as a mixed Hodge module is
isomorphic to the nearby cycle $\psi_{t-1}(\calM|_{X\times \CS})$ of it along $t=1$.
By the definition,
$\psi_{t-1}(M|_{X\times \CS})$ is equal to
\[(M/(t-1)M, F_{\bullet}M/(t-1)F_{\bullet}M, i^{-1}K, i^{-1}W_{\bullet+1}K).\]
This is isomorphic to the nearby cycle $\psi_{t}(\calM)$ by Corollary~\ref{imamade}.
\end{proof}

\subsection{{T}he equivalence of categories}\label{jyupitor}

We state the main result of this section.
We consider a tuple 
\[((\calM_{(-1,0]}, T_{s}, N), \calM_{-1}, c,v),\]
where $\calM_{(-1,0]}=(M_{(-1,0]},F_{\bullet}M_{(-1,0]},K_{(-1,0]},W_{\bullet}K_{(-1,0]})$ and $\calM_{-1}$ are graded polarizable mixed Hodge modules on $X$ and
$T_{s}\colon \calM_{(-1,0]}\simto \calM_{(-1,0]}$,
$N\colon \calM_{(-1,0]}\to \calM_{(-1,0]}(-1)$,
$c\colon \calM_{0}(:=\mathrm{Ker}(T_{s}-1)\subset \calM_{(-1,0]})\to \calM_{-1}$ and
$v\colon \calM_{-1}\to \calM_{0}(-1)$ are morphisms between mixed Hodge modules.
Let $((\calM_{(-1,0]}', T_{s}', N'), \calM_{-1}', c',v')$ be another tuple.
The morphism between them is a pair $(\varphi_{(-1,0]},\varphi_{-1})$ of morphisms
$\varphi_{(-1,0]}\colon \calM_{(-1,0]}\to \calM_{(-1,0]}'$ and
$\varphi_{-1}\colon \calM_{-1}\to \calM_{-1}'$ which are compatible with the morphisms $T_{s}$, $T_{s}'$, $N$, $N'$, $c$, $c'$, $v$ and $v'$.
For a tuple $((\calM_{(-1,0]}, T_{s}, N), \calM_{-1}, c,v)$,
we consider the following three conditions:
\begin{enumerate}
\item[($\star$-i)] $T_{s}$ commutes with $N$.
\item[($\star$-ii)] The $D$-module $M_{(-1,0]}$ is decomposed as
\begin{align}\label{takarajimakai}
M_{(-1,0]}=\bigoplus_{\alpha\in (-1,0]\cap \QQ}M_{\alpha},
\end{align}
where $M_{\alpha}:=\mathrm{Ker}(T_{s}-e(\alpha))\subset M$ ($e(\alpha)=\exp(-2\pi\sqrt{-1}\alpha)$).
\item[($\star$-iii)] $vc\colon \calM_{0}\to \calM_{0}(-1)$ is (the restriction of) $-N$.
\end{enumerate}
We denote by $\mathscr{G}(X)$ the category of tuples $((\calM_{(-1,0]}, T_{s}, N), \calM_{-1}, c,v)$ with the conditions ($\star$-i), ($\star$-ii) and ($\star$-iii).
Let $\mathrm{MHM}^{p}_{\mathrm{mon}}(X\times \CC)$ be the category of monodromic graded polarizable mixed Hodge modules on $X\times \CC$.
For an object $\calM\in \mathrm{MHM}^{p}_{\mathrm{mon}}(X\times \CC)$,
we get a tuple 
\[((\psi_{t}\calM, T_{s}, N), \phi_{t,1}\calM, \mathrm{can},\mathrm{var}),\]
where $T_{s}$ (resp. $N$) is the semisimple part (resp. $\frac{-1}{2\pi\sqrt{-1}}$ times the logarithm of the unipotent part) of the monodromy automorphism of the nearby cycle functor and
$\mathrm{can}$ and $\mathrm{var}$ is the can morphism $\psi_{t,1}\calM\to \phi_{t,1}\calM$ and the var morphism $\phi_{t,1}\calM\to \psi_{t,1}\calM(-1)$.
Then, 
$((\psi_{t}\calM, T_{s}, N), \phi_{t,1}\calM, \mathrm{can},\mathrm{var})$
is an object in $\mathscr{G}(X)$.
In this way,
we obtain a functor
\[F\colon \mathrm{MHM}^{p}_{\mathrm{mon}}(X\times \CC)\to \mathscr{G}(X).\]
The following theorem is our second main result. 
\begin{thm}\label{main2kai}
There is a quasi-inverse functor $G\colon  \mathscr{G}(X)\to \mathrm{MHM}^{p}_{\mathrm{mon}}(X\times \CC)$ of $F$, i.e.
$F$ and $G$ induce an equivalence of categories:
\[\mathrm{MHM}^{p}_{\mathrm{mon}}(X\times \CC)\simeq \mathscr{G}(X).\]
\end{thm}

We will construct the functor $G$ and prove this theorem in Subsection~\ref{syoumei}. 

\subsection{The $D_{\CS}$-modules $S_{m,1}$, $S_{1,\lambda}$ and $L_{r}$}\label{machito}
To construct a monodromic mixed Hodge module on $X\times \CC$ from an object $((\calM_{(-1,0]}, T_{s}, N),$ $\calM_{-1}, c,v)\in \mathscr{G}(X)$ and define the functor $G$, we take the following way.
First, we define monodromic $D$-modules $S_{m,1}$, $S_{1,\lambda}$ and $L_{r}$ for $m,r\in \ZZ_{\geq 1}$ and $\lambda\in \CC$ on $\CS$ (in this subsection).
The monodromy action of the underlying local system of $S_{m,1}$ (resp. $S_{1,\lambda}$) is semisimple, and its eigenvalues are the $m$-th roots of unity (resp. $\lambda$).
The one of $L_{r}$ is unipotent, and its Jordan normal form is a single Jordan cell of eignenvalue $1$ with size $r$.
Next, we define monodromic mixed Hodge modules $\calS^{H}_{m,1}[1]$ and $\calL^{H}_{r}[1]$ whose underlying $D$-modules are $S_{m,1}$ and $L_{r}$ (in Subsection~\ref{yoruno}).
Then, 
we define the mixed Hodge module $\wt{\calM_{(-1,0]}}$ on $X\times \CS$ as a quotient of a mixed Hodge module
$\calM\boxtimes (\calS^{H}_{m,1}\otimes \calL^{H}_{r})[1]$ on $X\times \CS$ for some $m, r\in \ZZ_{\geq 1}$ (in Subsection~\ref{chako}).
Finally, by gluing the tuple $(\wt{\calM_{(-1,0]}}, \calM_{-1},c,v)$,
we will get the desired mixed Hodge module $\vardbtilde{\calM_{(-1,0]}}$ on $X\times \CC$ (in Subsection~\ref{syoumei}), and define the functor $G$ so that
$G(((\calM_{(-1,0]}, T_{s}, N),$ $\calM_{-1}, c,v))=\vardbtilde{\calM_{(-1,0]}}$.

In this subsection, we introduce a $D$-modules $S_{m,1}$, $S_{1,\lambda}$ and $L_{r}$.
For $m\in \ZZ_{\geq 1}$, $\KK=\CC$ or $\QQ$ and $\lambda\in \CC$,
let $V_{m,1}^{\semi,\KK}$ (resp. $V_{1,\lambda}^{\semi,\CC}$) be the $\KK$(resp. $\CC$)-vector space of dimension $m$ with the automorphism $T$ whose eigenvalues are $\{\mu\in \CC\ |\ \mu^m=1\}$ (resp. $\lambda$).
Note that we have 
\begin{align}\label{mira}
V_{m,1}^{\semi,\CC}\simeq \bigoplus_{\lambda^m=1}V_{1,\lambda}^{\semi,\CC}.
\end{align} 
For $r\in \ZZ_{\geq 1}$, let $V_{r}^{\nilp,\KK}$ be the $\KK$($=\CC$ or $\QQ$)-vector space of dimension $r$ with the nilpotent morphism $N$ of $V_{r}^{\nilp,\KK}$ to $V_{r}^{\nilp,\KK}(-1)$ such that $N^{r}=0$ and $N^{r-1}\neq 0$,
where we set $V_{r}^{\nilp,\QQ}(-1):=\frac{1}{2\pi\sqrt{-1}}V_{r}^{\nilp,\QQ}(\subset V_{r}^{\nilp,\CC})$ and $V_{r}^{\nilp,\CC}(-1)=V_{r}^{\nilp,\CC}$.
The Jordan normal form of $2\pi\sqrt{-1}N(\colon V_{r}^{\nilp,\KK}\to V_{r}^{\nilp,\KK})$ is the single Jordan cell $J_{0,r}$ of the eigenvalue $0$ with size $r$.
We will fix the basis $e_{1},\dots,e_{r}$ of $V_{r}^{\nilp,\KK}$ so that we have $2\pi\sqrt{-1}Ne_{i}=e_{i-1}$ ($e_{0}:=0$).
We define the automorphism $T$ of $V_{r}^{\nilp,\KK}$ as
$T:=\exp(-2\pi\sqrt{-1}N)$.

For $m,r\in \ZZ_{\geq 1}$, $\lambda\in \CC$ and $\KK=\CC$ or $\QQ$,
let $\calS_{m,1}^{\KK}$ (resp. $\calS_{1,\lambda}^{\CC}$, $\calL_{r}^{\KK}$) be the local system on $\CS$ corresponding to the vector space with automorphism $(V_{m,1}^{\semi,\KK}, T)$ (resp. $(V_{1,\lambda}^{\semi,\CC},T)$, $(V_{r}^{\nilp,\KK},T)$).
We use the same symbol $T$ for the monodromy automorphism of $\calS_{m,1}^{\KK}$, $\calS_{1,\lambda}^{\CC}$ and $\calL_{r}^{\KK}$.
Moreover, we define $N\colon \calL_{r}^{\QQ}\to \calL_{r}^{\QQ}(-1)$ as $\frac{-1}{2\pi\sqrt{-1}}$ times the logarithm of the unipotent part of the monodromy automorphism $T\colon \calL_{r}^{\QQ}\simto \calL_{r}^{\QQ}$.
By (\ref{mira}), note that we have 
\begin{align}\label{skeld}
\calS_{m,1}^{\CC}\simeq \bigoplus_{\lambda^m=1}\calS_{1,\lambda}^{\CC}
\end{align}

We define the $D$-module (an integrable connection) $S_{m,1}$ (resp. $S_{1,\lambda}$, $L_{r}$) on $\CS$ as
\begin{align*}
S_{m,1}:=\calS_{m,1}^{\CC}\otimes_{\CC}\calO_{\CS}\quad\\
\left(\mbox{resp.\ } 
\begin{array}{ll}
S_{1,\lambda}:=\calS_{1,\lambda}^{\CC}\otimes_{\CC}\calO_{\CS},\\
L_{r}:=\calL_{r}^{\CC}\otimes_{\CC}\calO_{\CS}.
\end{array}\right)
\end{align*}
whose underlying local system is $\calS_{m,1}^{\CC}$ (resp. $\calS_{1,\lambda}^{\CC}$, $\calL_{r}^{\CC}$).
By (\ref{skeld}), we have
\[S_{m,1}=\bigoplus_{\lambda^m=1}S_{1,\lambda}.\]
We can describe the structure of $S_{1,\lambda}$ and $L_{r}$ as follows.

\begin{lem}\label{luckyland}
For $\lambda=\exp(-2\pi\sqrt{-1}\alpha)\in \CC$ ($\alpha\in (-1,0]\cap \QQ$) (resp. $r\in \ZZ_{\geq 1}$),
$S_{1,\lambda}$ (resp. $L_{r}$) is a free $O_{\CS}$-module of rank $1$ (resp. $r$) and there is a basis of $S_{1,\lambda}$ (resp. $L_{r}$) whose connection matrix is
\begin{align*}
&\dfrac{\alpha}{t}\\
(\mbox{resp.\ } &\dfrac{1}{2\pi\sqrt{-1}}\dfrac{J_{0,r}}{t}),
\end{align*}
where $J_{0,r}$ is the Jordan cell of the eigenvalue $0$ with size $r$.
In particular, we can regard $S_{1,\lambda}$ and $L_{r}$ as algebraic $D$-modules (or an algebraic integrable connections), $S_{1,\lambda}$ and $L_{r}$ are monodromic and we have
\begin{align}
S_{1,\lambda}&=S_{1,\lambda}^{\alpha}\otimes \CC[t^{\pm}],\tunagi\\
L_{r}&=L_{r}^{0}\otimes\CC[t^{\pm}],\label{bien}
\end{align}
as $O$-modules, where 
$S_{1,\lambda}^{\alpha}:=\mathrm{Ker}(t\partial_{t}-\alpha \colon L_{r}\to L_{r})$
(resp. $L_{r}^{0}:=\mathrm{Ker}((t\partial_{t})^r\colon L_{r}\to L_{r})$).
Moreover, we can define natural isomorphisms
\begin{align}\label{communication}
\sigma\colon V_{1,\lambda}^{\semi,\CC}\simto S_{1,\lambda}^{\alpha}\\
(\mbox{resp.\ } \sigma\colon V_{r}^{\nilp,\CC}\simto L_{r}^{0}).\label{amongus}
\end{align}
so that for $v\in V_{1,\lambda}^{\semi, \CC}$ (resp. $v\in V_{r}^{\nilp,\CC}$), we have
\begin{align}\label{magic}
&\quad\quad\qquad   \begin{array}{l}
(t\partial_{t}-\alpha)\sigma(v)=0,\tunagi\\ 
\sigma(Tv)=\exp(-2\pi\sqrt{-1}t\partial_{t})\sigma(v)\\
\end{array}\\
&\label{gathering}\left(\mbox{resp.\ }
\begin{array}{l} 
\sigma(2\pi\sqrt{-1}Nv)=2\pi\sqrt{-1}t\partial_{t}\sigma(v),\tunagi\\
\sigma(Tv)=\exp(-2\pi\sqrt{-1}t\partial_{t})\sigma(v).
\end{array}
\right).
\end{align}
\end{lem}

\begin{rem}
In general, for a regular holonomic $D$-module $M$ on $X\times \CC_{t}$,
there is a natural isomorphism
\begin{align}\label{zeldaa}
\mathrm{DR}_{X\times \{0\}}(\psi_{t,\lambda}(M))\simeq {}^p\psi_{t,\lambda}(\mathrm{DR}_{X\times \CC}(M)),
\end{align}
with some relations similar to (\ref{gathering})
between the monodromy action on ${}^p\psi_{t,\lambda}(\mathrm{DR}_{X\times \CC}(M))$ and 
$\exp(-2\pi\sqrt{-1}t\partial_{t})$.
In the setting in Lemma~\ref{luckyland},
we have 
\begin{align*}
\psi_{t,1}(L_{r})&=L_{r}^{0},\\
{}^p\psi_{t,1}(\calL_{r}[1])&=V_{r}^{\nilp, \CC},\tunagi\\
\mathrm{DR}(L_{r})&=\calL_{r}^{\CC}[1].
\end{align*}
Then, we can regard the isomorphism (\ref{amongus}) as the isomorphism (\ref{zeldaa}) in a special case.
We can say the same for $S_{1,\lambda}$ and the vanishing cycles.
\end{rem}

\subsection{The mixed Hodge modules $\calS^{H}_{m,1}[1]$ and $\calL^{H}_{r}[1]$}\label{yoruno}

We endow $\calS_{m,1}^{\QQ}$ and $\calL_{r}^{\QQ}$ with mixed Hodge module structures.
We define a Hodge filtration of $V_{m,1}^{\semi,\CC}$ as
\[F_{p}V_{m,1}^{\semi,\CC}:=
\left\{
\begin{array}{ll}
V_{m,1}^{\semi,\CC}& (p\geq  0),\\
0& (p<0).
\end{array}
\right.
\]
Recall that $e_{1},\dots,e_{r}$ is a (fixed) basis of $V_{r}^{\nilp,\KK}$ such that $2\pi\sqrt{-1}Ne_{i}=e_{i-1}$ ($e_{0}:=0$). 
Then, we define filtrations $F_{\bullet}V_{r}^{\nilp,\CC}$ and $W_{\bullet}V_{r}^{\nilp,\QQ}$ as
\begin{align*}
{F}_{i}V_{r}^{\nilp,\CC}&:=
\sum_{j=r-i}^{r}\CC e_{j},\\
{W}_{-2i}V_{r}^{\nilp,\QQ}=W_{-2i+1}V_{r}^{\nilp,\QQ}&:=
\sum_{j=1}^{r-i}\QQ e_{j}.
\end{align*}
Note that we have
\begin{align*}
N(F_{p}V_{r}^{\nilp,\CC})\subset  F_{p+1}V_{r}^{\nilp,\CC},
\end{align*}
for any $p\in \ZZ$.
It is obvious that
\begin{align*}
W_{-2i}V_{r}^{\nilp,\QQ}=W_{-2i+1}V_{r}^{\nilp,\QQ}=
\left\{\begin{array}{ll}
\mathrm{Ker}(N^{r-i}\colon V_{r}^{\nilp,\QQ}\to V_{r}^{\nilp,\QQ}(-(r-i)))& (i<r),\\
0& (i\geq r).
\end{array}\right.
\end{align*}
Note that $W_{\bullet}V_{r}^{\nilp,\QQ}$ is the monodromy weight filtration of $N$ centered at $-(r-1)$.
We set $V_{r}^{\nilp,H}:=(V_{r}^{\nilp,\CC},F_{\bullet}V_{r}^{\nilp,\CC}, V_{r}^{\nilp,\QQ}, W_{\bullet}V_{r}^{\nilp,\QQ})$.
Recall that the $l$-th Tate twist $V_{r}^{\nilp, H}(l)$ of $V_{r}^{\nilp, H}$ is the tuple:
\[(V_{r}^{\nilp, \CC},F_{\bullet-l}V_{r}^{\nilp, \CC}, (2\pi\sqrt{-1})^{l}V_{r}^{\nilp, \QQ}, W_{\bullet+2l}V_{r}^{\nilp, \QQ}).\]
The following lemma follows from the definition.
\begin{lem}
\begin{enumerate}
\item[(i)]
The triple $V_{m,1}^{\semi,H}:=(V_{m,1}^{\semi,\CC},F_{\bullet}V_{m,1}^{\semi,\CC},V_{m,1}^{\QQ})$ is polarizable pure Hodge structure of weight $1$.
\item[(ii)]
The tuple 
\begin{align}
V_{r}^{\nilp,H}=(V_{r}^{\nilp, \CC},F_{\bullet}V_{r}^{\nilp, \CC}, V_{r}^{\nilp,\QQ}, W_{\bullet}V_{r}^{\nilp, \QQ})
\end{align}
is a graded polarizable mixed Hodge structure.
\item[(iii)]
The morphism $N$ defines a nilpotent morphism between the mixed Hodge structures
\[N\colon V_{r}^{\nilp,H}\to V_{r}^{\nilp,H}(-1).\]
\item[(iv)]
The mixed Hodge structure $V_{r}^{\nilp,H}$
is a direct sum of some Tate twists of the pure Hodge structure $V_{1}^{\nilp,H}$ , i.e. we have
\[V_{r}^{\nilp,H}\simeq V_{1}^{\nilp,H}\oplus V_{1}^{\nilp,H}(1)\oplus\dots \oplus V_{1}^{\nilp,H}(r-1).\]
\end{enumerate}
\end{lem}

We define a filtration $F_{\bullet}S_{m,1}$ of $S_{m,1}$ as
\begin{align*}
F_{p}S_{m,1}:=\left\{
\begin{array}{ll}
S_{m,1}& (p\geq 0),\\
0& (p<0).
\end{array}\right.
\end{align*}
Moreover, we set
\begin{align*}
F_{\bullet}L_{r}^{0}&:=\sigma(F_{\bullet}V_{r}^{\nilp,\CC}),\tunagi\\
W_{\bullet}L_{r}^{0}&:=\sigma(W_{\bullet-1}V_{r}^{\nilp,\CC}),
\end{align*}
where $L_{r}^{0}$ and $\sigma$ are the ones defined in Lemma~\ref{luckyland}.
By the expression (\ref{bien}),
we define the Hodge and weight filtration of $L_{r}$ as
\begin{align*}
F_{\bullet}L_{r}&:=F_{\bullet}L_{r}^{0}\otimes \CC[t^{\pm}],\tunagi\\
W_{\bullet}L_{r}&:=W_{\bullet}L_{r}^{0}\otimes \CC[t^{\pm}].
\end{align*}
Since we have $\sigma(Nv)=t\partial_{t}\sigma(v)$ for $v\in V_{r}^{\nilp,\CC}$ (see (\ref{gathering})),
we can express the filtration $W_{\bullet}L^{0}_{r}$ of $L^{0}_{r}$
as
\[
W_{-2i-1}L^{0}_{r}=W_{-2i}L^{0}_{r}
=\left\{
\begin{array}{ll}
\mathrm{Ker}((t\partial_{t})^{r-i-1}\colon L^{0}_{r}\to L^{0}_{r})& (i<r-1)\\
0& (i\geq r-1).
\end{array}
\right.
\]
In particular, $W_{k}L_{r}\subset L_{r}$ is a $D$-submodule of $L_{r}$ for $k\in \ZZ$.
On the other hand, we can also define the weight filtration $W_{\bullet}\calL_{r}^{\KK}$ of the local system $\calL_{r}^{\KK}$ so that $W_{k}\calL_{r}^{\KK}\otimes_{\KK}\calO_{\CS}=W_{k}L_{r}$.

\begin{prop}\label{hoikoro}
\item[(i)] The tuple $\calS_{m,1}^{H}[1]:=(S_{m,1},F_{\bullet}S_{m,1},\calS_{m,1}^{\QQ}[1])$ (resp. $\calL_{r}^{H}[1]:=(L_{r},F_{\bullet}L_{r},\calL_{r}^{\QQ}[1],W_{\bullet}\calL_{r}^{\QQ})$) is a polarizable pure Hodge module (resp. graded polarizable mixed Hodge module) on $\CS$.
\item[(ii)]
The monodromy automorphism $T\colon \calS_{m,1}^{\QQ}\simto \calS_{m,1}^{\QQ}$ induces an automorphism $T\colon \calS_{m,1}^{H}[1]\simto \calS_{m,1}^{H}[1]$ of $\calS_{m,1}^{H}[1]$ as a mixed Hodge module.
\item[(iii)]
The morphism $N\colon \calL_{r}^{\QQ}\to \calL_{r}^{\QQ}(-1)$ induces a nilpotent morphism $N\colon \calL_{r}^{H}\to \calL_{r}^{H}(-1)$.
\item[(iv)]
The non-trivial cohomology of $\calS_{m,1}^{H}\otimes \calL_{r}^{H}[1]\in \DBkai\mathrm{MHM}(\CS)$ is only $0$-th one.
In particular, we can regard $\calS_{m,1}^{H}\otimes \calL_{r}^{H}[1]$ as a mixed Hodge module on $\CS$ associated to the variation of mixed Hodge structure whose underlying local system is $\calS_{m,1}^{\QQ}\otimes_{\QQ}\calL_{r}^{\QQ}$.
\item[(v)] We have
\begin{align*}
\psi_{t}(\calL^{H}_{r}[1])=&(V_{r}^{\nilp,\CC},F_{\bullet}V_{r}^{\nilp,\CC},V_{r}^{\nilp,\QQ},W_{\bullet}V_{r}^{\nilp,\QQ}),\tunagi\\
\psi_{t}(\calS^{H}_{m,1}\otimes \calL^{H}_{r}[1])=&
(V_{m,1}^{\semi,\CC}\otimes V_{r}^{\nilp,\CC}, F_{\bullet}(V_{m,1}^{\semi,\CC}\otimes V_{r}^{\nilp,\CC}), V_{m,1}^{\semi,\QQ}\otimes V_{r}^{\nilp,\QQ}, V_{m,1}^{\semi,\QQ}\otimes W_{\bullet}V_{r}^{\nilp,\QQ}),
\end{align*}
where $F_{p}(V_{m,1}^{\semi,\CC}\otimes V_{r}^{\nilp,\CC}):=\sum_{s+t=p}F_{s}V_{m,1}^{\semi,\CC}\otimes F_{t}V_{r}^{\nilp,\CC}$.
\item[(vi)] For $a\in \ZZ_{\geq 1}$, we have a natural morphism which commutes with $T$:
\begin{align}\label{murata}
\calS^{H}_{m,1}[1]\to \calS^{H}_{am,1}[1].
\end{align}
\item[(vii)] For $l\in \ZZ_{\geq 0}$,
we have natural morphisms which commute with $N$:
\begin{align}\label{kubota}
\calL^{H}_{r+l}[1]\to& \calL^{H}_{r}[1],\tunagi\\
\calL^{H}_{r}[1]\to& \calL^{H}_{r+l}[1](-l).\label{torosa}
\end{align}
\end{prop}
\begin{proof}
\noindent (i):\ 
Let $\QQ^{H}_{\CS}[1]\in \mathrm{HM}^{p}(\CS)$ be the constant polarizable pure Hodge module on $\CS$.
For $m\in \ZZ_{\geq 1}$, let $r_{m}\colon \CS\to \CS$ be the morphism $t\mapsto r_{m}(t):=t^{m}$.
Then, 
$\calS_{m,1}^{H}[1]$ is isomorphic to the pushforward $(r_{m})_{*}\QQ_{\CS}^{H}[1]$ of the mixed Hodge module $\QQ_{\CS}^{H}[1]$, which is a polarizable pure Hodge module of weight $1$.
On the other hand, 
since the tuple $(L_{r}, F_{\bullet}L_{r}, \calL_{r}^{\QQ},W_{\bullet+1}\calL_{r}^{\QQ})$ is an admissible graded polarizable variation of mixed Hodge structure on $\CS$,
the tuple $\calL_{r}^{H}[1]$ is a graded polarizable mixed Hodge module on $\CS$.

\noindent (ii), (iii), (iv), (v):\ These follow from the definition.

\noindent (vi) and (vii):\ 
For $m,a\in \ZZ_{\geq 1}$, 
the natural morphism 
\[V_{m,1}^{\semi,\QQ}\to V_{am,1}^{\semi,\QQ},\]
which commutes with $T$ induces the morphism (\ref{murata}).
Similarly, for $r\in \ZZ_{\geq 1}$ and $l\in \ZZ_{\geq 0}$,
the natural morphisms
\begin{align*}
V_{r+l}^{\nilp,\QQ}&\to V_{r}^{\nilp,\QQ},\tunagi\\
V_{r}^{\nilp,\QQ}&\to V_{r+l}^{\nilp,\QQ}(-l),
\end{align*}
which commute with $N$ induce the morphisms (\ref{kubota}) and (\ref{torosa}).

\end{proof}

\begin{rem}
The Hodge filtration of the variation of mixed Hodge structure $(L_{r}, F_{\bullet}L_{r}, \calL_{r}^{\QQ},W_{\bullet+1}\calL_{r}^{\QQ})$ is the nilpotent orbit in the sense of \cite{SchVar} defined from the mixed Hodge structure with a nilpotent morphism $((V_{r}^{\nilp,\CC},F_{\bullet}V_{r}^{\nilp,\CC}, V_{r}^{\nilp,\QQ}, W_{\bullet}V_{r}^{\nilp,\QQ}),N)$.
\end{rem}











\subsection{A construction of monodromic mixed Hodge modules on $X\times \CS$}
\label{chako}

In order to construct $G$ in Theorem~\ref{main2kai},
we show a weak version of the theorem, i.e.
we construct a mixed Hodge module on $X\times \CS$ (not $X\times \CC$) from a mixed Hodge module on $X$ with some operators with the aid of the Hodge modules on $\CS$ defined in the previous subsection.

We consider a triple $(\calM,T_{s},N)$, where $\calM=(M,F_{\bullet}M,K,W_{\bullet}K)$ is a graded polarizable mixed Hodge module on $X$,
$T_{s}\colon  \calM\simto \calM$ and $N\colon \calM\to \calM(-1)$ are morphisms with the conditions ($\star$-i) and ($\star$-ii) in Subsection~\ref{jyupitor}:
\begin{enumerate}
\item[($\star$-i)] $T_{s}$ commutes with $N$.
\item[($\star$-ii)] The $D$-module $M$ is decomposed as
\begin{align}\label{takarajima}
M=\bigoplus_{\alpha\in (-1,0]\cap \QQ}M_{\alpha},
\end{align}
where $M_{\alpha}:=\mathrm{Ker}(T_{s}-e(\alpha))\subset M$ ($e(\alpha):=\exp(-2\pi\sqrt{-1}\alpha)$).
\end{enumerate}
We denote by $\mathscr{G}'(X)$ the category of such triples.
By the condition ($\star$-ii),
for $p\in \ZZ$ and $k\in \ZZ$ we have decompositions
\begin{align}
F_{p}M=&\bigoplus_{\alpha\in (-1,0]\cap\QQ}F_{p}M_{\alpha},\label{cruise}\tunagi\\
W_{k}M=&\bigoplus_{\alpha\in (-1,0]\cap \QQ}W_{k}M_{\alpha},\notag
\end{align}
where 
we set $F_{p}M_{\alpha}=F_{p}M\cap M_{\alpha}\subset F_{p}M$ and
$W_{k}M_{\alpha}:=W_{k}M\cap M_{\alpha}\subset W_{k}M$.
Note that the condition ($\star$-ii) is equivalent to the condition:
\[\mbox{($\star$-ii') the perverse sheaf $K_{\CC}:=K\otimes \CC$ is decomposed as $K_{\CC}=\bigoplus_{\alpha\in (-1,0)\cap \QQ}K_{\CC,e(\alpha)}$}, \]
where $K_{\CC,\alpha}:=\mathrm{Ker}(T_{s}-e(\alpha))\subset K_{\CC}$.
Since the weight filtration of $\calM$ is a finite filtration, $N$ is a nilpotent operator.

Let $\mathrm{MHM}^{p}_{\mathrm{mon}}(X\times \CS)$ be the category of monodromic graded polarizable mixed Hodge modules on $X\times \CS$.
For an object $\calM\in \mathrm{MHM}^{p}_{\mathrm{mon}}(X\times \CS)$, 
we get a triple $(\psi_{t}\calM, T_{s},N)$, where $T_{s}$ (resp. $N$) is the semisimple (resp. $\frac{-1}{2\pi\sqrt{-1}}$ times the logarithm of the unipotent) part of the monodromy automorphism.
In this way, we obtain a functor
\[F'\colon \mathrm{MHM}^{p}_{\mathrm{mon}}(X\times \CS)\to \mathscr{G}'(X).\]
We will show the following weak version of Theorem~\ref{main2kai}.
\begin{lem}\label{daijyoubuda}
There is a quasi-inverse functor $G'\colon \mathscr{G}'(X)\to \mathrm{MHM}^{p}_{\mathrm{mon}}(X\times \CS)$ of $F'$, i.e.
$F'$ and $G'$ induce an equivalence of categories:
\[\mathrm{MHM}^{p}_{\mathrm{mon}}(X\times \CS)\simeq \mathscr{G}'(X).\]
\end{lem}

Let $(\calM,T_{s},N)$ be an object in $\mathscr{G}'(X)$.
For $r\in \ZZ_{\geq 1}$,
we define the mixed Hodge modules on $X\times \CS$ as
\begin{align*}
\wt{\calM}^{\skaa}_{r}&:=\kaa(N\boxtimes 1-1\boxtimes N\colon 
\calM\boxtimes \calL_{r}^H[1]\to \calM\boxtimes \calL_{r}^{H}[1](-1))(-(r-1)),\tunagi\\
\wt{\calM}^{\scokaa}_{r}&:=\cokaa(N\boxtimes 1-1\boxtimes N\colon 
\calM\boxtimes \calL_{r}^H[1](1)\to \calM\boxtimes \calL_{r}^{H}[1]).
\end{align*}
The automorphism $T_{s}$ induces automorphisms of $\wt{\calM}^{\skaa}_{r}$ and $\wt{\calM}^{\scokaa}_{r}$, which will be denoted by the same symbol $T_{s}$.

The morphisms defined in (vii) of Proposition~\ref{hoikoro} induce morphisms
\begin{align}\label{evans}
\left\{
\begin{array}{l}
\wt{\calM}^{\skaa}_{r+l}\to \wt{\calM}^{\skaa}_{r}(-l),\\
\wt{\calM}^{\skaa}_{r}\to \wt{\calM}^{\skaa}_{r+l},\\
\wt{\calM}^{\scokaa}_{r+l}\to \wt{\calM}^{\scokaa}_{r},\tunagi\\
\wt{\calM}^{\scokaa}_{r}\to \wt{\calM}^{\scokaa}_{r+l}(-l).
\end{array}
\right.
\end{align}
Take a sufficiently large $l_{0}\in\ZZ_{\geq 1}$ so that
\[M\subset \kaa(N^{l_{0}}\colon M\to M(-l_{0})).\]
Then, by the direct computation (in the category of $D$-modules), we obtain the following lemma.

\begin{lem}\label{baloon}
For $r,l\in \ZZ_{\geq 1}$, 
if $l\geq l_{0}$,
the morphisms
\begin{align*}
\wt{\calM}^{\skaa}_{r+l}&\to \wt{\calM}^{\skaa}_{r}(-l),\tunagi\\
\wt{\calM}^{\scokaa}_{r}&\to \wt{\calM}^{\scokaa}_{r+l}(-l)
\end{align*}
are zero.
\end{lem}

\begin{cor}\label{gift}
\begin{enumerate}
\item[(i)]
If $r\in \ZZ_{\geq 1}$ is sufficiently large,
$\wt{\calM}^{\skaa}_{r}$ and $\wt{\calM}^{\scokaa}_{r}$ do not depend on the choice of $r$, i.e.
for sufficiently large $r\in \ZZ_{\geq 1}$ we have isomorphisms:
\begin{align*}
\wt{\calM}^{\skaa}_{r}&\simto \wt{\calM}^{\skaa}_{r+1},\tunagi\\
\wt{\calM}^{\scokaa}_{r+1}&\simto \wt{\calM}^{\scokaa}_{r}.
\end{align*}

\item[(ii)] There is a natural morphism 
\[\wt{M}^{\skaa}_{r}\to \wt{M}^{\scokaa}_{r},\]
which commute with $T_{s}$ and if $r$ is large enough, it is isomorphic.
\end{enumerate}
\end{cor}
\begin{proof}
For $r,l\in \ZZ_{\geq 1}$,
we consider a diagram
\[
\xymatrix{
0\ar[r]& \calM\boxtimes\calL^{H}_{r}[1](1)\ar[d]
\ar[r]& \calM\boxtimes\calL^{H}_{r+l}[1](1-l)\ar[d]
\ar[r]&  \calM\boxtimes\calL^{H}_{l}[1](1-l)\ar[d]
\ar[r]& 0\\
0\ar[r]& \calM\boxtimes\calL^{H}_{r}[1]
\ar[r]& \calM\boxtimes\calL^{H}_{r+l}[1](-l)
\ar[r]&  \calM\boxtimes\calL^{H}_{l}[1](-l)
\ar[r]& 0,\\
}
\]
where the vertical arrows are $N\boxtimes 1-1\boxtimes N$.
This induces an exact sequence
\begin{align}\notag
0\to \wt{\calM}^{\skaa}_{r}(r)\to\wt{\calM}^{\skaa}_{r+l}(r)\to &\wt{\calM}^{\skaa}_{l}
\\\to \wt{\calM}^{\scokaa}_{r}&\to \wt{\calM}^{\scokaa}_{r+l}(-l)\to \wt{\calM}^{\scokaa}_{l}(-l)\to 0. \label{sazan}
\end{align}
By Lemma~\ref{baloon}, this sequence implies that 
if $r\geq l_{0}$ the morphisms $\wt{\calM}^{\skaa}_{r}\to\wt{\calM}^{\skaa}_{r+1}$ and $\wt{\calM}^{\skaa}_{r+1}\to\wt{\calM}^{\skaa}_{r}$ are isomorphic.

Moreover, if $l=r\geq l_{0}$,
the first and last morphism in (\ref{sazan}) are isomorphic by the first assertion of this Lemma.
Therefore, the middle morphism $\wt{\calM}^{\skaa}_{r}\to \wt{\calM}^{\scokaa}_{r}$ is an isomorphism.

\end{proof}

Let $q\colon X\times \CS\to \CS$ be the second projection.
Consider the complex of mixed Hodge modules $q^{*}(\calS^{H}_{m,1})\in \mathrm{D}^{\mathrm{b}}{\mathrm{MHM}^{p}(X\times \CS)}$, whose underlying constructible sheaf is a local system $q^{-1}\calS^{\QQ}_{m,1}$.
Note that this is the shifted object by $\dim X+1$ of the mixed Hodge module associated to the pull back of the variation of mixed Hodge structure $(S_{m,1},F_{\bullet}S_{m,1},\calS_{m,1}^{\QQ})$ by $q$.
Moreover, remark that 
by the definition the non-trivial cohomologies of 
$\wt{\calM}^{\skaa}_{r}\otimes q^{*}\calS^{H}_{m,1}$ and 
$\wt{\calM}^{\scokaa}_{r}\otimes q^{*}\calS^{H}_{m,1}(\in \mathrm{D}^{\mathrm{b}}{\mathrm{MHM}^{p}(X\times \CS)})$ are only $0$-th ones, where $\otimes$ means the tensor product in the derived category of mixed Hodge modules (i.e. the pullback by the diagonal embedding $X\times \CS\to (X\times \CS)\times (X\times \CS)$ of the exterior product).
Therefore, we can regard $\wt{\calM}^{\skaa}_{r}\otimes q^{*}\calS^{H}_{m,1}$ and $\wt{\calM}^{\scokaa}_{r}\otimes q^{*}\calS^{H}_{m,1}$ as mixed Hodge modules on $X\times \CS$.

For $m\in \ZZ_{\geq 1}$,
we define mixed Hodge modules on $X\times \CS$ by
\begin{align*}
\wt{\calM}^{\skaa,\skaa}_{r,m}&:=\kaa(T_{s}\otimes 1-1\otimes T\colon 
\wt{\calM}^{\skaa}_{r}\otimes q^{*}\calS^{H}_{m,1}
\to 
\wt{\calM}^{\skaa}_{r}\otimes q^{*}\calS^{H}_{m,1}),\\
\wt{\calM}^{\skaa,\scokaa}_{r,m}&:=\cokaa(T_{s}\otimes 1-1\otimes T\colon 
\wt{\calM}^{\skaa}_{r}\otimes q^{*}\calS^{H}_{m,1}
\to 
\wt{\calM}^{\skaa}_{r}\otimes q^{*}\calS^{H}_{m,1}),\\
\wt{\calM}^{\scokaa,\skaa}_{r,m}&:=\kaa(T_{s}\otimes 1-1\otimes T\colon 
\wt{\calM}^{\scokaa}_{r}\otimes q^{*}\calS^{H}_{m,1}
\to 
\wt{\calM}^{\scokaa}_{r}\otimes q^{*}\calS^{H}_{m,1}),\tunagi\\
\wt{\calM}^{\scokaa,\scokaa}_{r,m}&:=\cokaa(T_{s}\otimes 1-1\otimes T\colon 
\wt{\calM}^{\scokaa}_{r}\otimes q^{*}\calS^{H}_{m,1}
\to 
\wt{\calM}^{\scokaa}_{r}\otimes q^{*}\calS^{H}_{m,1}).
\end{align*}

\begin{lem}\label{mirai}
\begin{enumerate}
\item[(i)]
$\wt{\calM}^{\skaa,\skaa}_{r,m}$, $\wt{\calM}^{\skaa,\scokaa}_{r,m}$, 
$\wt{\calM}^{\scokaa,\skaa}_{r,m}$ and $\wt{\calM}^{\scokaa,\scokaa}_{r,m}$
do not depend on the choice of $m$.
More precisely, for $a\in \ZZ_{\geq 1}$ and $m':=am$,
there is a natural isomorphism
\begin{align*}
\wt{\calM}^{\skaa,\skaa}_{r,m}&\simto \wt{\calM}^{\skaa,\skaa}_{r,m'},\\
\wt{\calM}^{\skaa,\scokaa}_{r,m}&\simto \wt{\calM}^{\skaa,\scokaa}_{r,m'},\\
\wt{\calM}^{\scokaa,\skaa}_{r,m}&\simto \wt{\calM}^{\scokaa,\skaa}_{r,m'},\tunagi\\
\wt{\calM}^{\scokaa,\scokaa}_{r,m}&\simto \wt{\calM}^{\scokaa,\scokaa}_{r,m'}.
\end{align*}

\item[(ii)]
There is a commutative diagram:
\begin{align}\label{irodori}
\xymatrix{
\wt{\calM}^{\skaa,\skaa}_{r,m}\ar[r]\ar[d]&\wt{\calM}^{\skaa,\scokaa}_{r,m}\ar[d]\\
\wt{\calM}^{\scokaa,\skaa}_{r,m}\ar[r]&\wt{\calM}^{\scokaa,\scokaa}_{r,m}.
}
\end{align}
If $r$ is large enough, all the morphisms are isomorphism.
\end{enumerate}
\end{lem}

\begin{proof}
The natural morphism defined in (vi) of Proposition~\ref{hoikoro}:
$\calS^{H}_{m,1}\to \calS^{H}_{m',1}$ induces a morphism $\wt{\calM}^{\skaa,\skaa}_{r,m}\simto \wt{\calM}^{\skaa,\skaa}_{r,m'}$.
Its underlying morphism between $\CC$-perverse sheaves is
the identity 
\begin{align*}
\bigoplus_{\lambda^m=1}K^{\lambda}_{\CC}\boxtimes (\calL^{\CC}_{r}\otimes \calS_{1,\lambda}^{\CC})[1]
\to 
\bigoplus_{\lambda^m=1}K^{\lambda}_{\CC}\boxtimes (\calL^{\CC}_{r}\otimes \calS_{1,\lambda}^{\CC})[1].
\end{align*}
This implies the assertion (i) for $\wt{\calM}^{\skaa,\skaa}_{r,m}$.
In the same way, we can show (i) for $\wt{\calM}^{\skaa,\scokaa}_{r,m}$, $\wt{\calM}^{\scokaa,\skaa}_{r,m}$ and $\wt{\calM}^{\scokaa,\scokaa}_{r,m}$.

By (ii) of Corollary~\ref{gift}, we have natural morphisms
\begin{align*}
\wt{\calM}^{\skaa,\skaa}_{r,m}&\to \wt{\calM}^{\scokaa,\skaa}_{r,m},\tunagi\\
\wt{\calM}^{\skaa,\scokaa}_{r,m}&\to \wt{\calM}^{\scokaa,\scokaa}_{r,m},
\end{align*}
(the vertical arrows in (\ref{irodori})) which are isomorphic for sufficiently large $r$.
We define the morphism $\wt{\calM}^{\skaa,\skaa}_{r,m}\to \wt{\calM}^{\skaa,\scokaa}_{r,m}$ as the composition of 
the natural monomorphism $\wt{\calM}^{\skaa,\skaa}_{r,m} \hookrightarrow \wt{\calM}^{\skaa}_{r}\otimes q^{*}\calS^{H}_{m,1}$
and the natural epimorphism $\wt{\calM}^{\skaa}_{r}\otimes q^{*}\calS^{H}_{m,1}\twoheadrightarrow \wt{\calM}^{\skaa,\scokaa}_{r,m}$.
Its underlying morphism between $\CC$-perverse sheaves is
the identity
\[
\bigoplus_{\lambda^m=1}K_{\CC,\lambda}\boxtimes (\calL^{\CC}_{r}\otimes \calS_{1,\lambda}^{\CC})[1]
\to 
\bigoplus_{\lambda^m=1}K_{\CC,\lambda}\boxtimes (\calL^{\CC}_{r}\otimes \calS_{1,\lambda}^{\CC})[1].
\]
Hence, $\wt{\calM}^{\skaa,\skaa}_{r,m}\to \wt{\calM}^{\skaa,\scokaa}_{r,m}$ is an isomorphism.
In the same way, we get
a natural morphism $\wt{\calM}^{\scokaa,\skaa}_{r,m}\to \wt{\calM}^{\scokaa,\scokaa}_{r,m}$ and we get the desired diagram (\ref{irodori}).
\end{proof}

Corollary~\ref{gift} and Lemma~\ref{mirai} imply the following.

\begin{cor}
The mixed Hodge modules $\wt{\calM}^{\skaa,\skaa}_{r,m}$, $\wt{\calM}^{\skaa,\scokaa}_{r,m}$, 
$\wt{\calM}^{\scokaa,\skaa}_{r,m}$ and $\wt{\calM}^{\scokaa,\scokaa}_{r,m}$
do not depend on the choice of sufficiently large $r$ and $m$
and all of them are naturally isomorphic to each other.
\end{cor}

We define $\wt{\calM}:=\wt{\calM}^{\scokaa,\scokaa}_{r,m}$ for sufficiently large $r,m\in \ZZ_{\geq 1}$ and
set $\wt{M}:=\wt{M}^{\scokaa,\scokaa}_{r,m}$, $F_{\bullet}\wt{M}:=F_{\bullet}\wt{M}^{\scokaa,\scokaa}_{r,m}$, $\wt{K}:=\wt{K}^{\scokaa,\scokaa}_{r,m}$ and
$W_{\bullet}\wt{K}:=W_{\bullet}\wt{K}^{\scokaa,\scokaa}_{r,m}$.

\begin{lem}\label{naminoutaha}
We have
\[\psi_{t}(\wt{\calM})\simeq \calM.\]
Moreover, 
under this isomorphism,
the semisimple (resp. $\frac{-1}{2\pi\sqrt{-1}}$ times the logarithm of the unipotent) part of the monodromy automorphism
is $T_{s}$ (resp. $N$).
\end{lem}
\begin{proof}
By the definition of $\wt{\calM}$, we have the following diagram:
\[
\xymatrix{
\calM\boxtimes \psi_{t}(\calL^{H}_{r}\otimes \calS^{H}_{m,1}[1](1))\ar[d]^{v_{1}}\ar[r]^{h_{1}}&
\calM\boxtimes \psi_{t}(\calL^{H}_{r}\otimes \calS^{H}_{m,1}[1])\ar[d]^{v_{2}}\ar[r]^{h_{2}}&
\psi_{t}(\wt{\calM}^{\scokaa}_{r}\otimes q^{*} \calS^{H}_{m,1})\ar[d]^{v_{3}}\ar[r]&0\\
\calM\boxtimes \psi_{t}(\calL^{H}_{r}\otimes \calS^{H}_{m,1}[1](1))\ar[r]^{h_{3}}&
\calM\boxtimes \psi_{t}(\calL^{H}_{r}\otimes \calS^{H}_{m,1}[1])\ar[r]^{h_{4}}&
\psi_{t}(\wt{\calM}^{\scokaa}_{r}\otimes q^{*} \calS^{H}_{m,1})\ar[d]^{v_{4}}\ar[r]&0\\
&&\psi_{t}(\wt{\calM})\ar[d]&\\
&&0&
},
\]
where 
$h_{1}$, $h_{3}$ are $N\boxtimes 1-1\boxtimes \psi_{t}(N)$,
$v_{1}$, $v_{2}$ are $T\boxtimes 1-1\boxtimes T_{s}$,
$v_{3}$ is the morphism induced by $v_{2}$ and
$h_{2}$, $h_{4}$ and $v_{4}$ are the quotient morphism.
By (v) of Proposition~\ref{hoikoro},
$\psi_{t}(\calL^{H}_{r}\otimes \calS^{H}_{m,1}[1])=V_{m,1}^{\semi,H}\otimes V_{r}^{\nilp, H}=(V_{m,1}^{\semi,\CC}\otimes V_{r}^{\nilp,\CC}, F_{\bullet}(V_{m,1}^{\semi,\CC}\otimes V_{r}^{\nilp,\CC}), V_{m,1}^{\semi,\QQ}\otimes V_{r}^{\nilp,\QQ}, V_{m,1}^{\semi,\QQ}\otimes W_{\bullet}V_{r}^{\nilp,\QQ})$.
Therefore,
the cokernel of $h_{1}(=h_{3})$ is isomorphic to $M\boxtimes V_{m,1}^{\semi,H}$.
Therefore, $v_{3}$ is the morphism
\[T\boxtimes T_{s}\colon \calM\boxtimes V_{m,1}^{\semi,H}\to \calM\boxtimes V_{m,1}^{\semi,H}.\]
The cokernel of it is isomorphic to $\calM$.
This completes the proof.
\end{proof}

Let us describe the mixed Hodge module $\wt{\calM}$ more concretely.
Recall that $M$ and its Hodge filtration have the decomposition (\ref{takarajima}) and (\ref{cruise}). 
We define an $O$-module $\wt{M}'$ on $X\times \CS$ as
\[\wt{M}':=\bigoplus_{\alpha\in (-1,0]\cap\QQ}M_{\alpha}\otimes \CC[t^{\pm}].\]
We can endow it with $D$-module structure so that
$t\partial_{t}-\alpha$ acts as $N$ on $M_{\alpha}$.
Then, for any $\alpha\in (-1,0]$, we have
\[(\wt{M}')^{\alpha}=M_{\alpha},\]
where $(\wt{M}')^{\alpha}=\bigcup_{l\geq 0}\mathrm{Ker}((t\partial_{t}-\alpha)^l\colon \wt{M}'\to \wt{M}')$.
We define the Hodge (resp. weight) filtration of $\wt{M}'$ as
\begin{align*}
F_{\bullet}\wt{M}'&:=\bigoplus_{\alpha\in (-1,0]\cap\QQ}F_{\bullet}M_{\alpha}\otimes \CC[t^{\pm}]\\
(\mbox{resp.\ }
W_{\bullet}\wt{M}'&:=\bigoplus_{\alpha\in (-1,0]\cap\QQ}W_{\bullet-1}M_{\alpha}\otimes \CC[t^{\pm}]
).
\end{align*}
Then, we can easily see the following.
\begin{lem}\label{nicol}
We have a natural isomorphism (between $D$-modules)
\[\wt{M}\simto {\wt{M}}'.\]
Moreover, this morphism induces isomorphisms
\begin{align*}
F_{p}\wt{M}&\simto F_{p}\wt{M}',\tunagi\\
W_{k}\wt{M}&\simto W_{k}\wt{M}',
\end{align*}
for $p,k\in \ZZ$.
\end{lem}

As a corollary, we get the following.

\begin{cor}\label{iwanna}
$\wt{\calM}$ is monodromic mixed Hodge module on $X\times \CS$.
\end{cor}

Therefore, we can define the functor $G'\colon \mathscr{G}'(X)\to \mathrm{MHM}^{p}_{\mathrm{mon}}(X\times \CS)$ so that $G'(\calM)=\wt{\calM}$.
Lemma~\ref{naminoutaha} implies the following.

\begin{cor}\label{light}
$F'\circ G'$ and $1$ is naturally isomorphic.
\end{cor}

Finally, we will see that $G'\circ F'$ and $1$ are naturally isomorphic and complete the proof of Lemma~\ref{daijyoubuda}.
Let $\calM=(M,F_{\bullet}M,K,W_{\bullet}K)$ be a graded polarizable mixed Hodge module on $X\times \CS$ (not $X$)
whose underlying perverse sheaf is a local system shifted by $d+1$,
i.e. there is a local system $\calP$ on $X\times \CS$ such that $K=\calP[d+1]$.
We consider a triple $(\psi_{t}\calM,T_{s},N)$,
where $T_{s}$ (resp. $N$) is the semisimple part ($\frac{-1}{2\pi\sqrt{-1}}$ times of the logarithm of the unipotent part) of the monodromy automorphism $T$ of the nearby cycle.
Then, the triple satisfies the conditions ($\star$-i) and ($\star$-ii) in the
first part of this section.
Therefore, we have a mixed Hodge module $\wt{\psi_{t}\calM}$ on $X\times \CS$.
Recall that $\pi\colon \CC\to \CS$ ($z\mapsto \exp(2\pi\sqrt{-1}z)$) is the universal covering of $\CS$.
Let $\pi_{1}\colon X\times \CC\to X$ be the first projection.
Then, $(\pi_{1})_{*}(1\times \pi)^{-1}\calP$ is a local system on $X$ and
we have
\[{}^p\psi_{t}K=(\pi_{1})_{*}(1\times \pi)^{-1}\calP[d].\]
Then, one can see the following two lemmas by the definition of $\wt{\calM}$.

\begin{lem}\label{gideon}
In this situation,
there is a natural isomorphism
\[\wt{{}^p\psi_{t}K}\simeq K.\]
Moreover, this induces a natural isomorphism
\[\wt{\psi_{t}\calM}\simeq \calM,\]
between mixed Hodge modules.
\end{lem}

\begin{lem}\label{bolas}
Let $(\calM_{(-1,0]},T_{s},N)$ be a triple satisfying the conditions ($\star$-i), ($\star$-ii).
For a (local) regular function $g$ on $X$,
we have natural isomorphism
\begin{align*}
\psi_{g}(\wt{\calM_{(-1,0]}})&\simeq \wt{(\psi_{g}\calM_{(-1,0]})},\\
\phi_{g,1}(\wt{\calM_{(-1,0]}})&\simeq \wt{(\phi_{g,1}\calM_{(-1,0]})},\tunagi\\
\Xi_{g}(\wt{\calM_{(-1,0]}})&\simeq \wt{(\Xi_{g}\calM_{(-1,0]})},
\end{align*}
where we regard $g$ as a function on $X\times \CS$ in the left hand side and
$\Xi_{g}$ is Beilinson's maximal extension functor.
\end{lem}

\begin{proof}[Proof of Lemma~\ref{daijyoubuda}]
Let $\calM$ be a monodromic mixed Hodge module on $X\times \CS$ (not necessary smooth) and 
consider the tuple $(\psi_{t}\calM, T_{s},N)$.
To show the isomorphism $\wt{\psi_{t}\calM}\simeq \calM$,
we use induction.
Let $d_{0}(\geq 0)$ be the dimension of the support of $\psi_{t}\calM$. 
Assume that we have a natural isomorphism $\wt{\psi_{t}\calN}\simeq \calN$ for any monodromic mixed Hodge module on $X\times \CS$ of which the dimension of the support of the nearby cycle $\psi_{t}\calN$ is less than $d_{0}$.
Locally, there is a regular function $g$ on $X$ such that $\calM|_{X\setminus g^{-1}(0)\times \CS}$ is a (non-zero) smooth mixed Hodge module, i.e. the underlying perverse sheaf is (the zero extension of) a shifted local system.
In this case, $\calM$ is reconstructed from the gluing data: $\calM|_{X\setminus g^{-1}(0)\times \CS}$, $\phi_{g,1}(\calM)$, the can morphism $\psi_{g,1}(\calM)\to \phi_{g,1}(\calM)$ and the var morphism $\phi_{g,1}(\calM)\to \psi_{g,1}(\calM)(-1)$.
More precisely,
by Beilinson's gluing formula for mixed Hodge modules (see \cite{Beilinson} for perverse sheaves), $\calM$ is isomorphic to the $0$-th cohomology of the complex 
\begin{align}\label{syuriken}
    \psi_{g,1}\calM\to \Xi_{g}\calM\oplus \phi_{g,1}\calM\to \psi_{g,1}\calM(-1),
\end{align}
where $\Xi_{g}\calM\oplus \phi_{g,1}\calM$ is in degree $0$.
By Lemma~\ref{gideon},
we have
\begin{align}\label{kikyou}{(\psi_{t}(\calM|_{X\setminus g^{-1}(0)\times \CS}))}^{\sim}\simeq \calM|_{X\setminus g^{-1}(0)\times \CS},
\end{align}
where we wrote $(\calN)^{\sim}$ instead of $\wt{\calN}$ for $\calN=\psi_{t}(\calM|_{X\setminus g^{-1}(0)\times \CS})$.
Since $\psi_{g,1}\calM$ and $\Xi_{g}\calM$ are determined by $\calM|_{X\setminus  g^{-1}(0)\times \CS}$,
the isomorphism (\ref{kikyou}) induces isomorphisms:
\begin{align}\label{masara2}
    \psi_{g,1}\calM\simeq& \psi_{g,1}(\wt{\psi_{t}\calM}),\quad \mbox{and}\\\label{siba2}
\Xi_{g}\calM\simeq& \Xi_{g}(\wt{\psi_{t}\calM}).
\end{align}
Moreover, by the inductive assumption, we have
\begin{align*}
    \phi_{g,1}\calM\simeq \wt{\psi_{t}\phi_{g,1}\calM} (\simeq \wt{\phi_{g,1}\psi_{t}\calM}).
\end{align*}
Therefore, by Lemma~\ref{bolas}
we have
\begin{align}\label{katana2}
    \phi_{g,1}\calM\simeq \phi_{g,1}\wt{\psi_{t}\calM}.
\end{align}
Substituting (\ref{masara2}), (\ref{siba2}) and (\ref{katana2}) into (\ref{syuriken}),
we obtain a complex
\[\psi_{g,1}\wt{\psi_{t}\calM}\to \Xi_{g}\wt{\psi_{t}\calM}\oplus \phi_{g,1}\wt{\psi_{t}\calM}\to \psi_{g,1}\wt{\psi_{t}\calM}(-1).\]
By Beilinson's gluing formula again, the $0$-th cohomology of this complex is isomorphic to $\wt{\psi_{t}\calM}$.
In this way, 
we obtain a natural isomorphism:
\begin{align}\label{jacekai}
\calM\simeq \wt{\psi_{t}\calM}.
\end{align}
This implies that $G'\circ F'$ and $1$ are naturally isomorphic.
Together with Lemma~\ref{light}, the proof is now complete. 
\end{proof}

\subsection{The proof of Theorem~\ref{main2kai}}\label{syoumei}

In this subsection,
we construct the functor $G$ in Theorem~\ref{main2kai} and prove the theorem.
Let $((\calM_{(-1,0]},T_{s},N),\calM_{-1},c,v)$ be an object in $\mathscr{G}(X)$.
Since $(\calM_{(-1,0]},T_{s},N)$ is an object in $\mathscr{G}'(X)$,
we have a monodromic mixed Hodge module $\wt{\calM_{(-1,0]}}$ on $X\times \CS$ defined in the previous subsection.
We set $\calM_{0}:=\ker(T_{s}-1)$, which is a mixed Hodge module on $X$.
Then, we have
\begin{align*}
\psi_{t,1}(\wt{\calM_{(-1,0]}})\simeq \calM_{0}.
\end{align*}

\begin{lem}\label{sidaruta}
Let $((\calM_{(-1,0]},T_{s},N),\calM_{-1},c,v)$ be an object in $\mathscr{G}(X)$.
Then, there exists an unique mixed Hodge module $\vardbtilde{\calM_{(-1,0]}}$ on
$X\times \CC$
with the following properties:
\begin{enumerate}
\item[(i)] 
we have 
\[\vardbtilde{\calM_{(-1,0]}}|_{X\times \CS}=\wt{\calM_{(-1,0]}}.\]
In particular, $\psi_{t,1}(\vardbtilde{\calM_{(-1,0]}})\simeq \psi_{t,1}(\wt{\calM_{(-1,0]}})\simeq \calM_{0}$.
\item[(ii)] 
we have 
\[\phi_{t,1}(\vardbtilde{\calM_{(-1,0]}})\simeq \calM_{-1}.\]
\item[(iii)] 
Under the above identifications, the can (resp. var) morphism $\mathrm{can}\colon \psi_{t,1}(\vardbtilde{\calM_{(-1,0]}})\to \phi_{t,1}(\vardbtilde{\calM_{(-1,0]}})$ 
(resp. $\mathrm{var}\colon \phi_{t,1}(\vardbtilde{\calM_{(-1,0]}})\to \psi_{t,1}(\vardbtilde{\calM_{(-1,0]}})(-1)$)
is isomorphic to $c\colon \calM_{0}\to \calM_{-1}$ (resp. $v\colon \calM_{-1}\to \calM_{0}(-1)$).
\end{enumerate}
\end{lem}

\begin{proof}
Let $\wt{\calL_{r}^{H}[1]}\in \mathrm{MHM}^{p}(\CC)$ be an extension (e.g. the minimal extension) of a mixed Hodge module $\calL_{r}^{H}[1]\in \mathrm{MHM}^{p}(\CS)$ on $\CS$ to $\CC$.
Then, the mixed Hodge module on $X\times \CC$
\[\cokaa(N\boxtimes 1-1\boxtimes N\colon 
\calM_{(-1,0]}\boxtimes\wt{\calL_{r}^H[1]}(1)\to \calM_{(-1,0]}\boxtimes \wt{\calL_{r}^{H}[1]})
\]
is an extension of $\wt{(\calM_{(-1,0]})}^{\scokaa}_{r}\in \mathrm{MHM}^{p}(X\times \CS)$.
In the same way,
we can see that 
there is a mixed Hodge module on $X\times \CC$ whose restriction to $X\times \CS$ is $\wt{(\calM_{(-1,0]})}^{\scokaa,\scokaa}_{r,m}(=\wt{\calM_{(-1,0]}})$.
We can say that $\wt{\calM_{(-1,0]}}$ is an extendable mixed Hodge module on $X\times \CS$.
Therefore, 
by ``Beilinson's gluing" for mixed Hodge modules,
we can construct a mixed Hodge module $\vardbtilde{{\calM_{(-1,0]}}}$ from the gluing data
$(\wt{\calM_{(-1,0]}}, \calM_{-1},c,v)$ which has the desired properties.
\end{proof}

\begin{proof}[Proof of Theorem~\ref{main2kai}]
We define the functor $G\colon \mathscr{G}(X)\to \mathrm{MHM}^{p}_{\mathrm{mon}}(X\times \CC)$ so that 
\[G(((\calM_{(-1,0]},T_{s},N),\calM_{-1},c,v))=\vardbtilde{{M_{(-1,0]}}}.\]

For $((\calM_{(-1,0]},T_{s},N),\calM_{-1},c,v)\in \mathscr{G}(X)$,
by Lemma~\ref{daijyoubuda} and the definition of $G$,
$F\circ G(((\calM_{(-1,0]},T_{s},N),\calM_{-1},c,v))$ is isomorphic to $((\calM_{(-1,0]},T_{s},N),\calM_{-1},c,v)$.
Hence, we have $F\circ G$ and $1$ are isomorphic.

For $\calM\in \mathrm{MHM}^{p}_{\mathrm{mon}}(X\times \CC)$,
we have $F(\calM)=((\psi_{t}\calM,T_{s},N),\phi_{t,1}\calM,\mathrm{can},\mathrm{var})$.
Note that we have 
\begin{align}\label{conan}
\psi_{t}(\calM|_{X\times \CS})=\psi_{t}\calM.
\end{align}
Then, we have
\begin{align*}
\vardbtilde{\psi_{t}\calM}|_{X\times \CS}
&=\wt{\psi_{t}\calM}\qquad (\mbox{by the definition})\\
&=\wt{(\psi_{t}(\calM|_{X\times \CS}))}\qquad(\mbox{by (\ref{conan})})\\
&=\calM|_{X\times \CS} \qquad (\mbox{by Lemma~\ref{daijyoubuda} or (\ref{jacekai})}).
\end{align*}
Moreover, by the definition of $\vardbtilde{\psi_{t}\calM}$,
we have $\psi_{t,1}(\vardbtilde{\psi_{t}\calM})=\psi_{t,1}\calM$ and
$\phi_{t,1}(\vardbtilde{\psi_{t}\calM})=\phi_{t,1}\calM$.
Under these identifications,
the can morphism 
$\psi_{t,1}(\vardbtilde{\psi_{t}\calM})\to \phi_{t,1}(\vardbtilde{\psi_{t}\calM})$ is
$\mathrm{can}\colon \psi_{t,1}\calM\to \phi_{t,1}\calM$ and
the var morphism $\phi_{t,1}(\vardbtilde{\psi_{t}\calM})\to \psi_{t,1}(\vardbtilde{\psi_{t}\calM})(-1)$ is 
$\mathrm{var}\colon  \phi_{t,1}\calM\to \psi_{t,1}\calM(-1)$.
Therefore, $\calM$ also satisfies the gluing condition (i)-(iii) in Lemma~\ref{sidaruta} for the object $F(\calM)\in \mathscr{G}(X)$.
Hence, $\calM$ is isomorphic to $\vardbtilde{\psi_{t}\calM}$.
This implies that we have $G\circ F$ and $1$ are isomorphic.

This completes the proof
\end{proof}

Finally, 
for $((\calM_{(-1,0]},T_{s},N),\calM_{-1},c,v)\in \mathscr{G}(X)$,
we describe the mixed Hodge module $\vardbtilde{{\calM_{(-1,0]}}}$ more concretely.
Let $\vardbtilde{{M_{(-1,0]}}}$ (resp. $\vardbtilde{{K_{(-1,0]}}}$) be the underlying $D$-module (perverse sheaf) of $\vardbtilde{{\calM_{(-1,0]}}}$.
Since $\wt{\calM_{(-1,0]}}$ is monodromic on $X\times \CS$, $\vardbtilde{{\calM_{(-1,0]}}}$ is also monodromic on $X\times \CC$.
Recall that $(\vardbtilde{{M_{(-1,0]}}})^{\alpha}=\bigcup_{l\geq 0}\mathrm{Ker}((t\partial_{t}-\alpha)^l\colon \vardbtilde{{M_{(-1,0]}}}\to \vardbtilde{{M_{(-1,0]}}})$.

\begin{prop}
For $\alpha\in [-1,0]$,
we have
\[(\vardbtilde{{M_{(-1,0]}}})^{\alpha}
\simeq \left\{
\begin{array}{ll}
M_{\alpha}&(\alpha\in (-1,0])\\
M_{-1}&(\alpha=-1).
\end{array}
\right.
\]
Moreover, 
under these identifications,
$t\partial_{t}-\alpha$ on $(\vardbtilde{{M_{(-1,0]}}})^{\alpha}$ is $N\colon M\to M$,
$-\partial_{t} \colon (\vardbtilde{{M_{(-1,0]}}})^{0}\to (\vardbtilde{{M_{(-1,0]}}})^{-1}$ is $c\colon M_{0}\to M_{-1}$ and
$t\colon (\vardbtilde{{M_{(-1,0]}}})^{-1}\to (\vardbtilde{{M_{(-1,0]}}})^{0}$ is $v\colon M_{-1}\to M_{0}$.
In particular,
$\vardbtilde{{M_{(-1,0]}}}$ is isomorphic to 
\[(\bigoplus_{\alpha\in [-1,0)\cap\QQ}M_{\alpha}\otimes \CC[\partial_{t}]\partial_{t})\oplus M_{-1}\oplus (\bigoplus_{\alpha\in (-1,0]\cap \QQ}M_{\alpha}\otimes \CC[t]),\]
where the $D$-module structure of it is defined so that for any $\alpha\in (-1,0]\cap \QQ$, $t\partial_{t}-\alpha$ acts on $M_{\alpha}$ as $N\colon M_{(-1,0]}\to M_{(-1,0]}$, 
$-\partial_{t}\colon M_{0}\to M_{-1}$ is $c\colon M_{0}\to M_{-1}$ and
$t\colon M_{-1}\to M_{0}$ is $v\colon M_{-1}\to M_{0}$.
Moreover, the Hodge filtration $F_{\bullet}\vardbtilde{{M_{(-1,0]}}}$ (resp. the weight filtration $W_{\bullet}\vardbtilde{{M_{(-1,0]}}}$) is isomorphic to
\begin{align*}
(\bigoplus_{i\geq 1}\bigoplus_{\alpha\in [-1,0)\cap \QQ}\partial_{t}^{i}F_{\bullet-i}M_{\alpha})\oplus F_{\bullet}M_{-1}\oplus (\bigoplus_{\alpha\in (-1,0]\cap \QQ}F_{\bullet}M_{\alpha}\otimes \CC[t])\\
(\mbox{resp.\ }
(\bigoplus_{\alpha\in [-1,0)\cap \QQ}W_{\bullet}M_{\alpha}\otimes \CC[\partial_{t}]\partial_{t})\oplus W_{\bullet}M_{-1}\oplus (\bigoplus_{\alpha\in (-1,0]\cap \QQ}W_{\bullet}M_{\alpha}\otimes \CC[t])).
\end{align*} 
\end{prop}
\begin{proof}
For $\alpha\in (-1,0]\cap \QQ$,
we have
\begin{align*}
(\vardbtilde{{M_{(-1,0]}}})^{\alpha}&\simeq \psi_{t,e(\alpha)}(\vardbtilde{{M_{(-1,0]}}})\quad (\mbox{by Corollary~\ref{nevancor}})\\
&\simeq \psi_{t,e(\alpha)}({\wt{M_{(-1,0]}}})\\
&\simeq M_{\alpha}\quad (\mbox{by Lemma~\ref{naminoutaha}}).
\end{align*}
Moreover,
$t\partial_{t}-\alpha$ on $(\vardbtilde{{M_{(-1,0]}}})^{\alpha}$ is $N$ under this identification.
Similarly to this,
we have
\begin{align*}
(\vardbtilde{{M_{(-1,0]}}})^{-1}&\simeq \phi_{t,1}(\vardbtilde{{M_{(-1,0]}}})\quad (\mbox{by Corollary~\ref{nevancor}})\\
&\simeq M_{-1}\quad (\mbox{by the definition of $\vardbtilde{{M_{(-1,0]}}}$}).
\end{align*}
By Corollary~\ref{nevancor} again, 
$-\partial_{t}\colon (\vardbtilde{{M_{(-1,0]}}})^{0}\to (\vardbtilde{{M_{(-1,0]}}})^{-1}$ is 
$\mathrm{can}\colon \psi_{t,1}(\vardbtilde{{M_{(-1,0]}}})\to \phi_{t,1}(\vardbtilde{{M_{(-1,0]}}})$
and 
$t\colon (\vardbtilde{{M_{(-1,0]}}})^{-1}\to (\vardbtilde{{M_{(-1,0]}}})^{0}$ is 
$\mathrm{var}\colon \phi_{t,1}(\vardbtilde{{M_{(-1,0]}}})\to \psi_{t,1}(\vardbtilde{{M_{(-1,0]}}})$.
By the definition of $\vardbtilde{{M_{(-1,0]}}}$, these are $c$ and $v$ respectively.
\end{proof}

\subsection{The Fourier-Laplace transformations of mixed Hodge modules}
\label{bell}

As an application of Theorem~\ref{main2kai},
we endow the Fourier-Laplace transformation of the underlying $D$-module of a monodromic mixed Hodge module with a mixed Hodge module structure.

Remark again that we identify a $D$-module $A$ on a smooth affine algebraic variety $Z$ with the $\Gamma(Z;D_{Z})$-module $\Gamma(Z;A)$.  
Let $M$ be a $D$-module on $X\times \CC_{t}$.
Let $\CC_{\tau}$ be the dual vector space of $\CC_{t}$ with a coordinate $\tau$ such that the pairing $\langle t,\tau\rangle$ is given by $t\tau$.
We consider an isomorphism $\Gamma(\CC_{t};D_{\CC_{t}})\simto \Gamma(\CC_{\tau};D_{\CC_{\tau}})$ given by $\tau\mapsto -\partial_{t}$ and $\partial_{\tau}\mapsto t$.
The Fourier-Laplace transformation of $M$ is equal to $M$ as $D_{X}$-module and the $\CC[\tau]\langle\partial_{\tau}\rangle$-structure is induced through the isomorphism $\Gamma(\CC_{t};D_{\CC_{t}})\simto \Gamma(\CC_{\tau};D_{\CC_{\tau}})$.
We denote it by $\FOU{M}$ and regard it as a $D$-module on $X\times \CC_{\tau}$.

Assume that $M$ is monodromic.
For simplicity, we assume that $X$ is affine in this subsection.
For a section $m\in \FOU{M}$,
we have
\begin{align*}
\tau\partial_{\tau}m&=-\partial_{t}tm\\
&=(-t\partial_{t}-1)m.
\end{align*}
Therefore, we have (as $\CC_{X}$-modules)
\[(\FOU{M})^{\beta}=M^{-\beta-1}.\]
In particular, $\FOU{M}$ is also monodromic.
Therefore, we have
\begin{align*}
\FOU{M}(=\bigoplus_{\beta\in \QQ}(\FOU{M})^{\beta})&=\bigoplus_{\beta\in \QQ}M^{-\beta-1},\\
\psi_{\tau}(\FOU{M})&= \bigoplus_{\alpha\in [-1,0)\cap \QQ}M^{\alpha},\\
\psi_{\tau,\alpha}(\FOU{M})&=M^{-\alpha-1}\quad (\alpha\in (-1,0]\cap \QQ),\tunagi\\
\phi_{\tau,1}(\FOU{M})&=M^{0}.
\end{align*}

Let $\calM=(M,F_{\bullet}M,K,W_{\bullet}K)$ be a monodromic mixed Hodge module on $X\times \CC_{t}$.
Recall that by Theorem~\ref{main1} we have
\[F_{\bullet}M=\bigoplus_{\beta\in \QQ}F_{\bullet}M^{\beta}.\]
Note that we have a decomposition 
\[\psi_{t}\calM=\psi_{t,1}\calM\oplus \psi_{t,\neq 1}\calM,\]
where $\psi_{t,\neq 1}\calM$ is the mixed Hodge module whose underlying $D$-module is $\bigoplus_{\alpha\in (-1,0)}M^{\alpha}$.
Then,
the underlying $D$-module of the mixed Hodge module $\phi_{t,1}M\oplus \psi_{t\neq 1}\calM$ (resp. $\psi_{t,1}\calM(-1)$) is the $D$-module $\psi_{\tau}(\FOU{M})=\bigoplus_{\alpha\in [-1,0)}M^{\alpha}$ (resp. $\phi_{\tau}(\FOU{M})=M^{0}$).
Consider a tuple 
\[((\phi_{t,1}\calM\oplus \psi_{t,\neq 1}\calM, 1\oplus T_{s}^{-1}, (\mathrm{can}\circ \mathrm{var})\oplus N), \psi_{t,1}M(-1),-\mathrm{var},\mathrm{can}).\]
One can see that this is in $\mathscr{G}(X)$.
By Theorem~\ref{main2kai} (or the functor $G\colon \mathscr{G}(X)\to \mathrm{MHM}^{p}_{\mathrm{mon}}(X\times \CC)$), we obtain a mixed Hodge module $\FOU{\calM}$ on $X\times \CC_{\tau}$, whose underlying $D$-module is $\FOU{M}$.
Since $\FOU{\calM}$ is monodromic,
we have a decomposition
\begin{align}\label{yoursing}
F_{\bullet}\FOU{M}=\bigoplus_{\beta\in \QQ}F_{\bullet}\FOU{M}^{\beta}.
\end{align}

\begin{rem}
There are other possible definitions of the mixed Hodge modules whose underlying $D$-modules are $\FOU{M}$.
For example,
the underlying $D$-module of the mixed Hodge module corresponding to a tuple
$((\phi_{t,1}(1)\calM\oplus \psi_{t,\neq 1}\calM, 1\oplus T_{s}^{-1}, (\mathrm{can}\circ \mathrm{var})\oplus N), \psi_{t,1}M,-\mathrm{var},\mathrm{can})$ is also $\FOU{M}$.
In this sense, the above definition of $\FOU{\calM}$ is just a candidate for the definition of ``the Fourier-Laplace transformation of a mixed Hodge module''.
\end{rem}

\begin{prop}\label{gegege}
We have
\begin{align*}
F_{p}{\FOU{M}^{\alpha}}=
\left\{
\begin{array}{l}
F_{p+1}M^{-1}\ (\alpha=0)\\
F_{p}M^{-\alpha-1}\ (\alpha\in [-1,0)).
\end{array}
\right.
\end{align*}
\end{prop}
\begin{proof}
By the definition of the Hodge filtration of $\psi_{t}M$ and $\phi_{t,1}M$, we have
\begin{align*}
F_{p}(\FOU{M})^{-1}=&F_{p-1}\phi_{\tau,1}(\FOU{M})\\
=&F_{p-1}\psi_{t,1}M(-1)\\
=&F_{p}M^{0},\\
F_{p}(\FOU{M})^{\alpha}=&F_{p}\psi_{\tau,\alpha}(\FOU{M})\\
=&F_{p}\psi_{t,-\alpha-1}M\\
=&F_{p}M^{-\alpha-1},\tunagi\\
F_{p}(\FOU{M})^{0}=&F_{p}\psi_{\tau,1}(\FOU{M})\\
=&F_{p}\phi_{t,1}(M)\\
=&F_{p+1}M^{-1},
\end{align*}
where $\alpha\in (-1,0)$.
\end{proof}

By this proposition, the decomposition~(\ref{yoursing}) and the strict specializability,
we can describe $F_{p}\FOU{M}$ concretely.

\footnotesize
\bibliographystyle{alpha}
\bibliography{reference}
\Addresses
\end{document}